\documentclass[12pt]{article}

\usepackage{amsmath}
\usepackage{amssymb}
\usepackage{mathrsfs}
\usepackage{amscd}
\usepackage{dsfont}
\usepackage{bm}
\usepackage{graphics}

\usepackage{bbm}
\usepackage{latexsym}
\usepackage[all]{xy}
\usepackage{makeidx}
\usepackage{tocloft}
\usepackage{fancyhdr}
\usepackage{ifthen}

\newcommand{\lsection}[2][""]{%
    \ifthenelse{\equal{#1}{""}}{%
        \section{#2}
    }{%
        \renewcommand{\sectionmark}[1]{\markright{\thesection.\ \MakeUppercase{#1}}}
        \section{#2}
        \renewcommand{\sectionmark}[1]{\markright{\thesection.\ \MakeUppercase{##1}}}
    }
}

\newcommand{\lchapter}[2][""]{%
    \ifthenelse{\equal{#1}{""}}{%
        \chapter{#2}
    }{%
        \renewcommand{\chaptermark}[1]{\markboth{\MakeUppercase{\chaptername\ \thechapter.\ #1}}{}}
        \chapter{#2}
        \renewcommand{\chaptermark}[1]{\markboth{\MakeUppercase{\chaptername\ \thechapter.\ ##1}}{}}
    }
}
\def\Z{\mathbb Z}

\def\R{\mathbb R}
\def\C{\mathbb C}

\def\iff{if and only if}
\def\mfd{manifold}
\def\fcn{function}
\def\str{structure}

\def\h{holomorphic}

\def\r{respectively}
\def\st{such that}
\def\(#1_#2){(#1_1,#1_2,\dots,#1_#2)}
\def\p #1_#2{#1_1#1_2\dots#1_#2}
\def\s#1_#2{#1_1+#1_2+\dots+#1_#2}
\def\wrt{with respect to}

\def\iso{isomorphism}
\def\ra{\rightarrow}
\def\lra{\longrightarrow}

\def\hra{\hookrightarrow}
\def\op{\operatorname}

\def\vb{vector bundle}

\def\bp{\bar\partial}

\def\ssm{\hspace{-.5mm}\smallsetminus\hspace{-.5mm}}

\def\nbd{neighborhood}

\def\scB{\mathscr B}
\def\scC{\mathscr C}

\def\scF{\mathscr F}
\def\scG{\mathscr G}

\def\scA{\mathscr A}
\def\scH{\mathscr H}
\def\scV{\mathscr V}
\def\scW{\mathscr W}%\def\T{\Theta}

\def\V{\mathcal V}
\def\W{\mathcal W}

\def\scO{\mathscr O}

\def\scS{\mathscr S}
\def\cS{\mathcal S}

\def\scE{\mathscr E}

\def\b{\beta }

\def\o{\omega}

\def\t{\theta}

\def\vP{\varPsi}
\def\vG{\varGamma}

\def\vt{\vartheta}

\def\vv {\vskip.2cm}
\newtheorem{theorem}{Theorem}[section]
\newtheorem{lemma}[theorem]{Lemma}
\newtheorem{corollary}[theorem]{Corollary}
\newtheorem{proposition}[theorem]{Proposition}

\newtheorem{exa}[theorem]{Example}
\newenvironment{example}{\begin{exa} \em}{\end{exa}}
\newtheorem{exas}[theorem]{Examples}

\newtheorem{prope}[theorem]{Property}

\newtheorem{defini}[theorem]{Definition}
\newenvironment{definition}{\begin{defini} \em}{\end{defini}}

\newtheorem{rema}[theorem]{Remark}
\newenvironment{remark}{\begin{rema} \em}{\end{rema}}

\newenvironment{equationth}{\stepcounter{theorem}\begin{equation}}{\end{equation}}

\newenvironment{proof}{{\noindent \sc Proof: } }{\mbox{ }\hfill$\Box$  
                        \vspace{1.5ex} \par}

\title{\bf  \Large {Sato hyperfunctions  via\\
relative Dolbeault cohomology}}

\author{Naofumi Honda\thanks{Supported by JSPS Grants 18K03316 and 21K03284}, Takeshi Izawa and Tatsuo Suwa\thanks{Supported by  JSPS 
Grants 16K05116 and 20K03572}}

\date{}

\newcommand{\SHmicro}{\mathscr{C}}
\newcommand{\SHCinfty}{\mathscr{E}}

\newcommand{\SHanalf}[1]{\mathscr{A}^{{#1}}}
\newcommand{\SHhyperf}[1]{\mathscr{B}^{{#1}}}

\newcommand{\hdccppk}[4]{{H^{{#1},{#2}}_{\bar{\vartheta}}({#3},{#4})}}
\newcommand{\dccdp}[3]{C^\bullet({#2},{#3};\SHCinfty^{({#1},\bullet)})}
\newcommand{\dccdpkk}[5]{C^{#2}({#4},{#5};\SHCinfty^{({#1},{#3})})}

\newcommand{\dccp}[2]{{\SHCinfty^{(0,\bullet)}({#1},{#2})}}

\newcommand{\hdccpk}[3]{{H^{0,{#1}}_{\bar{\vartheta}}({#2},{#3})}}

\newcommand{\hDCcpk}[3]{{H^{{#1}}_{D}({#2},{#3})}}

\newcommand{\DCcdkk}[4]{C^{#1}({#3},{#4};\SHCinfty^{({#2})})}

\newcommand{\ccp}[3]{C^\bullet({#1},{#2};{#3})}
\newcommand{\hccpk}[4]{H^{#1}({#2},{#3};{#4})}
\newcommand{\ccpk}[4]{C^{#1}({#2},{#3};{#4})}
\begin{document}
%\NoRunningHeads

%\pagestyle{headings}
\pagestyle{plain}

%\pagenumbering{roman}

\maketitle

\noindent
{\bf Abstract}

The relative Dolbeault cohomology which naturally comes up in the theory of 
\v{C}ech-Dolbeault cohomology  turns out to be canonically isomorphic with the local (relative) cohomology of A. Grothendieck and M. Sato so that  it 
provides 
a handy way of representing the latter. 
In this paper  we use this cohomology to give simple explicit expressions of  Sato hyperfunctions, some fundamental operations on them and related local duality theorems. This approach also yields a new insight into the
theory of hyper\fcn s and  leads to a number of  further results and applications. As one of such, we give
an explicit embedding morphism of Schwartz distributions into the space of hyperfunctions.
%that can hardly be achieved by the conventional way.
\bigskip

\noindent
{\it Mathematics Subject Classification} (2020)\,: Primary 32A45, 32C36, 35A27, 46F15,  58J15; Secondary 32C35, 46F05, 55N30.
\vv

\noindent
{\it Keywords and phrases}\,:  Sato hyperfunction,  relative Dolbeault cohomology, Thom class, boundary value
morphism, microlocal analyticity, integration morphism, Schwartz distribution.

%%%%%%%%%%%%%%%%%%%

\lsection{Introduction}
The theory of hyperfunctions was initiated by M.~Sato in \cite{Sa} and then the theory together with its philosophy and methodology
has been vastly developed to form a branch of mathematics called algebraic analysis
(see  M.~Sato, T.~Kawai and M.~Kashiwara \cite{Skk}, 
M.~Kashiwara and P.~Schapira \cite{KS}, P.~Schapira \cite{Sc} and references therein). The space $\mathscr{B}(M)$ of hyperfunctions on a real analytic manifold $M$ is considered to be large enough
in the sense that all the solutions to
a linear differential equation (with irregular singularities) are exhausted in
$\mathscr{B}(M)$,
%the space of hyperfunctions, 
while that of, for example, Schwartz distributions is too small 
as a solution space as shown in H.~Komatsu \cite{K1}. Thus the hyperfunctions may be thought of as natural generalizations of functions and play  important
roles in several areas of mathematics, in particular in the study of  linear differential equations.

A hyperfunction on a one-dimensional space is represented by a holomorphic function
%defined 
on the complement of the real axis in the complex plane and this representative is zero as a hyperfunction if and only if
it extends across the real axis
%to the whole space 
as a holomorphic function.
Contrary to its relative simplicity and concreteness in the one-dimensional case, 
the theory 
%of  hyperfunctions
in higher dimensions is rather abstract, as  hyper\fcn s
are defined in terms of  local cohomology,  with support in $M$ and %with 
coefficients in the sheaf $\scO$ of holomorphic functions on the complexification $X$ of $M$, and the  theory is described in the language of derived functors. Also the representation of  local cohomology is done via relative \v{C}ech cohomology with coefficients in $\scO$.
%\h\ \fcn\ coefficients.
%Thus the therory is not easily accessible for those who mainly work in analysis.
In order to understand  hyperfunctions without  heavy machineries, A.~Kaneko \cite{Ka} 
and M.~Morimoto \cite{Mo} introduced the so-called  ``intuitive representation'' of a hyperfunction, where 
it is represented by a formal sum of several holomorphic functions on infinitesimal wedges
with the edge along $M$. As a recent development in this direction,   D.~Komori and K.~Umeta \cite{KU} generalize
the intuitive representation to the case of local cohomology with coefficients in  an arbitrary sheaf under suitable conditions.

\

On the other hand,  the third named author of this paper observed  that the relative Dolbeault cohomology, which
naturally appears in the theory of \v{C}ech-Dolbeault cohomology in \cite{Su8} (see also \cite{ABST}, \cite{Su10} and \cite{Su11}), is canonically
isomorphic with the local cohomology with coefficients in the sheaf of \h\ forms and that, if we use this cohomology,
a hyperfunction $u \in \mathscr{B}(M)$
in an arbitrary dimension has a similar representation as that on the one-dimensional space.
That is, $u$ has a representative $\tau = (\tau_1, \tau_{01})$, where $\tau_1$
is a $C^\infty$ form of type $(0,n)$ defined on a \nbd\ of $M$ in $X$ 
%the whole space 
and
$\tau_{01}$ is a $C^\infty$ form of type $(0,n-1)$ defined on the complement of $M$, with a natural
cocycle condition.
Such a representation gives a new approach to the hyperfunction theory and makes its treatment  
%much 
more manurable, for instance\,:
\begin{enumerate}
\item[(1)] We can employ tools available in the $C^\infty$ category
such as  partitions of unity.
In particular, we may take a representative with a compact support 
if $u$ has a compact support, which is not possible in 
the framework of
  intuitive representation, as representatives are holomorphic functions.

\item[(2)] The conventional representation of a hyperfunction is obtained via 
\v{C}ech cohomology with coefficients in $\scO$ and hence the existence of a Stein open covering 
is indispensable, which sometimes makes the theory rather complicated.  
In our framework, however, arguments can proceed without  Stein  coverings.

\item[(3)] The integration of a hyperfunction is easily performed as it is represented by a pair of $C^\infty$~forms.
Furthermore, we may concretely describe the morphism 
associated with the cohomological 
residue map, 
%on the space of hyperfunctions, 
since all the related morphisms are constructed using
fine resolutions of sheaves.
\end{enumerate}

The purpose of this paper is to reestablish the theory of hyperfunctions in the framework of relative Dolbeault cohomology %from this viewpoint  
and to
% indicate
further pursue
%developments of 
the theory along this line.
It should be noted, however, that although our approach simplifies various expressions substantially, it is not certain if it also
leads to the simplification of the proofs of such fundamental facts as the pure codimensionality of
$M$ in $X$ \wrt\ $\scO$ and the flabbiness of the sheaf $\scB$ of hyper\fcn s.
Aside from this, 
%as is  seen throughout the paper, 
the  advantages as listed above 
provide a new insight and 
enable us to perform novel 
treatments of the theory.
% of hyperfunctions. 
Furthermore, this point of view leads to several results
%we can give several formulas
that are hardly 
%cannot 
be obtained by the traditional way.
% in the traditional framework of a hyperfunction.
One of such 
%formulas is given in
is in Theorem {\ref{thm:formula_embedding}} below,
%of Section \ref{secembdist},
which 
%makes it possible to 
gives 
%where 
an explicit representative
of the embedded image of  a Schwartz distribution in the space of hyperfunctions.
% is given.

%We establish, in this paper, the theory of hyperfunctions in the framework
%of \v{Cech}-Dolbeault cohomology. 
In the course of our study, we introduce several fundamental methods and ideas,
% to this approach,
 which 
become important ingredients for further works. 
%As a matter of 
In fact
this kind of cohomology theory serves to establish
%for establishment of 
the foundation
of various topics in algebraic analysis such as  the theory of Laplace hyperfunctions
(\cite{K2}, \cite{HU1} and \cite{HU2}) and  the symbol 
theory of analytic pseudodifferential operators 
(\cite{A1},
% and \cite{A2}, 
 see also \cite{A3}). Recently,
%from the viewpoint of \v{C}ech-Dolbeault cohomology, 
K.~Umeta  is working on the former  in \cite{U1}
%developing the theory of Laplace hyperfunctions 
and
D.~Komori is  doing for the latter  in \cite{Kd1},
%that of symbol theory of analytic pseudodifferential operators,
where our fundamental methods introduced in this paper play  essential roles in their
arguments.

\

The paper is organized as follows. In Section \ref{secS}, we recall orientation sheaves and the Thom class,
which are necessary for the orientation free expression of  hyperfunctions and the description of the topological aspect
%side 
of the theory. The definition and
fundamental properties of Sato hyper\fcn s are also   briefly recalled.  We review, in Sections \ref{secrdR}, \ref{secrD} and \ref{seclocald},
 the theory of relative de~Rham and  Dolbeault cohomologies.
%and their fundamental properties. 
In this paper we need the relative de~Rham cohomology in two ways. One is for
the integration theory and the other for the expression of sections of the complexified relative orientation sheaf.
Some canonical morphisms between
the relative de~Rham and the relative Dolbeault cohomologies and local duality morphism
are also discussed.

In Section {\ref{sec:hyperforms}},
we first express  hyperfunctions via relative Dolbeault cohomology and
describe some basic operations  such as differentiation and multiplication by real analytic functions. In fact we work  more generally on hyperforms, which have been traditionally referred to as forms
with hyperfunction coefficients. Then we study the integration of  hyperforms %in the absolute case
and 
establish 
%describe 
the duality theorem of Martineau in our settings. We also give some
 examples related to Dirac's delta function.
We note that the  theory of integration on \v{C}ech-de~Rham cohomology
directly descends to the integration theory on relative Dolbeault cohomology and in turn to the
integration theory of hyperforms. The duality paring that appears in the theorem of Martineau
is explicitly expressed in this context.

Several important morphisms in the hyperfunction theory
are studied in Section~{\ref{sec:morphisms}}.
We first discuss the embedding of real analytic \fcn s into the space of hyper\fcn s.
% as a special case of the boundary value morphism. 
Particularly noteworthy  
here 
is that the canonical morphism from the relative de~Rham cohomology to the
relative Dolbeault cohomology, which is induced from the projection of a $q$-form to its $(0,q)$-component, is
effectively used in the construction of the embedding morphism.
% the boundary value morphism. 
Also the Thom class in the relative 
de~Rham cohomology plays an essential role in this scene of interaction of topology and analysis, in particular,
combined with the above projection,
it is used to give an explicit expression of the embedding morphism 
%of real analytic functions into the space of  hyper\fcn s 
(Corollary~\ref{propemb}).
We then generalize this to the 
construction of the 
%case of 
boundary value morphism, by which we may regard
a holomorphic function on an open edge along $M$ as a hyperfunction. The essential idea here is that
we may ``confine'' the Thom class in the edge under a suitable condition.
We also give an interpretation of the microlocal analyticity in our language (Proposition~\ref{prop:equiv-ss}). This is a direct consequence of our expression of  the spectral map, which is simply a morphism of
relative Dolbeault cohomologies induced from the restriction of differential forms. Such operations as the external product of hyperforms,  the restriction of hyperfunctions and the fiber integration of hyperforms are also 
%neatly 
treated in our framework.
In Section \ref{secembdist}, we give an explicit embedding morphism of distributions into the space of hyper\fcn s 
in our context 
(Theorem \ref{thm:formula_embedding}), as mentioned above.

In Appendix \ref{appen:A},
%for reader's convenience,
we show the compatibility of 
the boundary value morphisms between 
our construction and the original functorial one 
in \cite{Skk} and \cite{KS}.

Some parts of this paper are summarized 
%Summaries of  parts of this paper are 
in \cite{Ho}, \cite{Su10},  \cite{Su12} and \cite{Su13}. 
%Elsewhere we give an explicit embedding morphism of distributions into the space of hyper\fcn s in our context.

We would like to thank the referee for valuable suggestions that improved the presentation of
the paper.

\lsection{Sato hyperfunctions}\label{secS}

\subsection{Relative cohomology}
%\paragraph{Local cohomology\,:}
In the sequel, by a sheaf we mean a sheaf with at least the \str\ of Abelian groups. As general references for the
sheaf cohomology theory, we list  \cite{B} and \cite{G}.
For a sheaf $\scS$   on a topological space $X$ and an open set $V$ in $X$, we denote by 
%$\vG(V;\scS)$  or 
$\scS(V)$ the group of sections on $V$. For an open subset $V'\subset V$ we denote by 
%$\vG(V,V';\scS)$ or 
$\scS(V,V')$ the group of sections 
on $V$ that vanish on $V'$. 
%Considering the closed set $S=V\ssm V'$ in $V$, it will also be denoted by $\vG_{S}(V;\scS)$. 

As reference cohomology theory we adopt the one via flabby resolution.
Thus for an open set $X'$ in $X$, $H^{q}(X,X';\scS)$ denotes the $q$-the cohomology of the complex $\scF^{\bullet}(X,X')$
with $0\ra\scS\ra\scF^{\bullet}$ a flabby resolution. It is uniquely determined modulo canonical \iso s, independently of the chosen
flabby resolution. Considering the closed set $S=X\ssm X'$, it will also be denoted by $H^{q}_{S}(X;\scS)$. 
This cohomology in the first expression is referred to as the {\em relative cohomology} of $\scS$ on $(X,X')$ 
(cf. \cite{Sa}) and in the
second expression the {\em local cohomology} of $\scS$ on $X$ with support in $S$
(cf. \cite{Ha}). 

We recall some of the fundamental properties of the cohomology\,:
\begin{proposition}\label{propfundloc} {\bf 1.} For a triple $(X,X',X'')$ of open sets with $X''\subset X'\subset X$,
there is an exact sequence
\[
\cdots\ra H^{q-1}(X',X'';\scS)\overset{\delta}\ra H^{q}(X,X';\scS)\overset{j^{*}}\ra H^{q}(X,X'';\scS)\overset{i^{*}}\ra H^{q}(X',X'';\scS)\ra\cdots.
\]
\vv

\noindent
{\bf 2. (Excision)} For any open set $V$ in $X$ containing $S$, there is a canonical
\iso
%\,{\rm :}
%\begin{equationth}\label{exciloc}
\[
H^{q}(X,X\ssm S;\scS)\simeq H^{q}(V,V\ssm S;\scS)\quad\text{or equivalently}\quad  H^{q}_{S}(X;\scS)\simeq H^{q}_{S}(V;\scS).
\]
\end{proposition}

\subsection{Derived sheaves}

Let $X$ be a topological space, $S$ a 
%locally 
closed set in $X$ and $\scS$   a sheaf on $X$.

\begin{definition}
The $q$-th derived sheaf $\scH_{S}^{q}(\scS)$ of $\scS$ with support in $S$ is the sheaf defined by the presheaf $V\mapsto H^{q}_{S\cap V}(V;\scS)$.
\end{definition}

By definition, for a point $x$ in $X$,
\begin{equationth}\label{stalk}
\scH_{S}^{q}(\scS)_{x}=\underset{V\ni x}{\underset{\lra}\lim} H^{q}_{S\cap V}(V;\scS).
\end{equationth}
Thus $\scH_{S}^{q}(\scS)$ is supported on $S$ and we may think of it as a sheaf on $S$.
As  such $\scH_{S}^{q}(\scS)$ is defined by the presheaf $U\mapsto H^{q}_{U}(V;\scS)$, where $U$
is an open set in $S$ and $V$ an open set in $X$ containing $U$ as a closed set. Note that by excision $H^{q}_{U}(V;\scS)$
does not depend on the choice of such a $V$. We quote (cf. \cite{KKK})\,:

\begin{proposition}\label{propsh} If $\scH_{S}^{i}(\scS)=0$ for $i<q$, then the above presheaf is a sheaf, i.e., for any
open set $U$ in $S$, there is a
canonical \iso
\[
\scH_{S}^{q}(\scS)(U)\simeq H^{q}_{U}(V;\scS),
\]
where $V$ is an open set in $X$ as above.
\end{proposition}

If the above is the case, we identify $\scH_{S}^{q}(\scS)(U)$ and $H^{q}_{U}(V;\scS)$  hereafter.

\begin{definition} We say that $S$ is {\em purely $q$-codimensional} in $X$ \wrt\ $\scS$ if
\[
\scH_{S}^{i}(\scS)=0\qquad\text{for}\ \ i\ne q.
\]

\end{definition}

Thus if this is the case, the statement of Proposition \ref{propsh} holds.

\subsection{Orientation sheaves and the Thom class}\label{subsOT}

We list \cite{B} and \cite{H} as references for this subsection. As to the Alexander duality, we refer to \cite{Br1}
and \cite{Su7}. Throughout the paper, the \mfd s 
%we consider 
are assumed to have a countable basis, thus 
they are
paracompact and have 
%a countable number of 
countably many connected components.
We discuss orientation sheaves and orientations in detail below. For the moment recall that 
%the orientability and 
orientations of vector spaces, \mfd s and simplices are defined as in \cite{ST}.

Let $X$ be a  $C^{\infty}$ \mfd\ of 
%class at least $C^{1}$ of
dimension $m$
and $\Z_{X}$  the constant sheaf on $X$ with stalk $\Z$. Then $H^{q}(X;\Z_{X})$ is canonically isomorphic
with the %singular (or simplicial) 
simplicial cohomology $H^{q}(X;\Z)$ of $X$ with $\Z$-coefficients on finite chains
and, for a closed set $S$,  $H^{q}_{S}(X;\Z_{X})$ is canonically isomorphic
with the relative cohomology $H^{q}(X,X\ssm S;\Z)$. We denote by $H_{q}(X;\Z)$ the Borel-Moore homology of $X$,
which in our case  is canonically isomorphic with the 
%singular (or simplicial) 
simplicial homology of $X$ of locally finite chains.

Suppose $X$ is orientable and is specified with an orientation, i.e., oriented. For a triangulation 
%$K$ 
of $X$, 
%by convention,
we orient simplices and dual cells so that the orientation of the cell dual  to a simplex followed by that of the simplex gives the
orientation of $X$. For a subcomplex $S$ of $X$ \wrt\ some triangulation, we have the Alexander duality
\begin{equationth}\label{isoalex}
A:H^{q}(X,X\ssm S;\Z)\overset\sim\lra H_{m-q}(S;\Z),
\end{equationth}
which is given by the left cap product with the fundamental class $[X]$, the class of the sum of all the $m$-simplices
in $X$.

First note that we have
\[
H^{q}_{\{0\}}(\R^{l};\Z_{\R^{l}})\simeq\begin{cases}\Z&\qquad
q=l,\\
0&\qquad q\ne l,
\end{cases}
\]
which 
can be seen  from the long exact sequence. We may  
as well 
interpret it as the Alexander duality
\[
A:H^{q}(\R^{l},\R^{l}\ssm\{0\};\Z)\overset\sim\lra H_{l-q}(\{0\};\Z).
\]
In particular, since the \iso\ $H^{l}(\R^{l},\R^{l}\ssm\{0\};\Z)\simeq H_{0}(\{0\};\Z)=\Z$ is given by
the left cap product with $[\R^{l}]$,
%Thus 
specifying a generator of $H^{l}_{\{0\}}(\R^{l};\Z_{\R^{l}})$ is equivalent to
specifying an orientation of $\R^{l}$.

Let $\pi:E\ra M$ be a $C^{\infty}$ real  \vb\ of rank $l$ on a $C^{\infty}$ \mfd\ $M$. We identify $M$ with the 
%denote by $Z$ the 
image of the embedding $i:M\hra E$ by  zero section. By excision, the sheaf $\scH^{q}_{M}(\Z_{E})$ may be identified with the sheaf
% and set $E_{0}=E\ssm Z$ and $\pi_{0}=\pi|_{E_{0}}$. 
$\scH^{q}_{M}(\pi;\Z_{E})$
defined by the presheaf $U\mapsto H^{q}_{U}(\pi^{-1}(U);\Z_{E})$. Setting $E_{x}=\pi^{-1}(x)$ for  $x\in M$, we have an \iso
\[
r^{*}_{x}:\scH^{q}_{M}(\pi;\Z_{E})_{x}\overset\sim\lra H^{q}_{\{0\}}(E_{x};\Z_{E_{x}})\simeq\begin{cases}\Z&\qquad
q=l,\\
0&\qquad q\ne l.
\end{cases}
\]
Thus $M$ is purely $l$-codimensional in $E$ \wrt\ $\Z_{E}$.

\begin{definition} We set $or_{M/E}=\scH^{l}_{M}(\Z_{E})$ and call it the {\em orientation sheaf} of the bundle $E\ra M$.
\end{definition}

\begin{remark}\label{remori} The sheaf $or_{M/E}$ is a sheaf on $M$. It is what is referred to as the relative orientation
sheaf of the embedding $i:M\hra E$ 
%of $M$ by the zero section 
and describes the orientations of the fibers of 
$\pi:E\ra M$.
\end{remark}

The sheaf $or_{M/E}$
 is locally constant, i.e., locally isomorphic with a product sheaf with stalk $\Z$ and, by Proposition \ref{propsh}, we have 
 %a canonical \iso
\[
or_{M/E}(U)= H^{l}_{U}(\pi^{-1}(U);\Z_{E}).
\]
In particular $or_{M/E}(M)= H^{l}_{M}(E;\Z_{E})$. Note that there is a canonical \iso 
\[
or_{M/E}\otimes_{{}_{\Z_{M}}} or_{M/E}\simeq\Z_{M},\qquad \psi\otimes\psi\leftrightarrow 1,
%\qquad\scH\hspace{-.6mm}om(or_{M/E},\Z_{M})\simeq or_{M/E},
\]
where $\psi$ is any generator  of  $or_{M/E,x}$, $x\in M$.

The bundle $E$ is said to be {\em orientable} if $or_{M/E}$ is a constant sheaf. This is equivalent to saying that
$or_{M/E}$ admits a global section $\psi\in or_{M/E}(M)= H^{l}_{M}(E;\Z_{E})$
%If this is the case, an orientation of $E$ is a section $\tau$ of 
%$or_{M/E}$ on $M$ 
\st\ $r^{*}_{x}(\psi(x))$ generates $H^{q}_{\{x\}}(E_{x};\Z_{E_{x}})$ over $\Z$ for
every $x\in M$, i.e., we may specify an orientation of each fiber in a ``coherent manner''. 
%Suppose $E$ is orientable. 
A section $\psi$ as above is called an {\em orientation} of $E$. 
%There are exactly two orientations, if $E$ is orientable.
Once we specify such a $\psi$ we say that $E$ is oriented.
From the definition, if the bundle $E$ is orientable, we have $or_{M/E}(M)\simeq\Z_{M}(M)$. The \iso\ is not uniquely determined,
however once we specify an orientation $\psi$, it is determined so that $\psi$ is sent to the constant \fcn\ $1$.
Note that if $M$ is connected, there are exactly two orientations.

Let $\pi:E\ra M$ be an oriented real \vb\ of rank $l$. A frame $(e_{1},\dots,e_{l})$ of $E$ on an open set
$U$ in $M$ is said to be {\em positive} if $(e_{1}(x),\dots,e_{l}(x))$ determines the prescribed orientation
of $E_{x}$ for every $x\in U$.

\begin{definition}\label{defthom} Let $E$ be an oriented real \vb\ of rank $l$. The {\em Thom class} $\varPsi_{E}$ of $E$
is the global section of $or_{M/E}$ that defines the orientation.
% and $\tau$ is called the {\em Thom class} of $E$ and is denoted by $\varPsi_{E}$. 
It may be thought of as being in $H^{l}_{M}(E;\Z_{E})\simeq H^{l}(E,E\ssm M;\Z)$. 
\end{definition}

Let $M$ be a $C^{\infty}$ \mfd\  of dimension $n$. The orientation sheaf $or_{M}$ of $M$ is defined to be  the orientation sheaf of the tangent bundle
$TM$. By excision and the exponential map, we have, for any $x\in M$ and an open \nbd\ $U$ 
%in $M$ containing 
of $x$,  a canonical \iso
\[
or_{M,x}\simeq H^{n}_{\{x\}}(U;\Z),
\]
which relates the orientation of $TM$ and that of $M$. The \mfd\ $M$ is  orientable \iff\ $TM$ is and, if this is the case,  we orient them so that if
 $(x_{1},\dots,x_{n})$ is a positive  coordinate system on $M$,   $(\frac \partial{\partial x_{1}},\dots, \frac \partial{\partial x_{n}})$ is a positive frame of $TM$.

Letting $\scE^{(n)}_{M}$ be the sheaf of $C^{\infty}$ $n$-forms on $M$,
the sheaf
%\[
$\scV_{M}=\scE^{(n)}_{M}\otimes_{{}_{\Z_{M}}} or_{M}$
%\]
is called the {\em sheaf of densities} on $M$. Denoting by $\vG_{c}(M;\scV_{M})$ the space of sections with
compact support, we have the integration
\[
%\int_{M}:
\vG_{c}(M;\scV_{M})\xrightarrow{\,\,\,\int_{M}\,\,\,}\R.
\]
If $M$ is  orientable and if an orientation is specified, there is an
\iso\
$\scV_{M}\simeq\scE^{(n)}_{M}$.

Let $E\ra M$ be a real \vb\ and $i:M\hra E$ the embedding by  zero section. Noting that the normal bundle $T_{M}E$ of $M$ in $E$ is the bundle $E$ itself, we have the exact sequence of \vb s on $M$\,:
\[
0\lra TM\lra TE|_{M}\lra E\lra 0.
\]
%Then 
From this we have \iso s
%\,:
\begin{equationth}\label{orifib}
i^{-1}{or}_{E} \simeq{or}_{M/E}\otimes {or}_M,
\qquad
{or}_{M/E}\simeq i^{-1}{or}_{E}\otimes {or}_M,
\end{equationth}
where ${or}_{E}$ denotes the orientation sheaf of the total space $E$ and the tensor products are over $\Z_{M}$. Note that there are several choices for the above \iso s.
From \eqref{orifib}, we see that if $M$ is orientable and $E$ is orientable as a bundle,
the total space $E$ is also orientable. Also if $M$ and the total space $E$ are orientable, $E$ is orientable 
as a bundle.

Let $X$ be a $C^{\infty}$ \mfd\ of dimension $m$ and $M\subset X$ a closed sub\mfd\ of dimension $n$. Then we have the exact sequence
\begin{equationth}\label{exactemb}
0\lra TM\lra TX|_{M}\overset\varpi\lra T_{M}X\lra 0
\end{equationth}
with $T_{M}X$ the  normal bundle of $M$ in $X$. The {\em relative orientation
sheaf} $or_{M/X}$ is defined to be the orientation sheaf of the normal bundle $\pi:T_{M}X\ra M$. By the tubular \nbd\ theorem and excision,
we have, for an open set $U\subset M$,
\begin{equationth}\label{isoori}
or_{M/X}(U)= H^{l}_{U}(\pi^{-1}(U);\Z_{T_{M}X})\simeq H^{l}_{U}(V;\Z_{X}),
\end{equationth}
where $l=m-n$ and $V$ is an open set in $X$ containing $U$ as a closed set. Thus we have a canonical \iso
\[
or_{M/X}\simeq\scH_{M}^{l}(\Z_{X}).
\]
Denoting by $i:M\hra X$ the embedding, we have 
%canonical 
\iso s (cf. \eqref{orifib})
%\,:
\begin{equationth}\label{orisub}
i^{-1}or_{X}\simeq or_{M/X}\otimes or_{M},\qquad or_{M/X}\simeq i^{-1}or_{X} \otimes or_{M}.
\end{equationth}
Thus if $M$ and $X$ are orientable, so is $T_{M}X$.

From (\ref{isoori}), we have
\[
or_{M/X}(M)= H^{l}_{M}(T_{M}X;\Z_{T_{M}X})\simeq H^{l}_{M}(X;\Z_{X}).
\]
Suppose the bundle
%$X$ and $M$ are oriented and 
$T_{M}X$ is oriented.
% according to the convention above. 
The Thom class $\varPsi_{M}$ of $M$ in $X$ is then defined to be the class in $H^{l}_{M}(X;\Z_{X})
\simeq H^{l}(X,X\ssm M;\Z)$ that corresponds to the Thom class of $T_{M}X$ by the  \iso\ above.

\subsection{Hyper\fcn s and hyperforms}\label{sshyp}

Let $M$ be a real analytic \mfd\ of dimension $n$ and $X$ its complexification. Let $\scO_{X}$ denote
the sheaf of \h\ \fcn s on $X$.
We quote (cf. \cite{KKK}, \cite{KS}, \cite{Skk})\,:

\begin{theorem}\label{thfund} {\bf 1.} $M$ is purely $n$-codimensional in $X$ \wrt\ $\scO_{X}$, i.e., $\scH^{i}_{M}(\scO_{X})=0$
for $i\ne n$.
\vv

\noindent
{\bf 2.} The sheaf $\scH^{n}_{M}(\scO_{X})$ is flabby.
\end{theorem}

We recall that the sheaf of Sato hyper\fcn s on $M$ is defined by
\[
\scB_{M}=\scH^{n}_{M}(\scO_{X})\otimes_{{}_{\Z_{M}}} or_{M/X}.
\]
%By Theorem \ref{thfund}.\,2, we see that $\scB_{M}$ is flabby.

Denoting by $\scO^{(p)}_{X}$ the sheaf of \h\ $p$-forms on $X$, we introduce the following 
\begin{definition} The sheaf of {\em $p$-hyperforms} is defined by
\[
\scB^{(p)}_{M}=\scH^{n}_{M}(\scO^{(p)}_{X})\otimes_{{}_{\Z_{M}}} or_{M/X}.
\]
\end{definition}

It is what is referred to as the sheaf of $p$-forms with coefficients in hyper\fcn s. Note that $M$ is purely $n$-codimensional \wrt\ $\scO^{(p)}_{X}$
and that $\scB^{(p)}_{M}$ is flabby.

Since $X$ is a complex \mfd, it is always orientable.
Thus, by \eqref{orisub}, the bundle $T_{M}X$ is orientable \iff\ $M$ is. If this is the case, for any open set $U\subset M$, we have (cf. (\ref{isoori}))
\begin{equationth}\label{isosec}
\scB^{(p)}_{M}(U)=H^{n}_{U}(V;\scO^{(p)}_{X})\otimes_{{}_{\Z_{M}(U)}} H^{n}_{U}(V;\Z_{X}),
\end{equationth}
where $V$ is an open set in $X$ containing $U$ as a close set. We refer to such  $V$ a {\em complex \nbd} of $U$
in $X$. Moreover, if we specify 
% If we specify the orientations of $X$ and $M$, 
an orientation 
of $T_{M}X$,
% is determined by the above convention and 
we have a canonical
\iso\ $or_{M/X}\simeq\Z_{M}$ so that we have a canonical \iso
%\begin{equationth}\label{isocan1}
\[
\scB^{(p)}_{M}\simeq\scH^{n}_{M}(\scO^{(p)}_{X})
\]
%\end{equationth}
and, for any open set $U\subset M$,
%\begin{equationth}\label{isocan2}
%\[
$\scB^{(p)}_{M}(U)\simeq H^{n}_{U}(V;\scO^{(p)}_{X})$
%\]
%\end{equationth}
with $V$  a complex \nbd\ of $U$.

In the sequel, at some point the cohomology $H^{n}_{U}(V;\Z_{X})$ is embedded in $H^{n}_{U}(V;\C_{X})$, which is represented
by the relative de~Rham cohomology, while $H^{n}_{U}(V;\scO^{(p)}_{X})$ will  be represented by the relative
Dolbeault cohomology.
%%%%
\lsection{Relative de~Rham cohomology}\label{secrdR}

Let $X$ be a $C^{\infty}$ \mfd\ of dimension $m$. We denote by $\scE^{(q)}_{X}$  the sheaf of $C^{\infty}$ $q$-forms  on $X$. We omit the suffix $X$ if there is no fear of confusion.

%\subsection{\v{C}ech-de~Rham cohomology}
\subsection{\v{C}ech-de~Rham cohomology}

We refer to \cite{BT} and  \cite{Su2} for  details on  the \v{C}ech-de~Rham cohomology. For relative de~Rham
cohomology and the Thom class in this context, see \cite{Su2}. 
%As to the Alexander duality, we refer to \cite{Br1} and \cite{Su7}.

\paragraph{de~Rham cohomology\,:}
The $q$-th de~Rham cohomology $H^{q}_{d}(X)$ of $X$  is the $q$-th cohomology
of the complex $(\scE^{(\bullet)}(X),d)$, $d:\scE^{(q)}(X)\ra \scE^{(q+1)}(X)$. The de~Rham  theorem says that there
is an \iso
%\ (cf. Remark \ref{remcan} below)\,:
%\begin{equationth}\label{dRth}
\[
H^{q}_{d}(X)\simeq H^{q}(X;\C).
\]
%\end{equationth}
Among the \iso s, there is a canonical one, i.e., the one that regards a  $q$-form $\o$ as a cochain
that assignes to each oriented $q$-simplex  the integration of $\o$ on the simplex.
% (cf. \cite{Su10}).

\paragraph{\v{C}ech-de~Rham cohomology\,:}
The \v{C}ech-de~Rham cohomology may be defined for an arbitrary covering of $X$.
%a $C^{\infty}$ \mfd. 
Here we recall the case
of coverings consisting of two open sets.
% and refer to \cite{BT}, \cite{Su2} for the general case and details.

Let $\V=\{V_{0},V_{1}\}$ be an open covering of $X$ and set $V_{01}=V_{0}\cap V_{1}$.
We set
\[
\scE^{(q)}(\V)=\scE^{(q)}(V_{0})\oplus\scE^{(q)}(V_{1})\oplus\scE^{(q-1)}(V_{01}).
\]
Thus an element in $\scE^{(q)}(\V)$ is expressed by a triple 
$\sigma=(\sigma_{0},\sigma_{1},\sigma_{01})$. 
We define 
the differential
\[
D:\scE^{(q)}(\V)\lra \scE^{(q+1)}(\V)\qquad\text{by}\ \ 
D (\sigma_{0},\sigma_{1},\sigma_{01})=(d\sigma_{0},d\sigma_{1},\sigma_{1}-\sigma_{0}-d\sigma_{01}).
\]
Then we see that $D\circ D=0$.

\begin{definition} The $q$-th {\em \v{C}ech-de~Rham cohomology} $H_{D}^{q}(\V)$ of $\V$ is the 
$q$-th cohomology of the complex $(\scE^{(\bullet)}(\V),D)$.
\end{definition}

\begin{theorem}\label{thdCd} The inclusion $\scE^{(q)}(X)\hra \scE^{(q)}(\V)$ given by $\o\mapsto (\o|_{V_{0}},\o|_{V_{1}},0)$
induces an \iso
\[
H^{q}_{d}(X)\overset\sim\lra H^{q}_{D}(\V).
\]
\end{theorem}

Note that the inverse is given by assigning to the class of $ (\sigma_{0},\sigma_{1},\sigma_{01})$ the class of 
$\rho_{0}\sigma_{0}+\rho_{1}\sigma_{1}-d\rho_{0}\wedge\sigma_{01}$, where $\{\rho_{0},\rho_{1}\}$ is a $C^{\infty}$
partition of unity subordinate to $\V$.

\subsection{Relative de~Rham cohomology}

Let $S$ be a closed set in $X$. Letting $V_{0}=X\ssm S$ and  $V_{1}$  a \nbd\ of $S$ in $X$, we  consider the 
covering $\V=\{V_{0},V_{1}\}$ of $X$. We also set $\V'=\{V_{0}\}$ and
\[
\scE^{(q)}(\V,\V')=\{\,\sigma\in\scE^{(q)}(\V)\mid \sigma_{0}=0\,\}=\scE^{(q)}(V_{1})\oplus\scE^{(q-1)}(V_{01}).
\]
Then we see that $(\scE^{(\bullet)}(\V,\V'),D)$ is a subcomplex of $(\scE^{(\bullet)}(\V),D)$.

\begin{definition}\label{defreldR} The $q$-th {\em  relative de~Rham cohomology} $H_{D}^{q}(\V,\V')$ of $(\V,\V')$  is the 
$q$-th cohomology of the complex $(\scE^{(\bullet)}(\V,\V'),D)$.
\end{definition}

From the exact sequence of  complexes
\[
0\lra \scE^{(\bullet)}(\V,\V')\overset{j^{*}}\lra\scE^{(\bullet)}(\V)\overset{i^{*}}\lra\scE^{(\bullet)}(V_{0})\lra 0,
\]
where $j^{*}(\sigma_{1},\sigma_{01})=(0,\sigma_{1},\sigma_{01})$ and $i^{*}(\sigma_{0},\sigma_{1},\sigma_{01})=\sigma_{0}$,
we have the  exact sequence
\begin{equationth}\label{lexactreldR}
\cdots\lra H^{q-1}_{d}V_{0})\overset{\delta}\lra H^{q}_{D}(\V,\V')\overset{j^{*}}\lra H^{q}_{D}(\V)\overset{i^{*}}\lra H^{q}_{d}(V_{0})\lra\cdots,
\end{equationth}
where $\delta$ assigns to the class of $\t$ the class of $(0,-\t)$. From 
the above sequence 
and Theorem \ref{thdCd}, we have\,:

\begin{proposition}
The cohomology $H^{q}_{D}(\V,\V')$ is determined uniquely, modulo canonical \iso s, independently of the choice of $V_{1}$.
\end{proposition}

In view of the above we denote $H^{q}_{D}(\V,\V')$ also by $H^{q}_{D}(X,X\ssm S)$ and call it the relative de~Rham cohomology of the pair $(X,X\ssm S)$. It is not difficult to see the following\,:

\begin{proposition}[Excision] For any open set $V$ containing $S$, there is a canonical \iso
\[
H^{q}_{D}(X,X\ssm S)\simeq H^{q}_{D}(V,V\ssm S).
\]\label{propexcisiondR}
\end{proposition}

The relative cohomology share other fundamental properties with the relative cohomology of $X$ with coefficients
in $\C$. In fact we have (cf. \cite{Su7}, \cite{Su11})\,:

\begin{theorem}[Relative de~Rham theorem]
There is a canonical \iso
\[
H^{q}_{D}(X,X\ssm S)\simeq H^{q}(X,X\ssm S;\C).
\]\label{thdRrel}
\end{theorem}

The excision in Proposition \ref{propexcisiondR} is compatible with the excision 
%(\ref{exciloc}) 
in Proposition \ref{propfundloc}.\,2 for $\scS=\C$ via the above \iso.

\paragraph{Complexification of the relative orientation sheaf\,:} Let $X$ be a $C^{\infty}$ \mfd\ of dimension $m$ and $M\subset X$ a closed sub\mfd\ of dimension $n$. Set $l=m-n$. We define the complexification of the relative orientation sheaf
by $or^{c}_{M/X}=or_{M/X}\otimes_{{}_{\Z_{M}}}\C_{M}$. Then by (\ref{isoori}) and Theorem \ref{thdRrel}, we have,
for any open set $U$ in $M$,
\begin{equationth}\label{cori}
or^{c}_{M/X}(U)\simeq H^{l}_{D}(V,V\ssm U),
\end{equationth}
where $V$ is an open set in $X$ containing $U$ as a closed set.

%%%%
\subsection{Integration}\label{ssint}

The integration on the \v{C}ech-de~Rham cohomology in general is defined by considering honeycomb systems. Here
we recall the relevant case.

Let $X$ be an oriented $C^{\infty}$ \mfd\ of dimension $m$. First we assume that $X$ is compact. Then the integration of $m$-forms induces the integration
\begin{equationth}\label{intdR}
%\int_{X}:
H^{m}_{d}(X)\xrightarrow{\,\,\,\int_{X}\,\,\,}\C.
\end{equationth}

Now let $K$ be a compact set in $X$ ($X$ may not be compact). Letting $V_{0}=X\ssm K$ and $V_{1}$ a \nbd\
of $K$, we  consider the coverings $\V_{K}=\{V_{0},V_{1}\}$ and $\V'_{K}=\{V_{0}\}$ of $X$ and $X\ssm K$. Let $R_{1}$ be an $m$-dimensional compact 
\mfd\ with $C^{\infty}$ boundary in $X$ containing $K$ in its interior. We set $R_{01}=-\partial R_{1}$ and 
define
\[
%\int_{X}:
\scE^{(m)}(\V_{K},\V'_{K})\xrightarrow{\,\,\,\int_{X}\,\,\,}\C\qquad\text{by}\ \ \int_{X}\sigma=\int_{R_{1}}\sigma_{1}+\int_{R_{01}}\sigma_{01}.
\]
Then it induces the integration
\begin{equationth}\label{intreldR}
%\int_{X}:
H^{m}_{D}(X,X\ssm K)\xrightarrow{\,\,\,\int_{X}\,\,\,}\C.
\end{equationth}
%%%

%\paragraph{Thom class in differential forms\,:} 
\subsection{Thom class in differential forms}\label{ssThomdiff} Let $\pi:E\ra M$ be a real oriented \vb\ of rank $l$ on a $C^{\infty}$ \mfd\ $M$. 
%We identify $M$ with the image of the zero section. 
%Suppose the bundle 
%it is orientable as a bundle and is specified with an orientation, i.e., 
%is oriented. 
We have the Thom class $\vP_{E}$ in $H^{l}(E,E\ssm M;\Z)\simeq 
%H^{l}_{M}(E;\Z_{E})
\scH^{l}_{M}(\Z_{E})(M)$
(cf. Definition~\ref{defthom}). Recall that $\vP_{E}$ is characterized as the class 
%in $H^{l}(E,E\ssm M;\Z)$
%\st\ its 
whose restriction $r_{x}^{*}(\vP_{E}(x))\in H^{l}(E_{x},E_{x}\ssm \{0\};\Z)\simeq\Z$ to each fiber $E_{x}$ %image by 
%\[
%r^{*}_{x}:H^{l}(E,E\ssm M;\Z)\ra H^{l}(E_{x},E_{x}\ssm \{0\};\Z)\simeq\Z
%\]
 is the prescribed generator of $H^{l}(E_{x},E_{x}\ssm \{0\};\Z)$, i.e., the Thom class of the bundle $E_{x}\ra\{x\}$.
We denote by $\vP_{E}^{c}$ the image of $\vP_{E}$ by the canonical morphism
$H^{l}(E,E\ssm M;\Z)\ra H^{l}(E,E\ssm M;\C)\simeq H^{l}_{D}(E,E\ssm M)$.

%\paragraph{Alexander duality in relative de~Rham cohomology\,:} 
%Let $X$ be an oriented $C^{\infty}$ \mfd\ and $S$  a subcomplex of $X$ \wrt\ some triangulation of $X$. Then, i
In view of Theorem \ref{thdRrel},
we may describe the Alexander duality \eqref{isoalex} with $\C$-coefficients  in terms of relative de~Rham cohomology.
We recall this in the case of   duality for the pair $(\R^{l},0)$.
%$H^{n}(\R^{n},\R^{n}\ssm\{0\};\C)\simeq H_{0}(\{0\};\C)=\C$.
%\]
Thus let $V_{0}=\R^{l}\ssm\{0\}$ and $V_{1}$ a \nbd\ of $0$ in $\R^{l}$ and consider the coverings
$\V=\{V_{0},V_{1}\}$ and $\V'=\{V_{0}\}$ of $\R^{l}$ and $\R^{l}\ssm\{0\}$. Let $R_{1}$ be a closed $n$-ball around $0$ in $V_{1}$ and set $R_{01}=-\partial R_{1}$. Then we have (cf. \cite[Ch.II, Example 3.13]{Su2})%, \cite{Su7})
\,:
\begin{proposition}\label{propalexinreldr} Suppose $\R^{l}$ is oriented some way. Then the Alexnader \iso
\[
H^{l}_{D}(\V,\V')\simeq H^{l}(\R^{l},\R^{l}\ssm\{0\};\C)\overset\sim\lra H_{0}(\{0\};\C)=\C.
\]
is induced from the morphism %that assigns to a cocycle $(\sigma_{1},\sigma_{01})$
\[
\scE^{(l)}(\V,\V')\lra\C\qquad\text{given by}\ \ (\sigma_{1},\sigma_{01})\mapsto\int_{R_{1}}\sigma_{1}
+\int_{R_{01}}\sigma_{01}.
\]
\end{proposition}

Note  that, since $H^{l}(\R^{l};\C)=0$, $\delta$  in \eqref{lexactreldR} is surjective and we may always take a cocycle of the form $(0,-\t)$ with $\t$ a closed $(l-1)$-form on $\R^{l}\ssm\{0\}$.

\begin{corollary}\label{corcharThom} The Thom class $\vP_{\R^{l}}^{c}$ of the bundle $\R^{l}\ra\{\rm pt.\}$ is characterized as a class
 in $H^{l}_{D}(\R^{l},\R^{l}\ssm\{0\})$ that is represented by a cocycle $(0,-\t)$ with $\int_{\partial R_{1}}\t=1$.
 %$(\sigma_{1},\sigma_{01})$ \st\ $\int_{R_{1}}\sigma_{1}+\int_{R_{01}}\sigma_{01}=1$.
\end{corollary}

As a particular choice for $\t$ as above, we have the angular form $\psi_{l}$ on 
$\R^{l}=\{(t_{1},\dots,t_{l})\}$. It is given by
\begin{equationth}\label{ang}
\psi_{l}=C_{l}\frac{\sum_{i=1}^{l}\varPhi_{i}(t)}{\Vert t\Vert^{l}},\qquad\varPhi_{i}(t)=(-1)^{i-1}t_{i}\,dt_{1}\wedge\cdots\wedge\widehat{dt_{i}}\wedge\cdots\wedge dt_{l},
\end{equationth}
where\ \  $\widehat{}$\ \  means the form under it is to be omitted. The constant $C_{l}$ above is given by $\frac {(k-1)!}{2\pi^{k}}$ if $l=2k$ and by $\frac{(2k)!}{2^{l}\pi^{k}k!}$ if $l=2k+1$. In particular,
\[
\psi_{1}=\frac 1 2\frac t{|t|}.
\]
The important fact is that the form is defined and closed in $\R^{l}\ssm\{0\}$ and, if we orient $\R^{l}$ so that $(t_{1},\dots,t_{l})$ is a positive coordinate system, $\int_{S^{l-1}}\psi_{l}=1$ for an $(l-1)$-sphere $S^{l-1}$ 
around $0$ in $\R^{l}$, oriented 
%where $\R^{l}$ is  oriented so that $(x_{1},\dots,x_{l})$ is a positive
%coordinate system and $S^{l-1}$ 
as the boundary of an $l$-ball.

We may think of Corollary \ref{corcharThom} also as the characterization of the Thom class
$\vP_{E}^{c}$ of the product bundle $E=M\times\R^{l}$.
%Let $E=M\times\R^{l}$. 
Letting $W_{0}=E\ssm M$ and $W_{1}=E$, we consider the coverings $\W=\{W_{0},W_{1}\}$ and $\W'=\{W_{0}\}$ of  $E$ and $E\ssm M$.
%By Corollary \ref{corcharThom}, 
Then we have\,:
% (cf. \cite[Ch.II]{Su2})

\begin{proposition}\label{thtrivial}
For the product bundle $E=M\times \R^{l}$, whose fiber $\R^{l}$ is oriented as above,  the cocycle
$(0,-\psi_{l})\in \scE^{(l)}(\W,\W')$
represents the Thom class $\varPsi_{E}^{c}$.
\end{proposition}

\lsection{Relative Dolbeault cohomology}\label{secrD}

\v{C}ech-Dolbeault cohomology and relative Dolbeault cohomology are defined the same way as in the de~Rham case,
replacing the de~Rham complex with the Dolbeault complex. In this section we recall the relevant part of the theory
and refer to \cite{Su8}, \cite{Su10}, \cite{Su11} for details.

Let $X$ be a complex \mfd\ of dimension $n$. We denote by $\scE^{(p,q)}_{X}$ and $\scO^{(p)}_{X}$ the sheaves of $C^{\infty}$ $(p,q)$-forms and \h\ $p$-forms on $X$, \r. 
%We omit the suffix $X$ if there is no fear of confusion.

\subsection{\v{C}ech-Dolbeault cohomology}

\paragraph{Dolbeault cohomology\,:}

The Dolbeault cohomology $H^{p,q}_{\bp}(X)$ of $X$ of type $(p,q)$ is the $q$-th cohomology
of the complex $(\scE^{(p,\bullet)}(X),\bp)$, $\bp:\scE^{(p,q)}(X)\ra \scE^{(p,q+1)}(X)$. The Dolbeault theorem says that there
is an \iso
%\ (cf. Remark \ref{remcan} below)\,:
\begin{equationth}\label{dth}
H^{p,q}_{\bp}(X)\simeq H^{q}(X,\scO^{(p)}).
\end{equationth}
Note that among the \iso s, there is a canonical one.
% (cf. \cite{Su10}).

\paragraph{\v{C}ech-Dolbeault cohomology\,:}

The \v{C}ech-Dolbeault cohomology may be defined for an arbitrary covering of
% $X$.
a complex \mfd. 
Here we recall the case
of coverings consisting of two open sets.
% and refer to \cite{Su8}, \cite{Su10} for the general case and details.

Let $\V=\{V_{0},V_{1}\}$ be an open covering of $X$ and set $V_{01}=V_{0}\cap V_{1}$.
We set
\[
\scE^{(p,q)}(\V)=\scE^{(p,q)}(V_{0})\oplus\scE^{(p,q)}(V_{1})\oplus\scE^{(p,q-1)}(V_{01}).
\]
Thus an element in $\scE^{(p,q)}(\V)$ is expressed by a triple $\xi=(\xi_{0},\xi_{1},\xi_{01})$. 
We define the differential
\[
\bar\vt:\scE^{(p,q)}(\V)\lra \scE^{(p,q+1)}(\V)\qquad\text{by}\ \ 
\bar\vt (\xi_{0},\xi_{1},\xi_{01})=(\bp\xi_{0},\bp\xi_{1},\xi_{1}-\xi_{0}-\bp\xi_{01}).
\]
Then we see that $\bar\vt\circ\bar\vt=0$.

\begin{definition} The {\em \v{C}ech-Dolbeault cohomology} $H_{\bar\vt}^{p,q}(\V)$ of $\V$ of type $(p,q)$ is the 
$q$-th cohomology of the complex $(\scE^{(p,\bullet)}(\V),\bar\vt)$.
\end{definition}

\begin{theorem}\label{thDCD} The inclusion $\scE^{(p,q)}(X)\hra \scE^{(p,q)}(\V)$ given by $\o\mapsto (\o|_{V_{0}},\o|_{V_{1}},0)$
induces an \iso
\[
H^{p,q}_{\bp}(X)\overset\sim\lra H^{p,q}_{\bar\vt}(\V).
\]
\end{theorem}

%Note that t
The inverse is induced from  the  assignment  $ (\xi_{0},\xi_{1},\xi_{01})\mapsto \rho_{0}\xi_{0}+\rho_{1}\xi_{1}-\bp\rho_{0}\wedge\xi_{01}$.
%, where $\{\rho_{0},\rho_{1}\}$ is a $C^{\infty}$ partition of unity subordinate to $\V$.

\subsection{Relative Dolbeault cohomology}

Let $S$ be a closed set in $X$. Letting $V_{0}=X\ssm S$ and  $V_{1}$  a \nbd\ of $S$ in $X$, we consider the 
coverings $\V=\{V_{0},V_{1}\}$ and $\V'=\{V_{0}\}$ of $X$ and $X\ssm S$. We  set  
\[
\scE^{(p,q)}(\V,\V')=\{\,\xi\in\scE^{(p,q)}(\V)\mid \xi_{0}=0\,\}=\scE^{(p,q)}(V_{1})\oplus\scE^{(p,q-1)}(V_{01}).
\]
Then we see that $(\scE^{(p,\bullet)}(\V,\V'),\bar\vt)$ is a subcomplex of $(\scE^{(p,\bullet)}(\V),\bar\vt)$.

\begin{definition} The {\em  relative Dolbeault cohomology} $H_{\bar\vt}^{p,q}(\V,\V')$ of $(\V,\V')$ of type $(p,q)$ is the 
$q$-th cohomology of the complex $(\scE^{(p,\bullet)}(\V,\V'),\bar\vt)$.
\end{definition}

The following exact sequence is obtained as (\ref{lexactreldR})\,:

\begin{equationth}\label{lexactrelD}
\cdots\lra H^{p,q-1}_{\bp}(V_{0})\overset{\delta}\lra H^{p,q}_{\bar\vt}(\V,\V')\overset{j^{*}}\lra H^{p,q}_{\bar\vt}(\V)\overset{i^{*}}\lra H^{p,q}_{\bp}(V_{0})\lra\cdots,
\end{equationth}
where $\delta$ assigns to the class of $\t$ the class of $(0,-\t)$. As in the de~Rham case,
%From the above sequence and Theorem \ref{thDCD}, 
we have\,:

\begin{proposition}
The cohomology $H^{p,q}_{\bar\vt}(\V,\V')$ is determined uniquely, modulo canonical \iso s, independently of the choice of $V_{1}$.
\end{proposition}

In view of the above we denote $H^{p,q}_{\bar\vt}(\V,\V')$ also by $H^{p,q}_{\bar\vt}(X,X\ssm S)$.
%%%
\begin{proposition}\label{proptriplerd} For a triple $(X,X',X'')$, there is a long exact sequence
\[
\cdots\lra H^{p,q-1}_{\bar\vt}(X',X'')\overset{\delta}\lra H^{p,q}_{\bar\vt}(X,X')\overset{j^{*}}\lra H^{p,q}_{\bar\vt}(X,X'')\overset{i^{*}}\lra H^{p,q}_{\bar\vt}(X',X'')\lra\cdots.
\]
\end{proposition}
%%%

\begin{proposition}[Excision] For any open set $V$ containing $S$, there is  a canonical \iso
\[
H^{p,q}_{\bar\vt}(X,X\ssm S)\simeq H^{p,q}_{\bar\vt}(V,V\ssm S).
\]\label{propexcision}
\end{proposition}

The relative cohomology share other fundamental properties with the local cohomology (relative cohomology) of $X$ with coefficients
in $\scO^{(p)}$. In fact we have (cf. \cite{Su10}, \cite{Su11})\,:

\begin{theorem}[Relative Dolbeault theorem]\label{thDrel}
There is a canonical \iso
\[
H^{p,q}_{\bar\vt}(X,X\ssm S)\simeq H^{q}_{S}(X;\scO^{(p)}).
\]
\end{theorem}

The excision in Proposition \ref{propexcision} is compatible with the excision 
%(\ref{exciloc}) 
in Proposition \ref{propfundloc}.\,2 for $\scS=\scO^{(p)}$ via the above \iso.

\paragraph{Differential\,:} 

First, the map $\partial:\scE^{(p,q)}(X)\ra \scE^{(p+1,q)}(X)$ given by $\o\mapsto (-1)^{q}\,\partial\o$ induces
$\partial:H_{\bp}^{p,q}(X)\ra H_{\bp}^{p+1,q}(X)$ and it is compatible with $d:H^{q}(X;\scO^{(p)})\ra H^{q}(X;\scO^{(p+1)})$ via the canonical \iso\ (\ref{dth}).

In the case  $\V=\{V_{0},V_{1}\}$, $\partial :H^{p,q}_{\bar\vt}(\V)\lra H^{p+1,q}_{\bar\vt}(\V)$ is induced by
\[
(\xi_{0},\xi_{1},\xi_{01})\mapsto (-1)^{q}\,(\partial\xi_{0},\partial\xi_{1},-\partial\xi_{01}).
\]

In the case $V_{0}=X'$
%\ssm S$ and $V_{1}$ is a \nbd\ of $S$,
we have the differential
\begin{equationth}\label{differential}
\partial:H^{p,q}_{\bar\vt}(X,X')\lra H^{p+1,q}_{\bar\vt}(X,X')\quad\text{induced by}\ \ (\xi_{1},\xi_{01})\mapsto
(-1)^{q}\,(\partial\xi_{1},-\partial\xi_{01}). 
\end{equationth}
We have the following commutative diagram\,{\rm :}
\[
\SelectTips{cm}{}
\xymatrix
@C=.7cm
@R=.5cm
{H^{p,q}_{\bar\vt}(X,X')\ar[r]^-{\partial} \ar@{-}[d]^-{\wr}
%\ar@{=}[d]^{\wr}
%\ar[d]^-{\wr}
&H^{p+1,q}_{\bar\vt}(X,X')\ar@{-}[d]^-{\wr}
%\ar[d]^-{\wr}
\\
 H^{q}(X,X';\scO^{(p)})\ar[r]^-{d} & H^{q}(X,X';\scO^{(p+1)}),}
\]
where the vertical \iso s are the ones in Theorem \ref{thDrel}.
%\end{proposition}

\begin{remark} A relative cohomology such as the relative de~Rham  or  relative Dolbeault cohomology  defined above may also be interpreted
as the cohomology of a complex dual to the mapping cone of a morphism of complexes in the theory of derived categories
%``co-mapping cone'' of a morphism of complexes 
and a theorem as Theorems~\ref{thdRrel}  or \ref{thDrel} may be proved
 from this viewpoint as well.  This way we also see that this kind of relative cohomology goes well with
 derived functors (cf. \cite{Su11}).
\end{remark}

%%%%%%
\subsection{Relative de~Rham and relative Dolbeault cohomologies}\label{ssdRD}

Let $X$ be a complex \mfd\ of dimension $n$. We consider the following two cases where 
there is a natural relation between the two cohomology theories (cf. \cite{Su10}).
%Then we have the following natural \homo s\,:
\vv

\noindent
{\bf (I)} Noting that, for any $(n,q)$-form $\o$, $\bp\o=d\o$,
there exist natural morphisms
\begin{equationth}\label{DdR}
H^{n,q}_{\bp}(X)\lra H^{n+q}_{d}(X)\quad\text{and}\quad H^{n,q}_{\bar\vt}(X,X\ssm S)\lra H^{n+q}_{D}(X,X\ssm S).
\end{equationth}

In particular, this is used to define the integration on the relative Dolbeault cohomology in the subsequent section.
\vv

\noindent
{\bf (II)} 
We define
$\rho^{q}:\scE^{(q)}\lra \scE^{(0,q)}$
by assigning to a $q$-form $\o$ its $(0,q)$-component $\o^{(0,q)}$. Then $\rho^{q+1}(d\o)=\bp(\rho^{q}\o)$
and we have a natural morphism of complexes
\[
\SelectTips{cm}{}
\xymatrix
%@C=.7cm
@R=.7cm
{ 0\ar[r]& \C\ar[r] \ar[d]^-{\iota}&\scE^{(0)}\ar[r]^-{d}\ar[d]^-{\rho^{0}} &\scE^{(1)}\ar[d]^-{\rho^{1}}\ar[r]^-{d}&\cdots\ar[r]^-{d}&\scE^{(q)}\ar[d]^-{\rho^{q}}\ar[r]^-{d}&\cdots\\
0\ar[r] & \scO \ar[r] & \scE^{(0,0)}\ar[r]^-{\bp}&\scE^{(0,1)}\ar[r]^-{\bp}&\cdots\ar[r]^-{\bp}&\scE^{(0,q)}\ar[r]^-{\bp}&\cdots.}
\]

 Thus there is a  natural morphism  
%\[
$\rho^{q}:H^{q}_{D}(X,X')\lra H^{0,q}_{\bar\vt}(X,X')$, which makes
%\]
%Moreover we have 
the following  diagram commutative\,{\rm :}
\begin{equationth}\label{dRtoD}
\SelectTips{cm}{}
\xymatrix
@C=.7cm
@R=.5cm
{ H^{q}_{D}(X,X')\ar[r]^-{\rho^{q}} \ar@{-}[d]^-{\wr}& H^{0,q}_{\bar\vt}(X,X')\ar@{-}[d]^-{\wr}\\
 H^{q}(X,X';\C)\ar[r]^-{\iota} &H^{q}(X,X';\scO) .}
\end{equationth}
Note that, if we take coverings $\V=\{V_{0},V_{1}\}$ and $\V'=\{V_{0}\}$ with $V_{0}=X'$ and $V_{1}$ a
\nbd\ of $X\ssm X'$, then $\rho^{q}:H^{q}_{D}(X,X')=H^{q}_{D}(\V,\V')\ra H^{0,q}_{\bar\vt}(X,X')=H^{0,q}_{\bar\vt}(\V,V')$ assigns to the class of $(\o_{0},\o_{01})$ the class of $(\o_{0}^{(0,q)},\o_{01}^{(0,q-1)})$.

Recalling that we have the analytic de~Rham complex
\[
0\lra\C\overset{\iota}\lra\scO\overset{d}\lra\scO^{(1)}\overset{d}\lra\cdots\overset{d}\lra\scO^{(n)}\lra 0,
\]
the above diagram is extended to an \iso\ of complexes
%\begin{equationth}\label{lcrd}
\[
\SelectTips{cm}{}
\xymatrix
@C=.5cm
@R=.6cm
{ 0\ar[r]& H^{q}_{D}(X,X')\ar[r]^-{\rho^{q}} \ar@{-}[d]^-{\wr}&H^{0,q}_{\bar\vt}(X,X')\ar[r]^-{\partial}\ar@{-}[d]^-{\wr} &H^{1,q}_{\bar\vt}(X,X')\ar@{-}[d]^-{\wr}\ar[r]^-{\partial}&\cdots\ar[r]^-{\partial}&H^{n,q}_{\bar\vt}(X,X')\ar[r]\ar@{-}[d]^-{\wr}&0\\
0\ar[r] & H^{q}(X,X';\C) \ar[r]^-{\iota} & H^{q}(X,X';\scO) \ar[r]^-{d}&H^{q}(X,X';\scO^{(1)})\ar[r]^-{d}&\cdots\ar[r]^-{d}&H^{q}(X,X';\scO^{(n)})\ar[r]&0.}
\]
%\end{equationth}

In the above situation we have\,:

\begin{theorem}\label{thexactpcd} If $H^{q}(X,X';\C)=0$ and $H^{q}(X,X';\scO^{(p)})=0$ for $p\ge 0$ and $q\ne q_{0}$, then the following sequence is exact\,{\rm :}
\[
0\ra H^{q_{0}}(X,X';\C)\overset{\iota}\ra H^{q_{0}}(X,X';\scO)\overset{d}\ra H^{q_{0}}(X,X';\scO^{(1)})\overset{d}\ra\cdots\overset{d}\ra H^{q_{0}}(X,X';\scO^{(n)})\ra 0.
\]
\end{theorem}

\lsection{Local duality morphism}\label{seclocald}

We recall the cup product and integration theory on  \v{C}ech-Dolbeault cohomology in the relevant case.
% of coverings with two open sets.
 Then we recall the local duality morphism.

Let $X$ be a complex \mfd\ of dimension $n$ and $\V=\{V_{0},V_{1}\}$ an open covering of $X$.

\paragraph{Cup product\,:} We define the \emph{cup product}
\begin{equationth}\label{cup}
\scE^{(p,q)}(\V)\times \scE^{(p',q')}(\V)\lra \scE^{(p+p',q+q')}(\V),\qquad (\xi,\eta)\mapsto \xi\smallsmile\eta
\end{equationth}
by
\[
\begin{aligned}
&(\xi\smallsmile\eta)_0=\xi_0\wedge\eta_0,\quad(\xi\smallsmile\eta)_1=\xi_1\wedge\eta_1\quad\text{and}\\
&(\xi\smallsmile\eta)_{01}=(-1)^{p+q}\xi_0\wedge\eta_{01}+\xi_{01}\wedge\eta_1.
\end{aligned}
\]
Then $\xi\smallsmile\eta$ is linear in $\xi$ and $\eta$ and we have
\[
\bar\vt(\xi\smallsmile\eta)=\bar\vt\xi\smallsmile\eta+(-1)^{p+q}\xi\smallsmile\bar\vt\eta.
\]
Thus it induces the cup product
%\begin{equationth}\label{cupCDtwo}
\[
H^{p,q}_{\bar\vt}(\V)\times H^{p',q'}_{\bar\vt}(\V)\lra
H^{p+p',q+q'}_{\bar\vt}(\V)
\]
%\end{equationth}
compatible, via the \iso\ of Theorem \ref{thDCD}, with the  product in the
Dolbeault cohomology induced from the exterior product of forms.

Let $S$ be a closed set in $X$. Letting $V_{0}=X\ssm S$ and $V_{1}$ a \nbd\ of $S$, we consider the
coverings $\V=\{V_{0},V_{1}\}$ and $\V'=\{V_{0}\}$. Then  (\ref{cup}) induces a pairing
%\begin{equationth}\label{cuprel}
\[
\scE^{(p,q)}(\V,\V')\times \scE^{(p',q')}(V_{1})\lra \scE^{(p+p',q+q')}(\V,\V'),
%\quad ((\xi_{1},\xi_{01}),\eta_{1})\mapsto (\xi_1\wedge\eta_1,\xi_{01}\wedge\eta_1)
\]
%\end{equationth}%
assigning to $\xi=(\xi_{1},\xi_{01})$  and $\eta_{1}$ the cochain
$(\xi_1\wedge\eta_1,\xi_{01}\wedge\eta_1)$.  It induces a pairing
\begin{equationth}\label{cuprel}
H^{p,q}_{\bar\vt}(X,X\ssm S)\times H^{p',q'}_{\bp}(V_{1})\lra H^{p+p',q+q'}_{\bar\vt}(X,X\ssm S).
\end{equationth}

\paragraph{Integration\,:} %For simplicity we assume that $X$ is connected. 
First we make the following
\begin{remark}\label{remoricomplex}
 Since $X$ is a complex \mfd, it is always orientable. However the orientation we consider is not necessarily the ``usual
one''. Here we say an orientation of $X$ is usual if $(x_{1},y_{1},\dots,x_{n},y_{n})$ is a positive coordinate
system when  $(z_{1},\dots,z_{n})$, $z_{i}=x_{i}+\sqrt{-1}y_{i}$, is a complex coordinate system on $X$.
%(cf. Subsection \ref{ssinttopics}, in particular Remark~\ref{remthom} below). 
%In the sequel the complex \mfd s we consider are complexifications of real analytic \mfd s. Thus let $M$ be a
%real analytic \mfd\ and $X$ its complexification.
\end{remark}
%As a real \mfd\ $X$ is orientable, however the orientation we consider may not
%be the usual one (cf. Remark \ref{remfund}.\,1).

First we assume that $X$ is compact. If $X$ is  oriented, from (\ref{DdR}) and (\ref{intdR}), we have the integration
%\begin{equationth}\label{intD}
\[
%\int_{X}:
H^{n,n}_{\bp}(X)\xrightarrow{\,\,\,\int_{X}\,\,\,}\C.
\]
%\end{equationth}
In the case we do not specify the orientation, we
define
\[
%\int_{X}:
H^{n,n}_{\bp}(X)\otimes_{{}_{\Z_{X}(X)}} or_{X}(X)\xrightarrow{\,\,\,\int_{X}\,\,\,}\C.
\]
as follows. For simplicity we assume that $X$ is connected.  It suffices to define it for a decomposable element $[\o]\otimes a$. Once we fix an orientation,
we have a canonical \iso\ $or_{X}(X)\simeq\Z$ so that $a$ determines an integer $n(a)$. On the other
hand we have a well-defined integral $\int_{X}[\o]$. We set
\[
\int_{X}[\o]\otimes a=n(a)\int_{X}[\o].
\]
If we take the opposite orientation for $X$, the above remains the same. 

Suppose $K$ is a
compact  set in $X$ ($X$ may not be compact). Letting $V_{0}=X\ssm K$ and $V_{1}$ a \nbd\ of $K$, we consider the covering 
$\V_{K}=\{V_{0},V_{1}\}$. Let $R_{1}$ and $R_{01}$ be as in Subsection \ref{ssint}.
If $X$ is oriented, we 
%may assume that $R_1$ is compact and we 
have the integration (cf. (\ref{DdR}) and (\ref{intreldR}))
\begin{equationth}\label{intDrel}
%\int_X:
H^{n,n}_{\bar\vt}(X,X\ssm K)\xrightarrow{\,\,\,\int_{X}\,\,\,} \C\qquad\text{given by}\ \ \int_X[\xi]=\int_{R_1}\xi_1+\int_{R_{01}}\xi_{01}.
\end{equationth}

In the case we do not specify the orientation, we may define, as before, the integration
%the integration as before\,:
\begin{equationth}\label{intDrelgen}
%\int_X: 
H^{n,n}_{\bar\vt}(X,X\ssm K)\otimes or_{X}(X)\xrightarrow{\,\,\,\int_{X}\,\,\,}\C.
\end{equationth}
%as before.

\paragraph{Duality morphisms\,:} We assume that $X$ is oriented for simplicity.
%We assume that $X$ is compact and oriented.
If $X$ is compact, 
then the bilinear pairing
\[
 H^{p,q}_{\bp}(X)\times H^{n-p,n-q}_{\bp}(X)\overset{\wedge}\lra H^{n,n}_{\bp}(X)
 \overset{\int_{X}}\lra \C
\]
induces the Kodaira-Serre duality
%\begin{equationth}\label{3.4}
\[
KS_{X}:H^{p,q}_{\bp}(X)\overset{\sim}{\lra}
 H^{n-p,n-q}_{\bp}(X)^*.
 \]
%\end{equationth}
 
Let $K$ be a  compact set in $X$ ($X$ may not be compact). The cup product (\ref{cuprel})
followed by the integration (\ref{intDrel}) gives a bilinear pairing
\[
H^{p,q}_{\bar\vt}(X,X\ssm K)\times H^{n-p,n-q}_{\bp}(V_1)\overset{\smallsmile}\lra H^{p,q}_{\bar\vt}(X,X\ssm K)\overset{\int_{X}}\lra\C.
\]
This induces a morphism
\begin{equationth}\label{3.5}
\bar A_{X,K}:H^{p,q}_{\bar\vt}(X,X\ssm K)\lra H^{n-p,n-q}_{\bp}[K]^*
=\lim_{\underset{V_{1}\supset K}{\lra}}H^{n-p,n-q}_{\bp}(V_1)^*,
\end{equationth}
which we call the {\em $\bar\partial$-Alexander morphism}. 
%In the above 
Here we consider  algebraic duals, however
in order to have the duality, we need to take topological duals (cf. Theorem \ref{thMH} below).

If $X$ is compact, we have the following commutative diagram\,:
\[
\SelectTips{cm}{}
\xymatrix
%@C=.3cm
{H^{p.q}_{\bar\vt}(X,X\ssm K)\ar[r]^-{j^*}\ar[d]^{\bar A_{X,K}} &H^{p.q}_{\bp}(X)\ar[d]^{KS_{X}}_-{\wr}\\
H^{n-p,n-q}_{\bp}[K]^{*} \ar[r]^-{j_*}&\ H^{n-p,n-q}_{\bp}(X)^{*}.}
\]

\lsection{Hyperforms via relative Dolbeault cohomology}{\label{sec:hyperforms}}

In this section we let $M$ denote a real analytic \mfd\ of dimension $n$ and $X$ its complexification. We assume  $M$ to be orientable
so that $or_{M}$ is trivial, i.e., a constant sheaf.  Thus $or_{M/X}$ is also trivial and, for any open set $U$ in $M$, the space of $p$-hyperforms is given by (cf. (\ref{isosec}))
\[
\scB^{(p)}_{M}(U)=H^{n}_{U}(V;\scO^{(p)}_{X})\otimes_{{}_{\Z_{M}(U)}} or_{M/X}(U),
\]
where $V$ is  a complex \nbd\ of $U$
in $X$. 
%In the sequel if the ring is omitted in the tensor product notation, it is meant to be the tensor product
%over $\Z_{M}(U)$.

Note that, in the above situation, there is an \iso\ $or_{M/X}\simeq\Z_{M}$, however there are various choices 
of the \iso. Once we fix an orientation of $T_{M}X$, the \iso\ is determined uniquely.
%the orientations of $X$ and $M$ the \iso\ is determined uniquely according to the convention in Subsection \ref{subsOT}.

%%%%
\subsection{Expressions of hyperforms and some basic operations}\label{ssexp}

%On the other hand, b
In the above situation
%(see also (\ref{Drel})), 
there is a canonical \iso\ $H^{n}_{U}(V,\scO^{(p)})\simeq H^{p,n}_{\bar\vt}(V,V\ssm U)$ 
(cf.~Theorem~\ref{thDrel})
so that there
is a canonical \iso
%\begin{equationth}\label{exphyp}
\[
\scB^{(p)}(U)\simeq H^{p,n}_{\bar\vt}(V,V\ssm U)\otimes_{{}_{\Z_{M}(U)}} or_{M/X}(U).
\]
%\end{equationth}

In the sequel we give explicit expressions of the classes in 
%$\scB^{(p)}(U)$ will denote 
%the relative Dolbeault cohomology 
$H^{p,n}_{\bar\vt}(V,V\ssm U)$ and some of the basic
operations on them. In fact, in the case the orientation of $T_{M}X$ is specified,
%$X$ and $M$ are oriented 
there is a canonical \iso
\[
\scB^{(p)}(U)\simeq H^{p,n}_{\bar\vt}(V,V\ssm U)
\]
and these may be thought of as giving descriptions for the hyperforms themselves.

%%%
\paragraph{One-dimensional case\,:} Before we proceed further, we review the original expression of hyperfunctions in one-dimensional case by Sato, with some fundamental examples.
In fact we will see below that our expression is in a sense a natural generalization of this.
In particular, the integration we discuss in Subsection~\ref{ssinttopics} may be thought of as a direct generalization of that in one-dimensional case.

Let $U$ be an open set in $\R=\{(x)\}$ and 
$V$ a complex \nbd\ of $U$ in $\C=\{(z)\}$, $z=x+\sqrt{-1}\,y$. Here we fix the orientation of $\C$ so that $(y,x)$ is a positive coordinate system. The space of hyperfunctions on $U$ is originally defined by
\[
\scB(U)=\scO(V\ssm U)/\scO(V).
\]
Thus a hyperfunction is represented by a \h\ \fcn\ $F$ on $V\ssm U$. In our framework, it is 
represented by the pair $(0,-F)$ (cf.~\eqref{hyperfone} below).

For example, the constant funtion $1$, as a hyperfunction, is represented by such functions as
\[
\psi=\begin{cases} \ \ \frac 1 2\\ -\frac 1 2,
\end{cases}                                       \ \psi_{+}=\begin{cases} 1\\ 0,
\end{cases}                                        \ \psi_{-}=\begin{cases} \ 0\\ -1
\end{cases}\ \text{and so forth},
\]
where the value in the upper column is the value on $V_{+}=\{\,z\in V\mid y>0\,\}$ and that 
in the lower column the value on $V_{-}=\{\,z\in V\mid y<0\,\}$. We will see later that $\psi$ 
is 
%in a sense 
a natural representation, while $\psi_{\pm}$ is a representation that has support in $V_{\pm}$. 
%In our framework, they are represented by the pairs
They correspond to the representations 
$(0,-\psi)$ and $(0,-\psi_{\pm})$, respectively,
in our framework (cf.~Example~\ref{exone} and Remark~\ref{remdimone} below). 

Also the $\delta$-function is represented by the \fcn\ 
\[
-\frac 1 {2\pi\sqrt{-1}}\frac 1 z=-\frac 1 {2\pi\sqrt{-1}}\frac {\psi_{+}-\psi_{-}} z,
\]
which is expressed as 
\[
-\frac 1 {2\pi\sqrt{-1}}\Big (\frac 1 {x+\sqrt{-1}\, 0}-\frac 1 {x-\sqrt{-1}\, 0}\Big)
\]
to emphasize that it is the difference of the boudary values of \h\ functions.
In our framework, the $\delta$-function is represented by the pair (cf.~Definition~\ref{defdeltafcn} and Example~\ref{exDirac} below)
\[
\Big(0,\,\frac 1 {2\pi\sqrt{-1}}\frac 1 z\Big).
\]

Note that a direct proof of the fact that $\scB(U)$ is independent of the choice of
$V$ requires Runge's theorem. 
In our framework, it follows from the excision property of the relative Dolbeault cohomology (cf.~Proposition~\ref{propexcision}).

\paragraph{Expression of hyperforms\,:} Coming back to the general situation, let $U$ be an open set in $M$ and $V$ a complex \nbd\ of $U$ in $X$, as above.
Letting $V_{0}=V\ssm U$ and $V_{1}$ a \nbd\ of $U$ in $V$ (it could be $V$ itself), we consider the  coverings $\V=\{V_{0},V_{1}\}$ 
and $\V'=\{V_{0}\}$ of $V$ and $V\ssm U$. 
Then $H^{p,n}_{\bar\vt}(V,V\ssm U)=H^{p,n}_{\bar\vt}(\V,\V')$ and a class in $H^{p,n}_{\bar\vt}(V,V\ssm U)$
%$p$-hyperform
 is represented by a cocycle $(\tau_{1},\tau_{01})$ with $\tau_{1}$ a 
%$\bar\partial$-closed 
$(p,n)$-form on $V_{1}$, which is automatically $\bp$-closed, and  $\tau_{01}$ a $(p,n-1)$-form on $V_{01}$ \st\ $\tau_{1}=\bar\partial\tau_{01}$ on $V_{01}$. 
We have the exact sequence (cf. \eqref{lexactrelD})
\[
H_{\bp}^{p,n-1}(V)\lra H^{p,n-1}_{\bp}(V\ssm U)\overset{\delta}\lra H^{p,n}_{\bar\vt}(V,V\ssm U)\overset{j^{*}}\lra H^{p,n}_{\bp}(V),
\]
where $\delta$ assigns to the class of $\t$ the class of $(0,-\t)$. 

Here we quote the following (cf. \cite{Gr})\,:

\begin{theorem}[Grauert] Every  real analytic \mfd\ admits a fundamental system of  \nbd s consisting of Stein open sets in its complexification.\label{thGrauert}
\end{theorem}

By the above theorem, we may further simplify the expression. Namely,
if we take as $V$ a Stein \nbd, we have $H^{p,n}_{\bp}(V)\simeq H^{n}(V;\scO^{(p)})=0$. Thus $\delta$ is
surjective and every element in $H^{p,n}_{\bar\vt}(V,V\ssm U)$ is represented by a cocycle of the form $(0,-\t)$ with $\t$ a $\bar\partial$-closed
$(p,n-1)$-form on $V\ssm U$.

In the case $n>1$, $H^{p,n-1}_{\bp}(V)\simeq H^{n-1}(V;\scO^{(p)})=0$ and $\delta$ is an \iso\,:
\[
H^{p,n-1}_{\bp}(V\ssm U)\simeq H^{p,n}_{\bar\vt}(V,V\ssm U),\qquad [\t]\leftrightarrow [(0,-\t)].
\]

In the case $n=1$,
we have the exact sequence
\[
H_{\bp}^{p,0}(V)\lra H^{p,0}_{\bp}(V\ssm U)\overset{\delta}\lra H^{p,1}_{\bar\vt}(V,V\ssm U)\lra 0,
\]
where $H^{p,0}_{\bp}(V\ssm U)= H^{0}(V\ssm U;\scO^{(p)})$ and 
$H^{p,0}_{\bp}(V)= H^{0}(V;\scO^{(p)})$. In particular, for $p=0$,  we have the \iso
\begin{equationth}\label{hyperfone}
H^{0}(V\ssm U;\scO)/H^{0}(V;\scO)\simeq H^{0,1}_{\bar\vt}(V,V\ssm U), \qquad [F]\leftrightarrow [(0,-F)],
\end{equationth}
where $F$ is a \h\ \fcn\ on $V\ssm U$. 
%Note that t
The left hand side is
%recover 
the original expression discussed above, while
%of hyperfunctions by Sato in the one-dimensional case. 
the right hand side is the expression in terms of relative Dolbeault cohomology.

\begin{remark} Although a hyperform may be represented by a single differential form in most of the cases, it is important to keep in mind
that it is represented by a pair $(\tau_{1},\tau_{01})$ of forms in general.
\end{remark}

\paragraph{Multiplication by real analytic \fcn s\,:} Let $\scA_{M}$ denote the sheaf of real analytic \fcn s on $M$,
which is given by $\scA_{M}=i^{-1}\scO_{X}$ with $i:M\hra X$ the inclusion. We define
the multiplication
\[
\scA(U)\times H^{p,n}_{\bar\vt}(V,V\ssm U)\lra H^{p,n}_{\bar\vt}(V,V\ssm U)
\]
by assigning to $(f,[\tau])$ the class of $(\tilde f\tau_{1},\tilde f\tau_{01})$ with $\tilde f$ a \h\ extension
of $f$. Then the following diagram is commutative\,:
\[
\SelectTips{cm}{}
\xymatrix
@C=.7cm
@R=.5cm
{\scA(U)\times H^{p,n}_{\bar\vt}(V,V\ssm U)\ar[r]^-{} \ar@{-}[d]^-{\wr} &H^{p,n}_{\bar\vt}(V,V\ssm U) \ar@{-}[d]^-{\wr}\\
\scA(U)\times H^{n}_{U}(V;\scO^{(p)}) \ar[r]^-{}& H^{n}_{U}(V;\scO^{(p)}).}
\]

\paragraph{Partial derivatives\,:} Suppose that $U$ is a coordinate \nbd\ with coordinates $(x_{1},\dots,x_{n})$. We  define the partial derivative
\[
\frac \partial {\partial x_{i}}:H^{0,n}_{\bar\vt}(V,V\ssm U)\lra H^{0,n}_{\bar\vt}(V,V\ssm U)
\]
as follows. Let $(\tau_{1},\tau_{01})$ represent a hyper\fcn\ on $U$. We write $\tau_{1}=f\,d\bar z_{1}\wedge\cdots
\wedge d\bar z_{n}$
%\overline{\vPhi(z)}$  
and 
$\tau_{01}=\sum_{j=1}^{n}g_{j}\,d\bar z_{1}\wedge\cdots\wedge\widehat{d\bar z_{j}}\wedge\cdots
\wedge d\bar z_{n}$.
%\overline{\vPhi_{j}(z)}$ as before. 
Then $\frac \partial {\partial x_{i}}[\tau]$
is represented by the cocycle
\[
\Bigl(\frac{\partial f}{\partial z_{i}}\,d\bar z_{1}\wedge\cdots
\wedge d\bar z_{n},\,\sum_{j=1}^{n}\frac{\partial g_{j}}{\partial z_{i}}\,d\bar z_{1}\wedge\cdots\wedge\widehat{d\bar z_{j}}\wedge\cdots
\wedge d\bar z_{n}\Bigr).
\]
With this the following diagram is commutative\,:
\[
\SelectTips{cm}{}
\xymatrix
@C=.7cm
@R=.5cm
{H^{0,n}_{\bar\vt}(V,V\ssm U)\ar[r]^-{\frac \partial {\partial x_{i}}} \ar@{-}[d]^-{\wr} &H^{0,n}_{\bar\vt}(V,V\ssm U) \ar@{-}[d]^-{\wr}\\
 H^{n}_{U}(V;\scO) \ar[r]^-{\frac \partial {\partial z_{i}}}& H^{n}_{U}(V;\scO).}
\]

Thus for a differential operator $P(x,D)$, $P(x,D):H^{0,n}_{\bar\vt}(V,V\ssm U)\ra H^{0,n}_{\bar\vt}(V,V\ssm U)$ is well-defined.

\paragraph{Differential\,:} We  define the differential (cf. \eqref{differential}, here we denote $\partial$ by $d$)
\begin{equationth}\label{defdiff}
d:H^{p,n}_{\bar\vt}(V,V\ssm U)\lra H^{p+1,n}_{\bar\vt}(V,V\ssm U).
\end{equationth}
by assigning to the class of $(\tau_{1},\tau_{01})$ the class of $(-1)^{n}(\partial\tau_{1},-\partial\tau_{01})$.  Then the following diagram is commutative\,:
\[
\SelectTips{cm}{}
\xymatrix
@C=.7cm
@R=.5cm
{H^{p,n}_{\bar\vt}(V,V\ssm U)\ar[r]^-{d} \ar@{-}[d]^-{\wr} &H^{p+1,n}_{\bar\vt}(V,V\ssm U) \ar@{-}[d]^-{\wr}\\
 H^{n}_{U}(V;\scO^{(p)}) \ar[r]^-{d}& H^{n}_{U}(V;\scO^{(p+1)}).}
\]

The above operations are readily carried over to those for the hyperforms $\scB^{(p)}(U)=H^{p,n}_{\bar\vt}(V,V\ssm U)\otimes or_{M/X}(U)$.

In particular, since we have a canonical \iso\ (cf. (\ref{cori}))
\[
H^{n}_{D}(V,V\ssm U)\otimes or_{M/X}(U)\simeq\C_{M}(U),
\]
we have, from Theorem \ref{thexactpcd}, an exact sequence of sheaves on $M$\,:
%\begin{equationth}\label{hypdR}
\[
0\lra\C\lra \scB\overset{d}\lra\scB^{(1)}\overset{d}\lra\cdots\overset{d}\lra\scB^{(n)}\lra 0.
\]
%\end{equationth}
We come back to this topic in Subsection \ref{ssembra} below.

\subsection{Integration and related topics}\label{ssinttopics}

\paragraph{Support of a hyperform\,:} 

%%%
Let $U$ be an open set in $M$ and $K$ a compact set in $U$.
We define the space $\scB^{(p)}_{K}(U)$ of $p$-hyperforms on $U$ with support in $K$ as the kernel
of the restriction $\scB^{(p)}(U)\ra \scB^{(p)}(U\ssm K)$.
\begin{proposition} 
%Suppose the restriction $or_{M/X}(U)\ra or_{M/X}(U\ssm K)$ is an \iso. Then 
For any open set $V$ in $X$ containing $K$, the cohomology $H^{p,n}_{\bar\vt}(V,V\ssm K)$ may be thought of
as a
$\Z_{M}(U)$-module and there is a canonical \iso\,{\rm :}
\[
\scB^{(p)}_{K}(U)\simeq H^{p,n}_{\bar\vt}(V,V\ssm K)\otimes_{{}_{\Z_{M}(U)}} or_{M/X}(U).
%\simeq \A^{n-p}_{\o}(K)^{*},
\]
\end{proposition}
\begin{proof} 
%Applying 
By Proposition \ref{proptriplerd} for the triple $(V,V\ssm K,V\ssm U)$, we have the exact sequence
\[
 H^{p,n-1}_{\bar\vt}(V\ssm K,V\ssm U)\overset{\delta}\lra H^{p,n}_{\bar\vt}(V,V\ssm K)\overset{j^{*}}\lra H^{p,n}_{\bar\vt}(V,V\ssm U)\overset{i^{*}}\lra H^{p,n}_{\bar\vt}(V\ssm K,V\ssm U).
\]
By Proposition \ref{propexcision}, we may assume that $V$ is a complex \nbd\ of $U$ and that each
connected component of $V$ contains at most one connected component of $U$. This shows that each of the cohomologies in the sequence has a natural $\Z_{M}(U)$-module \str. 
%The first part of the proposition follows from this. 
Since $or_{M/X}(U)\simeq\Z_{M}(U)$, taking the tensor product with  $or_{M/X}(U)$ over $\Z_{M}(U)$
is an exact functor. By definition, $H^{p,n}_{\bar\vt}(V,V\ssm U)\otimes or_{M/X}(U)=\scB^{(p)}(U)$. Noting that $V\ssm K$ is a complex \nbd\ of $U\ssm K$
and that $V\ssm U=(V\ssm K)\ssm (U\ssm K)$, we have $H^{p,n}_{\bar\vt}(V\ssm K,V\ssm U)\otimes or_{M/X}(U\ssm K)=\scB^{(p)}(U\ssm K)$, where the tensor product is over $\Z_{M}(U\ssm K)$. Since the restriction
\[
r^{*}:H^{p,n}_{\bar\vt}(V\ssm K,V\ssm U)\otimes _{{}_{\Z_{M}(U)}}or_{M/X}(U)\lra H^{p,n}_{\bar\vt}(V\ssm K,V\ssm U)\otimes_{{}_{\Z_{M}(U\ssm K)}} or_{M/X}(U\ssm K)
\]
is injective, $\op{Ker}i^{*}=\op{Ker}(r^{*}\circ i^{*})$.
On the other hand, since $U\ssm K$ is pure $n$-codimensional in $V\ssm K$ \wrt\ $\scO^{(p)}_{X}$, we have  
%(cf. Theorem \ref{thfund}.\,1 and Remark \ref{remfund}.\,2)
%\[
$H^{p,n-1}_{\bar\vt}(V\ssm K,V\ssm U)\simeq H^{n-1}_{U\ssm K}(V\ssm K;\scO^{(p)}_{X})=0$.
%\]
% (cf. Remark \ref{rempc}).
\end{proof}

\begin{remark} By the flabbiness of $\scB^{(p)}$, we have 
the following exact sequence\,:
\[
0\lra \scB^{(p)}_{K}(U)\lra \scB^{(p)}(U)\lra \scB^{(p)}(U\ssm K)\lra 0.
\]
\end{remark}

\paragraph{Orientation convention\,:} 
%Before we define integration of hyperforms, we adopt some convention on various orientations. 
Recall the sequence \eqref{exactemb} and the \iso s  \eqref{orisub} with $M$ and $X$ as above.
Note that, since we assumed that $M$ is orientable, all the orientation sheaves involved  are trivial.
Let $(z_{1},\dots,z_{n})$, $z_{i}=x_{i}+\sqrt{-1}y_{i}$,  be a complex coordinate system on $X$ \st\ $(x_{1},\dots,x_{n})$ is a
coordinate system on $M$. Thus $(\varpi(\frac \partial{\partial y_{1}}),\dots,\varpi(\frac \partial{\partial y_{n}}))$ is a frame of $T_{M}X$.
When we orient $X$, $M$ and $T_{M}X$, we adopt  the following\,:
%Let $\psi_{M}$, $\psi_{M/X}$ and $\psi_{X}$ be sections of $or_{M}$, $or_{M/X}$ and $or_{X}$, \r, that
%make $(x_{1},\dots,x_{n})$, $(y_{1},\dots,y_{n})$ and $(x_{1},y_{1},\dots,x_{n},y_{n})$ positive coordinate system.
%%%%%
%\paragraph{Convention 1.} Let $\psi_{X}$ and $\psi_{M}$ be prescribed orientations of $X$ and $M$, \r.
%Then we take the orientation $\psi_{M/X}$ of $T_{M}X$ so that, if  $(x_{1},y_{1},\dots,x_{n},y_{n})$
%and $(x_{1},\dots,x_{n})$ are positive coordinate systems on $X$ and $M$, \r, $(\varpi(\frac \partial{\partial y_{1}}),\dots,\varpi(\frac \partial{\partial y_{n}}))$ is a positive frame of $T_{M}X$.
\paragraph{Convention 1.} Let $\psi_{M}$ and $\psi_{M/X}$ be prescribed orientations of $M$ and $T_{M}X$,
\r.
Then we take the orientation $\psi_{X}$ of $X$ so that, if $(x_{1},\dots,x_{n})$ is a positive coordinate systems on  $M$ and if $(\varpi(\frac \partial{\partial y_{1}}),\dots,\varpi(\frac \partial{\partial y_{n}}))$ is a positive frame of $T_{M}X$, then $(y_{1},\dots,y_{n}, x_{1},\dots,x_{n})$
%$(y_{1},x_{1},\dots,y_{n},x_{n})$
is a positive coordinate systems on  $X$.
\vv

The above amounts to saying that we make identification
\begin{equationth}\label{oriident}
i^{-1}or_{X}= or_{M/X}\otimes or_{M}\qquad\text{by}\qquad i^{-1}\psi_{X}\longleftrightarrow\psi_{M/X}\otimes
\psi_{M}
\end{equationth}
with  $\psi_{M/X}$, $\psi_{M}$ and  $\psi_{X}$ as in Convention 1.

In the case $M$ is an open set in $\R^{n}=\{(x_{1},\dots,x_{n})\}$, we may think of $X$ as an open set in
$\C^{n}=\{(z_{1},\dots,z_{n})\}$. In this case, $T_{M}X$ is trivial and we make identification 
$T_{M}X=M\times\R^{n}_{y}$, $\R^{n}_{y}=\{(y_{1},\dots,y_{n})\}$, by 
$\sum_{i=1}^{n}\eta_{i}\varpi(\frac\partial{\partial y_{i}})\leftrightarrow (x;\eta_{1}(x),\dots,\eta_{n}(x))$, $x\in M$.

\begin{remark}\label{remthom} 
{\bf 1.} Let $\psi_{M}$, $\psi_{M/X}$ and $\psi_{X}$ be as in Convention 1. Then the difference between $\psi_{X}$  and the usual orientation of $X$ (cf. Remark~\ref{remoricomplex}) is a sign of
$(-1)^{\frac{n(n+1)}2}$.
%
%the usual orientation $\psi_{X}^{\rm u}$
%of $X$ is to take $(x_{1},y_{1},\dots,x_{n},y_{n})$ as a positive coordinate system (cf. Remark~\ref{remoricomplex}). The difference between $\psi_{X}$ above and $\psi_{X}^{\rm u}$ is a sign of
%$(-1)^{\frac{n(n+1)}2}$.
%$(-1)^{n}$.
\smallskip

\noindent
%There is another possible choice for the orientation of $X$, which we denote by $\psi_{X}^{\rm f}$. That is to take $(y_{1},\dots,y_{n},x_{1},\dots,x_{n})$ as a positive coordinate system. 
%This is 
%a natural choice 
{\bf 2.} The choice of $\psi_{X}$ as above is natural
from
the fiber bundle viewpoint in the following sense.
Note that $\psi_{M/X}$ above is the Thom class $\varPsi_{M}\in
H^{n}_{M}(X;\Z)$ of $M$ in $X$ (cf.~Subsection~\ref{subsOT}). Recall that the Alexander \iso
\[
A:H^{n}_{M}(X;\Z)=H^{n}(X,X\ssm M)\simeq H_{n}(M;\Z)
\]
is given by the left cap product with the fundamental class $[X]$. If we take $\psi_{X}$ as the orientation of $X$, 
%$\psi_{X}'$ which make $(y_{1},\dots,y_{n},x_{1},\dots,x_{n})$ a positive coordinate system, 
then
$A(\varPsi_{M})=[M]$, the fundamental class of $[M]$ with orientation $\psi_{M}$. 
\end{remark}

%We make another orientation convention when we define fiber integration in Subsection~\ref{subsection:integration} below.

\paragraph{Integration\,:} 
We consider the sheaf of real analytic densities on $M$\,:
\[
\scV_{M}=\scA^{(n)}_{M}\otimes_{{}_{\Z_{M}}} or_{M}.
\]
The sheaf of hyperdensities is defined by
\[
\scW_{M}=\scB_{M}\otimes_{{}_{\scA_{M}}}\scV_{M}=\scB^{(n)}_{M}\otimes_{{}_{\Z_{M}}} or_{M}.
\]
We define
\[
%\int_{M}:
\vG_{c}(M;\scW_{M})\xrightarrow{\,\,\,\int_{M}\,\,\,}\C
\]
as follows. We  assume $M$ (and $X$) to be connected for simplicity. 
%%%%%%

%We also assume $M$ (and $X$) to be connected for simplicity. 

For any compact set $K$ in $M$, 
we have, by the identification \eqref{oriident},
\[
\vG_{K}(M;\scW_{M})=H^{n,n}_{\bar\vt}(X,X\ssm K)\otimes or_{M/X}(M)\otimes or_{M}(M)
=H^{n,n}_{\bar\vt}(X,X\ssm K)\otimes i^{-1}or_{X}(M),
\]
where the tensor products are over $\Z_{M}(M)=\Z$.
Thus we have the integration
\[
%\int_{M}:
\vG_{K}(M;\scW_{M})\xrightarrow{\,\,\,\int_{M}\,\,\,}\C
\]
as defined in (\ref{intDrelgen}), which we recall for the sake of completeness.
Given 
\[
u \otimes a
\in 
H^{n,n}_{\bar\vt}(X,X\ssm K)\otimes i^{-1}{or}_{X}(M).
\]
Once we fix an orientation of $X$,
we have a well-defined integer $n(a)$ as we saw before (cf. Section \ref{seclocald}).
If we take the opposite orientation of $X$, the sign changes.
Letting $V_0=X\ssm K$ and $V_1= X$, consider the coverings $\V_{K}= \{V_0,V_1\}$ and 
$\V'_{K}= \{V_0\}$.  Then $u\in H^{n,n}_{\bar\vt}(X,X\ssm K)=H^{n,n}_{\bar\vt}(\V_{K},\V'_{K})$
is represented by
\[
\tau = (\tau_1,\tau_{01}) \in \scE^{(n,n)}(\V_{K},\V'_{K})=\scE^{(n,n)}(V_{1})\oplus\scE^{(n,n-1)}(V_{01}).
\]
%be a representative of $u \in \hdccppk{n}{n}{\mathcal{V}_K}{\mathcal{V}'_K}$.
Let $R_{1}$ be a $2n$-dimensional compact 
\mfd\ with $C^{\infty}$ boundary in $X$ containing $K$ in its interior and set $R_{01}=-\partial R_{1}$. 
Then, once we fix an orientation of $X$, we may define
\[
\int_{M}u=\int_{R_{1}}\tau_{1}+\int_{R_{01}}\tau_{01}.
\]
If we take the opposite orientation of $X$, the sign changes. Thus 
\[
\int_{M}u\otimes a=n(a)\int_{M}u
\]
 does
not depend on the choice of the orientation of $X$.

\begin{remark}\label{remintone} In the case $n=1$, the above definition is consistent with the original one
under the correspondence \eqref{hyperfone}.
% (cf. Remark \ref{remthom}.\,1).
\end{remark}
%\begin{remark}\label{reminthyp} Note that in the above, the identification $or_{M/X}\otimes or_{M}=i^{-1}or_{X}$ is done according to the convention described in Subsection \ref{subsOT}.
%\end{remark}

\paragraph{A theorem of Martineau\,:} The following theorem of A. Martineau \cite{M} (see also \cite{H1}, \cite{K}) may
%naturally 
be  interpreted in our framework as one of the  cases where the $\bp$-Alexander morphism 
(cf. \eqref{3.5})   is an \iso\ 
with topological duals so that the duality
pairing is given by the cup product followed by integration as described above.
See \cite{Su10} for a little more detailed discussions on this.
The essential point of the proof is that the Serre duality holds for $V\ssm K$, which is a consequence of a result of B. Malgrange \cite{Mal}. 

In the below we assume that $\C^{n}$ is oriented, however the orientation may not be the usual one.

\begin{theorem}\label{thMH} Let $K$ be a compact set in $\C^{n}$ \st\ $H^{p,q}_{\bp}[K]=0$ for $q\ge 1$.
Then for any open set $V\supset K$, $H^{p,q}_{\bar\vt}(V,V\ssm K)$ and $H^{n-p,n-q}_{\bp}[K]$ admits 
natural \str s of FS and DFS spaces, \r, and we have\,{\rm :}
\[
\bar A:H^{p,q}_{\bar\vt}(V,V\ssm K)\overset\sim\lra H^{n-p,n-q}_{\bp}[K]'=\begin{cases} 0\qquad & q\ne n\\
                                                                       \scO^{(n-p)}[K]'\qquad & q=n,
                                                                       \end{cases}
\]
where $'$ denotes the strong dual.
\end{theorem}

The theorem is originally stated  in terms of local cohomology for $p=0$.
In our framework the duality (in the case $q=n$) is described as follows.
Letting $V_{0}=V\ssm K$ and $V_{1}$ a \nbd\ of $K$ in $V$, we  consider the coverings $\V_{K}=\{V_{0},V_{1}\}$ and
$\V'_{K}=\{V_{0}\}$ of $V$ and $V_{0}$, as before.
The duality pairing is given, for a  cocycle $(\tau_{1},\tau_{01})$ in $\scE^{(p,n)}(\V_{K},\V'_{K})$ and a \h\ $(n-p)$-form $\eta$
near $K$, by
\begin{equationth}\label{locpair}
%\[
\int_{R_{1}}\tau_{1}\wedge\eta+\int_{R_{01}}\tau_{01}\wedge\eta,
%\]
\end{equationth}
where $R_{1}$ is a real $2n$-dimensional \mfd\ with $C^{\infty}$ boundary in $V_{1}$ containing $K$ in its interior and
$R_{01}=-\partial R_{1}$. We may let $\tau_{1}=0$ if $V$ is Stein.

Note that the hypothesis $H^{p,q}_{\bp}[K]=0$, $q\ge 1$, is satisfied
if $K$ is a subset of $\R^{n}$ by Theorem \ref{thGrauert}.

%%%%%
Suppose $K\subset\R^{n}$ and denote by $\scA^{(p)}$ the sheaf of real analytic $p$-forms on $\R^{n}$. Then
%we have
\[
\scO^{(p)}[K]= \lim_{\underset{V_{1}\supset K}{\lra}}\scO^{(p)}(V_{1})\simeq \lim_{\underset{U_{1}\supset K}{\lra}}\scA^{(p)}(U_{1})
=\scA^{(p)}[K],
\]
where $V_{1}$ runs through \nbd s of $K$ in $\C^{n}$ and $U_{1}=V_{1}\cap\R^{n}$.
In the sequel we set $X=\C^{n}$ and $M=\R^{n}$.
Then noting that $\scB^{(p)}_{K}(U)= H^{p,n}_{\bar\vt}(V,V\ssm K)\otimes or_{M/X}(U)$,
from Theorem~\ref{thMH} we have, without specifying the orientation of $\C^{n}$ or of $\R^{n}$,

\begin{corollary}\label{functional} For any open set $U\subset\R^{n}$ containing $K$, 
there is an \iso
\[
\scB^{(p)}_{K}(U)\simeq (\scA^{(n-p)}[K]\otimes or_{M}(U))'.
\]
\end{corollary}

Note that
the duality pairing above is given by (cf. \eqref{locpair})
\[
\scB^{(p)}_{K}(U)\times (\scA^{(n-p)}[K]\otimes or_{M}(U))\lra H^{n,n}_{\bar\vt}(V,V\ssm K)\otimes i^{-1}or_{X}(U)
\overset{\int}\lra\C.
\]
%is topologically non-degenerate so that

\paragraph{Delta \fcn\,:} We consider the case $K=\{0\}\subset\R^{n}$. We set
\[
\varPhi(z)= dz_{1}\wedge\cdots\wedge dz_{n}\quad\text{and}\quad \varPhi_i(z)=(-1)^{i-1}z_i\, dz_1\wedge\cdots\wedge\widehat{dz_i}\wedge\cdots\wedge dz_n.
\]
The $0$-Bochner-Martinelli form on $\C^{n}\ssm\{0\}$ is defined as
\[
\beta_n^{0}=C'_n\frac {\sum_{i=1}^n\overline{\varPhi_i(z)}}{\|z\|^{2n}},\qquad
%C'_n=(-1)^{\frac{n(n-1)}2}\,\frac{(n-1)!}{(2\pi\sqrt{-1})^n}.
C'_n=(-1)^{\frac {n(n-1)} 2}\,\frac{(n-1)!}{(2\pi\sqrt{-1})^n}
\]
so that 
\[
\b_{n}=\beta_n^{0}\wedge\varPhi(z)
\]
is the Bochner-Martinelli form on $\C^{n}\ssm\{0\}$. Note that 
%\begin{equationth}\label{BMone}
\[
\b_{1}^{0}=\frac 1{2\pi\sqrt{-1}}\frac 1 z\qquad\text{and}\qquad
\b_{1}=\frac 1{2\pi\sqrt{-1}}\frac {dz} z.
\]
%\end{equationth}

We denote by $\psi_{M/X}$  the section of $or_{M/X}$ that corresponds to $+1$ when we choose $(y_{1},\dots,y_{n})$ as
a positive coordinate system in the normal direction.

\begin{definition}\label{defdeltafcn} The {\em delta \fcn} $\delta(x)$ is the hyper\fcn\  in 
\[
\scB_{\{0\}}(\R^{n})=H^{0,n}_{\bar\vt}(\C^{n},\C^{n}\ssm\{0\})\otimes or_{M/X}(\R^{n})%\quad\text{represented by} (0,-(-1)^{\frac {n(n+1)}2}\b_{n}^{0})\otimes\psi_{M/X}
\]
%\quad\text{
that is represented by
%}\ \
\[
(0,-(-1)^{\frac{n(n+1)}2}\b_{n}^{0})\otimes\psi_{M/X}.
\]
\end{definition}

Recall the \iso\ in Corollary \ref{functional} in this case\,:
\[
\scB_{\{0\}}(\R^{n})\simeq (\scA^{(n)}_{0}\otimes or_{M}(\R^{n}))'.
\]
For a  representative $h(x)\varPhi(x)$ of a germ in $\scA^{(n)}_{0}$, $h(z)\varPhi(z)$ is its complex representative. 

Let $\psi_{M}$ denote the section of $or_{M}$ that corresponds to $+1$ when we choose $(x_{1},\dots,x_{n})$
as a positive coordinate system on $\R^{n}$. Thus, by Convention 1, $\psi_{M/X}\otimes \psi_{M}$ is identified with $i^{-1}\psi_{X}$, where $\psi_{X}$ is the section of $or_{X}$
that corresponds to $+1$ when we choose $(y_{1},\dots,y_{n},x_{1},\dots,x_{n})$
as a positive coordinate system on $\C^{n}$.
% (cf. Remark \ref{remfund}.\,1). 
%Note that the sign difference between this system and the
%usual system $(x_{1},y_{1},\dots,x_{n},y_{n})$ is $(-1)^{\frac{n(n+1)}2}$.
Let $R_{1}$ be 
%the 
a closed ball 
%of radius $\sqrt{n}\,\varepsilon$ 
around $0$ in $\C^{n}$ so that $R_{01}=
-\partial R_{1}=-S^{2n-1}$. Here $S^{2n-1}$ is a $(2n-1)$-sphere, whose
% of radius $\sqrt{n}\,\varepsilon$.
 orientation  may not be the usual one.

\begin{theorem}\label{deltadual} The delta \fcn\ $\delta(x)$ is the hyperfunction that assigns the value $h(0)$ to a representative $\o\otimes \psi_{M}$, $\o=h(x) \varPhi(x)$.
\end{theorem}
\begin{proof} By definition $\delta(x)$ assigns to  $\o\otimes \psi_{M}$ the value
\[
-(-1)^{\frac{n(n+1)}2}n(\psi_{M/X}\otimes \psi_{M})\int_{R_{01}}h(z)\b_{n}.
\]
%If we choose the orientation of $\C^{n}$ so that $(y_{1},\dots,y_{n},x_{1},\dots,x_{n})$ is a positive
%coordinate system, then $n(\tau_{M/X}\otimes \tau_{M})=1$. Thus in the ``standard'' coordinate
If we take the usual orientation for $\C^{n}$, $n(\psi_{M/X}\otimes \psi_{M})=(-1)^{\frac{n(n+1)}2}$ and
% so that the  usual system $(x_{1},y_{1},\dots,x_{n},y_{n})$ 
%is positive,  
the above is equal to
\[
-\int_{R_{01}}h(z)\b_{n}=\int_{S^{2n-1}}h(z)\b_{n}=h(0),
\]
where $S^{2n-1}$ is the sphere with the usual orientation.
\end{proof}

\paragraph{Delta form\,:} We again consider the case $K=\{0\}\subset\R^{n}$.

\begin{definition}\label{defdf} The {\em delta form} $\delta^{(n)}(x)$ is the $n$-hyperform in 
\[
\scB^{(n)}_{\{0\}}(\R^{n})=H^{n,n}_{\bar\vt}(\C^{n},\C^{n}\ssm
\{0\})\otimes or_{M/X}(\R^{n})
\]
that is
%\quad\text{
represented by
%}\ \
\[
(0,-(-1)^{\frac{n(n+1)}2}\b_{n})\otimes\psi_{M/X}.
\]
%where $\psi_{M/X}$ is as in Definition \ref{defdeltafcn}.
\end{definition}

Recall the \iso\ in Corollary \ref{functional} in this case\,:
\[
\scB^{(n)}_{\{0\}}(\R^{n})\simeq (\scA_{0}\otimes or_{M}(\R^{n}))'.
\]
For a  representative $h(x)$ of a germ in $\scA_{0}$, $h(z)$ is its complex representative.

\begin{theorem} The delta form $\delta^{(n)}(x)$ is the $n$-hyperform that assigns  the value $h(0)$ to a representative $h(x)\otimes \psi_{M}$.
\end{theorem}

Let us compare the above description with the traditional way of expressing the delta \fcn. The difference becomes apparent in the case $n>1$ and we consider this case. We also choose, for simplicity,  the orientation of $\C^{n}$ so that the usual system is positive. Set $W_{i}=\{z_{i}\ne 0\}$, $i=1,\dots,n$, and consider the coverings $\W=\{\C^{n},W_{i}\}$ and $\W'=\{W_{i}\}$ of $\C^{n}$ and 
  $\C^{n}\ssm\{0\}$, which are Stein coverings. Then there is a canonical \iso\ $H^{n}_{\{0\}}(\C^{n};\scO)\simeq
H^{n}(\W,\W';\scO)$, the relative \v{C}ech cohomology. In our case, the long exact sequence for 
\v{C}ech
cohomology yields  the \iso\ $H^{n}(\W,\W';\scO) \simeq H^{n-1}(\W';\scO)$.
On the other hand, we have the canonical Dolbeault \iso
\[
%\B_{\{0\}}(\R^{n})\simeq 
H^{0,n-1}_{\bp}(\C^{n}\ssm\{0\})\simeq H^{n-1}(\W';\scO),%\simeq H^{n}(\W,\W';\O).
\]
under which the class of $\b_{n}^{0}$
corresponds to the class of 
\[
(-1)^{\frac {n(n-1)}2}\o_{n}^{0},\qquad\o_{n}^{0}=\Big(\frac 1 {2\pi\sqrt{-1}}\Big)^{n}\frac 1 {z_{1}\cdots z_{n}}
\]
(cf. \cite{Su10}, see also \cite{H2} where this correspondence is first studied in general, with a different sign convention).
The class corresponding to $[\o_{n}^{0}]$ in $H^{n}(\W,\W';\scO)$ is the traditional delta \fcn.
We set
\[
\o_{n}=\o_{n}^{0}\,\varPhi(z).
\]
Let 
$\vG$ be the $n$-cycle $\bigcap_{i=1}^{n}\{|z_{i}|=\varepsilon\}$ oriented
so that $\op{arg}(z_{1})\wedge\cdots\wedge\op{arg}(z_{n})$ is positive. Let $S^{2n-1}$ denote the sphere of radius $\sqrt{n}\,\varepsilon$ with usual orientation. Then, for a \h\ \fcn\
$h$ on a \nbd\ of $0$, we have
\[
\int_{S^{2n-1}}h(z)\b_{n}=\int_{\vG}h(z)\o_{n}%(=h(0)),
\]
%which is the traditional way of expressing the $\delta$-\fcn.
and either way we have the value $h(0)$. Note that the right hand side is a special case of Grothendieck
residues.
See, e.g., \cite{TNN} for applications of the above residue pairing from the computational aspect.

\lsection{Further operations}{\label{sec:morphisms}}

%Throughout the section, all the manifolds are assumed to be countable at infinity. 
In this section, we let $M$ be a real analytic manifold of dimension $n>0$ and $X$ its complexification.
We  assume, for simplicity, $M$ to be orientable.
% in what follows.

\subsection{Embedding of real analytic \fcn s}\label{ssembra}

The embedding of the sheaf $\scA_{M}$ of real analytic \fcn s on $M$ into the sheaf $\scB_{M}$ of hyper\fcn s or
more generally the embedding of the sheaf of real analytic forms $\scA^{(p)}_{M}$ into the sheaf of hyperforms $\scB^{(p)}_{M}$ is
determined  by the canonical identification of the constant \fcn\ $1$
%\fcn\ constantly equal to $1$ 
as a hyperfunction.

From the canonical identification $\Z_{M}=or_{M/X}\otimes or_{M/X}$ and the canonical morphism
 $or_{M/X}=\scH_{M}^{n}(\Z_{X})\ra \scH_{M}^{n}(\scO_{X})$, we have the canonical morphism\,:
\[
\Z_{M}=or_{M/X}\otimes or_{M/X}\lra\scB_{M}=\scH_{M}^{n}(\scO_{X})\otimes or_{M/X}.
\]
The image of  $1$ by this morphism is the corresponding hyper\fcn. We try to find it explicitly in our framework. For this we consider the complexification $or_{M/X}^{c}=\scH_{M}^{n}(\C_{X})$ of $or_{M/X}$. Then the above morphism is extended as  
\begin{equationth}\label{embC}
\C_{M}=or^{c}_{M/X}\otimes or_{M/X}\lra\scB_{M}.
\end{equationth}
Note that this is injective (cf. Theorems \ref{thfund}.\,1 and \ref{thexactpcd}). We  express the above morphism in terms of
%In the following we investigate the morphism $\scH_{M}^{n}(\C_{X})\ra \scH_{M}^{n}(\scO_{X})$ using the 
relative de~Rham and Dolbeault cohomologies.

As $M$ is assumed to be orientable, the sheaf  $or_{M/X}$ admits a global section which generates each stalk 
%of $or_{M/X}$ 
%over $\mathbb{Z}$ 
and for any of such  sections
$\psi\in or_{M/X}(M)=H_{M}^{n}(X;\Z)$, we have 
\begin{equationth}\label{tensgen}
\psi\otimes\psi=1.
\end{equationth}
We fix such a section and denote it by $\mathds{1}$ hereafter. 
%Note that, if $M$ is connected, we have two possibilities of the choice of a generator, i.e., either $\mathds{1}$ or  $- \mathds{1}$. 
Note  that it is what is referred to as the Thom class $\varPsi_{M}$ of $M$ in $X$  in Subsection \ref{subsOT}.
The image of $\mathds{1}$ by the canonical morphism $or_{M/X}(M)=H_{M}^{n}(X;\Z)\ra or^{c}_{M/X}(M)=H_{M}^{n}(X;\C)$, which is
% clearly 
injective, is  denoted by $\mathds{1}^{c}$.
%%%

Let $U$ be an open set in $M$ and $V$ a complex \nbd\ of $U$ in $X$. Then from \eqref{embC}, we have the following
commutative diagram  (cf. \eqref{cori} and \eqref{dRtoD})\,:
 \[
\SelectTips{cm}{}
\xymatrix
@C=.7cm
@R=.5cm
{\C_{M}(U)= H^{n}_{U}(V;\C)\otimes or_{M/X}(U)\ar[r]^-{\iota\otimes 1}\ar@{-}[d]^-{\wr}& 
H^{n}_{U}(V;\scO)\otimes or_{M/X}(U)=\scB(U)\ar@{-}[d]^-{\wr}\\
H^{n}_{D}(V,V\ssm U)\otimes or_{M/X}(U)\ar[r]^-{\rho^{n}\otimes 1}& H^{0,n}_{\bar\vt}(V,V\ssm U)\otimes or_{M/X}(U).}
\]

Setting $V_{0}=V\ssm U$ and $V_{1}=V$, consider the  coverings $\V=\{V_{0},V_{1}\}$ and $\V'=\{V_{0}\}$. 
From the above considerations we have\,:

\begin{theorem}\label{oneashf}
If $(\nu_{1},\nu_{01})$ is a representative of $\mathds{1}^{c}$ in $H^{n}_{D}(V,V\ssm U)=H^{n}_{D}(\V,\V')$, 
then the constant \fcn\ $1$ is identified with the hyper\fcn\ represented by $(\nu_{1}^{(0,n)},\nu_{01}^{(0,n-1)})\otimes \mathds{1}$ in $\scB(U)\simeq H^{(0,n)}_{\bar\vt}(\V,\V')\otimes or_{M/X}(U)$.
\end{theorem}
%%%

Note that the above identification of $1$ does not depend on the choice of
 $(\nu_{1},\nu_{01})$, as $\rho^{n}:H^{n}_{D}(\V,\V')\ra H^{0,n}_{\bar\vt}(\V,\V')$ is a well-defined morphism.
It does not depend on the choice of $\mathds{1}$ either by \eqref{tensgen}. 

Let $\nu=(\nu_{1},\nu_{01})$ be as in the above theorem.
% representative of 
% $\mathds{1}$ in the situation of Example \ref{exO=V}, for example we may take $(0,-\psi_{n})$ as $\nu$
% (cf. Example~\ref{exone}).
 % Let us take as $\mathds{1}$ the one   in Example~\ref{exone},
% which is represented by
 % $(0,-\psi_{n}^{(0,n-1)})$. 
 We define a morphism
\[
\scA^{(p)}(U)\lra\scB^{(p)}(U)=H^{p,n}_{\bar\vt}(V,V\ssm U)\otimes or_{M/X}(U)
\]
by assigning to $\o(x)\in \scA^{(p)}(U)$ the class $[(\nu_{1}^{(0,n)}\wedge\o(z),\nu_{01}^{(0,n-1)}\wedge\o(z))]\otimes\mathds{1}$, where $\o(z)$ denotes the complexification of $\o(x)$. 
Note that $(\nu_{1}^{(0,n)}\wedge\o,\nu_{01}^{(0,n-1)}\wedge\o)$ is a cocycle, as $\o$ is \h.
Then it can be readily shown  that it does not depend on the choice of the 
generator $\mathds{1}$ or its representative $\nu$.  Thus it induces a sheaf morphism $\iota^{(p)}:\scA^{(p)}\ra\scB^{(p)}$, which is injective.
In particular in the case $p=0$, we have\,:

\begin{corollary}\label{propemb}
The embedding $\scA\hra\scB$ is locally given by assigning to a  real analytic \fcn\ $f$ the  hyper\fcn\
$[(\tilde f\, \nu_{1}^{(0,n)},\tilde f\,\nu_{01}^{(0,n-1)})]\otimes\mathds{1}$, where $\tilde f$ is
a complexification of $f$.
\end{corollary}

The above embeddings are compatible with differentials\,:
\begin{proposition} The following diagram is commutative\,{\rm :}
\[
\SelectTips{cm}{}
\xymatrix%@C=.5cm
@R=.7cm
{\scA^{(p)}\ar[r]^-{\iota^{(p)}}\ar[d]^-{d}& \scB^{(p)}\ar[d]^-{d}\\
\scA^{(p+1)}\ar[r]^-{\iota^{(p+1)}}& \scB^{(p+1)}.}
\]
\end{proposition}
\begin{proof} Recall that $d:H^{p,n}_{\bar\vt}(V,V\ssm U)\ra H^{p+1,n}_{\bar\vt}(V,V\ssm U)$ 
assigns to the class of $(\nu_{1}^{(0,n)}\wedge\o,\nu_{01}^{(0,n-1)}\wedge\o)$ the class of $(-1)^{n}(\partial(\nu_{1}^{(0,n)}\wedge\o),-\partial(\nu_{01}^{(0,n-1)}\wedge\o))$ (cf. \eqref{defdiff}). 
From $D(\nu_{1},\nu_{01})=0$, we have 
$\partial\nu_{1}^{(0,n)}+\bp\nu_{1}^{(1,n-1)}=0$ and $\partial\nu_{01}^{(0,n-1)}+\bp\nu_{01}^{(1,n-2)}=\nu_{1}^{(1,n-1)}$. Then, using $\bp\o=0$, we compute
\[
\begin{aligned}
(-1)^{n}(\partial(\nu_{1}^{(0,n)}\wedge\o), -\partial(\nu_{01}^{(0,n-1)}\wedge\o))=(\nu_{1}^{(0,n)}&\wedge\partial\o, \nu_{01}^{(0,n-1)}\wedge\partial\o)\\
&+(-1)^{n-1}\bar\vt(\nu_{1}^{(1,n-1)}\wedge\o,\nu_{01}^{(1,n-2)}\wedge\o).
\end{aligned}
\]
Since $\partial\o(z)$ is the complexification of $d\o(x)$, we have the proposition.
\end{proof}

We finish this subsection by giving  particular representatives of $\mathds{1}^{c}$ and $\rho^{n}(\mathds{1}^{c})$. 

\begin{example}\label{exone} Let $U$ and $V$ coordinate \nbd s with
coordinates $(x_1,\dots,x_{n})$ and $(z_1,\dots,z_{n})$, $z_{i}=x_i + \sqrt{-1}y_i$.
%, \cdots, z_n= x_n + iy_n)$. 
We  set $V_0 = V \ssm U$ and $V_1 = V$.
We orient $T_{M}X$ so that
%$X$ and $M$ so that $(y_{1},\dots,y_{n},x_{1},\dots,x_{n})$ is a positive coordinate system on $X$
%and $(x_{1},\dots,x_{n})$ is a positive coordinate system on $M$. Thus 
$(y_{1},\dots,y_{n})$ is a positive fiber
coordinate system
% in the normal direction 
and this specifies the generator 
$\mathds{1}$ of $or_{M/X}(U)$.

Let 
\[
%\begin{equationth}	\tau_{01} 
\psi_{n}= %(-1)^{n(n+1)/2} 
C_n \dfrac{%\displaystyle
\sum_{i=1}^n 
(-1)^{i-1}y_i dy_1 \wedge \dots \wedge \widehat{dy_i} \wedge \dots \wedge
dy_n}{\Vert y \Vert^n}
\]
%\end{equationth}
be the angular form on $\R^{n}_{y}$ (cf. \eqref{ang}).
Then  by Proposition \ref{thtrivial},
%it is well-known $(${\bf{need reference HERE}}$)$ that
\[
%\begin{equationth}
\nu = (0,-\psi_{n}) 
\in \scE^{(n)}(V_1) \oplus \scE^{(n-1)}(V_{01})
=\scE^{(n)}(\V,\V') 
\]
%\end{equationth}
%gives the image of a generator of $or_{M/X}(M)$ (cf. \cite[Ch. II]{Su2}).  
represents $\mathds{1}^{c}$.
% (cf. Proposition \ref{thtrivial}). 
%Letting $\psi_{n}^{(0,n-1)}$ be the $(0,n-1)$-component of $\psi_{n}$,
In this case, $\rho^{n}(\nu) \in \scE^{(0,n)}(\V,\V')$ is
given by $(0,-\psi_{n}^{(0,n-1)})$, where $\psi_{n}^{(0,n-1)}$ is the $(0,n-1)$-component of $\psi_{n}$.
We may compute it using $y_i = \frac 1{2\sqrt{-1}}(z_i - \bar{z}_i)$\,:
%we have $\Vert y\Vert =\frac 12\Vert z-\bar z\Vert$
%and is given by
\[
%\begin{equationth}(0,%\,\, (-1)^{n(n+1)/2}-\psi_{n}^{0,n-1}),\qquad 
\psi_{n}^{(0,n-1)}=(\sqrt{-1})^n C_n \dfrac{%\displaystyle
	\sum_{i=1}^n 
	(-1)^i(z_i- \bar{z}_i) d\bar{z}_1 \wedge \dots \wedge 
	\widehat{d\bar{z}_i} \wedge \dots \wedge
	d\bar{z}_n}{\Vert z-\bar{z}\Vert^n}.
\]

In particular, in the case $n=1$,
\[
\psi_{1}^{(0,0)}=\psi_{1}=\frac 12 \frac y{|y|}.
\]
\end{example}

\begin{remark}\label{remdimone}
In the case $n=1$, 
 $(0,-\frac 1 2\frac y {|y|})$ represents both
%the canonical generator 
$\mathds{1}^{c}\in or_{M/X}^{c}(U)$ and $\rho^{1}(\mathds{1}^{c})\in H^{0,1}_{\bar\vt}(V,V\ssm U)$.
% in this case. 
They are also represented by
$(0,-\psi_{\pm})$, $\psi_{\pm}=\frac 1 2\frac y {|y|}\pm\frac 1 2$. Note that the support of $\psi_{\pm}$ in
$V\ssm U$ is $V_{\pm}=\{\pm y>0\}$ (cf. Lemma \ref{lem:unit_one_support} and Example \ref{exDirac} below,
$\psi_{\pm}=\pm\varphi_{\Omega_{\pm}}$ by the notation there). 
Any of those cocycles may be thought of as representing 
the generator $\mathds{1}$ of $or_{M/X}(U)$, as $or_{M/X}\ra or_{M/X}^{c}$ is injective.
\end{remark}

The contents of this subsection are generalized in the next subsection.

\subsection{Boundary value morphism}\label{ssbv}
The boundary value morphism is one of the most important tools in the theory of  hyperfunctions,
by which we can regard a holomorphic function on an open wedge along $M$
as a hyperfunction. In this subsection, we will define the boundary value morphism
in the framework of  relative de~Rham and Dolbeault cohomologies.

%%%%

We consider a pair $(V,\Omega)$ of an open neighborhood $V$ of $M$ in $X$ and 
an open set $\Omega$  in  $X$ satisfying
the following two conditions\,:
\begin{enumerate}
\item[($\rm{B}_1$)] $\overline{\Omega} \supset M$.
\item[($\rm{B}_2$)] The inclusion 
%\begin{equationth}
%\[
$(V \ssm \Omega) \ssm M  \hra V \ssm \Omega$
%\]
%\end{equationth}
is a homotopy equivalence.
\end{enumerate}
We give some examples of such pairs $(V,\Omega)$.

\begin{example}\label{exO=V} 
%Let $U$ be an open set in $M$ and $V$ a complex \nbd\ of $U$ in $X$.
If we take 
%Let $V$ be an open neighborhood of $M$, and 
$\Omega$  to be $V$,
the pair $(V,\Omega)$ satisfies the above the conditions, in particular,
the condition $(\rm{B}_2)$~is automatically satisfied as the both subsets are empty.
This is the situation we considered in the previous subsection, where
%Note that we encounter this situation when 
a real analytic function is regarded as a hyperfunction. This is a special case of boundary value
morphism.
\end{example}

\begin{example}{\label{exa:cone_case}}
Let $M$ be an open subset of $\R^n$ and 
$X = M \times \sqrt{-1}\R_y^n \subset \C^n$ 
with coordinates $(z_1, \dots, z_n)$, $z_{i}=x_i + \sqrt{-1}y_i$.
Let $V$ be an open neighborhood of $M$ such that
$V \cap (\{x\} \times \sqrt{-1}\R_y^n)$ is convex for any $x \in M$, and
let $\Gamma$ be a non-empty  open cone in $\R^n_y$ for which
$\R^n_y \ssm (\Gamma \cup \{0\})$ is contractible.
Then the pair of $V$ and $\Omega = (M \times \sqrt{-1}\Gamma) \cap V$ satisfies the conditions ($\rm{B}_1$) and ($\rm{B}_2$).
\end{example}

Let us now define the boundary value morphism\,:
%\begin{equationth}
\[
\begin{aligned}
b_\Omega: \scO(\Omega) \lra \scB(M) 
&= H^n_M(X;\scO) \otimes_{{}_{\Z_M(M)}}
%\underset{\Z}{\otimes}
%H^n_M(X;\Z_X) \\
or_{M/X}(M) \\
&
\simeq H^{0,n}_{\bar\vt}(\V,\V') \otimes_{{}_{\Z_M(M)}} 
or_{M/X}(M),
%H^n_M(X;\Z_X),
\end{aligned}
\]
%\end{equationth}
where $\V = \{V_0, V_1\}$ and
$\V'= \{V_0\}$ with $V_0 = V \ssm M$ and $V_1 = V$. 
%We also set $V_{01} = V_0 \cap V_1$ as usual.
To give a concrete representative of $b_\Omega(f)$ ($f \in \scO(\Omega)$)
in the last cohomology  of the above diagram, we need some preparations.

Let $\mathds{1}\in or_{M/X}(M)=H^{n}_{M}(X;\Z_{X})$ and $\mathds{1}^{c}\in or_{M/X}^{c}(M)=H^{n}_{M}(X;\C_{X}) \simeq H^{n}_{D}(\V,\V')$ be as in the previous subsection.
The following lemma is crucial to our construction of $b_\Omega$\,:

\begin{lemma}{\label{lem:unit_one_support}}
The class $\mathds{1}^{c} \in H^{n}_{D}(\V,\V')$ has a representative
\[
(\nu_1, \nu_{01})\in \scE^{(n)}(V_1) \oplus \scE^{(n-1)}(V_{01})
=\scE^{(n)}(\V,\V') 
\]
which satisfies $\op{Supp}_{V_1}(\nu_1) \subset \Omega$ and 
$\op{Supp}_{V_{01}}(\nu_{01}) \subset \Omega$. 
\end{lemma}

\begin{proof}
Replacing $\Omega$ with $\Omega \cap V$ if necessary, we may assume $\Omega \subset V$.
Let $j: \Omega \hra V$ (resp. $i: V \ssm \Omega \hra V$)
be an open (resp. a closed) embedding. Then we have the exact sequence of sheaves on $V$\,:
\[
0 \lra j_! j^{-1} \C_V \lra  \C_V\lra  i_*i^{-1} \C_V\lra 0,
\]
where $j_{!}$ denotes the direct image with proper supports.
% (cf. \cite{KS}). 
This yields the exact sequence
\[
\cdots\lra H^{q-1}_{M}(V;i_{*} i^{-1} \C)     \lra H^{q}_{M}(V;j_! j^{-1} \C)\overset\iota\lra H^{q}_{M}(V;\C)\lra H^{q}_{M}(V;i_{*}i^{-1} \C)\lra\cdots.
\]

We claim that 
\begin{equationth}\label{van}
H^{q}_{M}(V;i_{*}i^{-1} \C)=0\qquad\text{for}\ \  q\ge 0
\end{equationth}
 so that $\iota$ is an \iso. For this we set $\scF=i_{*}i^{-1} \C$ and  consider the exact sequence
\[
H^{q-1}(V;\scF) \ra H^{q-1}(V\ssm M;\scF)     \ra H^{q}_{M}(V;\scF)\ra H^{q}(V;\scF)\ra H^{q}(V\ssm M;\scF).
\]
In the above, $H^{q}(V;\scF)=H^{q}(V\ssm\Omega;\C)$ and $H^{q}(V\ssm M;\scF)=H^{q}((V\ssm\Omega)\ssm M;\C)$.
Thus by the condition $(\rm{B}_2)$ above, we have \eqref{van} and $\iota$ is an \iso.

Let $\scE^{(\bullet)}_{V}$ denote the de~Rham complex on $V$, which gives
a fine resolution of $\C_V$. Since any of the sections of $j_{!}j^{-1}\scE^{(q)}_{V}$ may be thought of as a $q$-form with support
in the intersection of its domain of definition and $\Omega$, the sheaf $j_{!}j^{-1}\scE^{(q)}$ admits a natural action of the sheaf $\scE_{V}$ of $C^{\infty}$ \fcn s on $V$ and thus it is fine.
Therefore the complex  $j_{!}j^{-1}\scE^{(\bullet)}_{V}$ gives a fine resolution of
$j_{!}j^{-1}\C_{V}$. We denote by $d'$  its differential $j_{!}j^{-1}d$, which is in fact the usual exterior
derivative $d$ on forms with support in $\Omega$. We set $D'=\delta+(-1)^{\bullet}d'$ and consider the cohomology $H^{q}_{D'}(\V,\V')$ of  the complex $(j_{!}j^{-1}\scE^{(\bullet)}_{V}(\V,\V'),D')$,
which is defined similarly as the relative de~Rham cohomology, replacing $\scE^{(\bullet)}$ and $D$ by
$j_{!}j^{-1}\scE^{(\bullet)}_{V}$ and $D'$ in Definition \ref{defreldR}. Then there is a canonical morphism
$H^{q}_{D'}(\V,\V')\ra H^{q}_{D}(\V,\V')$. Moreover $H^{q}_{D'}(\V,\V')$ is canonically isomorphic with
$H^{q}_{M}(V;j_{!}j^{-1}\C)$ and
we have the
following commutative diagram (cf. \cite{Su11})\,:
\[
\SelectTips{cm}{}
\xymatrix
@C=.6cm
@R=.5cm
{H^{q}_{M}(V;j_{!}j^{-1}\C)\ar[r]^-{\sim}_-{\iota} \ar@{-}[d]^-{\wr}
%\ar@{=}[d]^{\wr}
&H^{q}_{M}(V;\C)\ar@{-}[d]^-{\wr}\\
 H^{q}_{D'}(\V,\V')\ar[r] & H^{q}_{D}(\V,\V').}
\]
In particular we have
\begin{equationth}\label{eq:iso-support-identity}
H^{n}_{D'}(\V,\V')\simeq H^{n}_{D}(\V,\V'),
\end{equationth}
which assures the existence of a desired representative.
\end{proof}

\begin{remark}
In the above lemma, since $\nu_{01}$ is a section defined only on $V_{01} = V \ssm M$,
its support is a closed set in $V \ssm M$,
however, it is not necessarily closed in $V$.
\end{remark}

Now we give some examples of representatives of $\mathds{1}^{c}$ described in 
Lemma \ref{lem:unit_one_support}. In the situation of Example \ref{exO=V}, we already gave a particular example 
in Example \ref{exone}.

\begin{example}{\label{exa:support_cone_fundamental}}
Let us consider the situation described in Example \ref{exa:cone_case}.
Here we may assume $M=\R^n$, $X = V = \C^n$ and
$\Omega = M \times \sqrt{-1} \Gamma$,
as the other cases are  obtained by restriction of this case.
We set $V_0 = X \ssm M$ and $V_1 = X$ as usual.

We first take $n$ linearly independent unit vectors $\eta_1,\dots, \eta_n$ in $\R^n_y$ so that
\[
\underset{1 \le k \le n}{\bigcap} H_{k} \subset \Gamma
\]
holds, where we set $H_{k} = \{\,y \in \R^n_y\mid \langle y,\, \eta_k\rangle > 0\,\}$.
We also set
%\[
$\eta_{n+1} = - (\eta_1 + \dots + \eta_n)$.
%\]
Then, let $\varphi_k$, $k=1,\dots,n+1$, be $C^\infty$ functions on $X \ssm M$ which satisfy
\begin{enumerate}
	\item[(1)] $\op{Supp}_{X \ssm M} (\varphi_k) \subset M \times \sqrt{-1}H_{k}$ for any $k=1,\dots, n+1$.
\item[(2)] $\displaystyle\sum_{k=1}^{n+1} \varphi_k = 1$ on $X \ssm M$.
\end{enumerate}
Set
%\begin{equationth}
\[
	\nu_{01} = 
	(-1)^n
	(n-1)!\, \check{\chi}_{H_{{n+1}}}\,
	d\varphi_1 \wedge \dots \wedge d\varphi_{n-1},
\]
%\end{equationth}
where $\check{\chi}_{H_{{n+1}}}$ is the anti-characteristic function of the set $H_{{n+1}}$,
that is, 
\[
\check{\chi}_{H_{{n+1}}}(z) = 
\left\{
	\begin{matrix}
0 \qquad &z \in H_{{n+1}}, \\
1 \qquad &\text{otherwise}.
\end{matrix}
\right.
\]
Then we can easily confirm that
$\nu_{01} \in \scE^{(n-1)}(X \ssm M)$ and
\[
\op{Supp}_{X \ssm M}(\nu_{01}) \subset 
M \times \sqrt{-1}\underset{1 \le k \le n}{\bigcap} H_{k} 
\subset \Omega.
\]
Furthermore, as  will be shown in Lemma \ref{lem:tau_is_generator} in Appendix,  
%that
%\begin{equationth}
\[
\nu = (0,\nu_{01}) 
\in \scE^{(n)}(V_1) \oplus
\scE^{(n-1)}(V_{01})
= \scE^{(n)}(\V,\V')
\]
%\end{equationth}
gives the image of a positively oriented generator of 
$or_{M/X}(M)$ with the orientations on $M$ and $X$ as  in Example \ref{exone}.
%standard orientations on $M$ and $X$ (see the next example, where we specify the standard orientations on these spaces). 
By  definition of $\rho^{n}:\scE^{(n)}(\V,\V')\ra \scE^{(0,n)}(\V,\V')$, we have
\[
%\begin{equationth}
\rho^{n}(\nu) =(0,\nu_{01}^{(0,n-1)})= (0,\,
(-1)^n(n-1)!\, \check{\chi}_{H_{{n+1}}}\,
	 \bar{\partial}\varphi_1 \wedge \dots \wedge \bar{\partial}\varphi_{n-1})
	 \in \scE^{(0,n)}(\V,\V'),
\]
%\end{equationth}
which satisfies $\bar{\vartheta} \rho^{n}(\nu) = \rho^{n}(D(\nu)) = 0$.
\end{example}

Now we are ready to define the boundary value morphism
\[
%\begin{equationth}
	b_\Omega: \scO(\Omega) \to H^{0,n}_{\bar\vt}(\V,\V')\otimes_{{}_{\Z_M(M)}}
	or_{M/X}(M) = \scB(M).
\]
%\end{equationth}
Let $\nu = (\nu_1,\nu_{01}) \in \scE^{(n)}(\V,\V')$
be a representative of $\mathds{1}^{c}$
with $\op{Supp}_{V_1}(\nu_1) \subset \Omega$ and
$\op{Supp}_{V_{01}}(\nu_{01}) \subset \Omega$
%Note that
(cf. Lemma {\ref{lem:unit_one_support}}).
%guarantees the existence of such a representative.
Take $f \in \scO(\Omega)$. 
Since $\operatorname{Supp}_{V_1}(\rho^n(\nu_1))\subset \Omega$
 and $\operatorname{Supp}_{V_{01}}(\rho^{n-1}(\nu_{01}))\subset \Omega$,
%are  in $\Omega$, 
we may regard 
$f \rho^n(\nu_1)$ as a $(0,n)$-form on $V_1$ 
and $f \rho^{n-1}(\nu_{01})$ 
as a $(0,n-1)$-form on $V_{01}$. Hence, we have
\[
f\rho(\nu) = (f \rho^n(\nu_1),\, f\rho^{n-1}(\nu_{01})) 
%\in \dccppk{0}{n}{\mathcal{W}}{\mathcal{W}'}.
\in \scE^{(0,n)}(V_1) \oplus \scE^{(0,n-1)}(V_{01})
= \scE^{(0,n)}(\V,\V').
\]
Moreover it is a $\bar{\vartheta}$-cocycle, as $f$ is holomorphic,
and defines a class
$[f\rho(\nu)] \in H^{0,n}_{\bar\vt}(\V,\V')$. Then we define
the boundary value morphism by
\[
%\begin{equationth}
	b_\Omega(f) = [f\rho(\nu)] \otimes \mathds{1}
	\in H^{0,n}_{\bar\vt}(\V,\V') \otimes_{{}_{\Z_M(M)}} or_{M/X}(M).
\]
%\end{equationth}
\begin{lemma}{\label{lem:boundary_map_ok}}
The above $b_\Omega$ is well-defined.
\end{lemma}
\begin{proof}
Let $\nu' = (\nu'_1,\nu'_{01}) \in \scE^{(n)}(\V,\V')$
be 
another representative of $\mathds{1}^{c}$ 
with $\op{Supp}_{V_1}(\nu'_1) \subset \Omega$ and
$\op{Supp}_{V_{01}}(\nu'_{01}) \subset \Omega$. 
By  \eqref{eq:iso-support-identity}, we can find
$\omega = (\omega_1,\omega_{01}) \in \scE^{(n-1)}(\V,\V')$
with 
$\op{Supp}_{V_1}(\omega_1) \subset \Omega$ and
$\op{Supp}_{V_{01}}(\omega_{01}) \subset \Omega$
satisfying
$
\nu - \nu' =D\o.
$
Then we have $f\rho(\omega) \in \scE^{(0,n-1)}(\V,\V')$
and
%\[
$f\rho(\nu) - f\rho(\nu') = \bar{\vartheta}(f\rho(\omega))$,
%\]
which shows
\[
[f\rho(\nu)] \otimes \mathds{1} = [f\rho(\nu')] \otimes \mathds{1}
	\quad\text{in}\quad H^{0,n}_{\bar\vt}(\V,\V').
\]
Hence $b_\Omega$ does not depend on the choice
of the representative of $\mathds{1}^{c}$. It does not depend on the choice of $\mathds{1}$ either by \eqref{tensgen}.
%This completes the proof.
\end{proof}

As a corollary, we have the following\,:
\begin{corollary}{\label{cor:restriction_boundary_map}}
Let $V' \subset V$ be an open neighborhood of $M$, and
let $\Omega' \subset \Omega$ be an open subset in $X$.
Assume that the pair $(V',\Omega')$ also satisfies the conditions $(\rm{B}_1)$~and 
$(\rm{B}_2)$. Then we have
\[
b_{\Omega'}(f|_{\Omega'}) = b_{\Omega}(f), \qquad f \in \scO(\Omega).
\]
\end{corollary}

\begin{example}\label{exDirac}
Let $M=\R$ and $X=\C$ with  coordinate $z = x+\sqrt{-1}y$ and set 
$V_0 = \C\ssm \R$ and $V_1 = \C$. Note that
$V_0$ is a disjoint union of $\Omega_{\pm} = \{\pm y > 0\}$.
Define coverings $\V = \{V_0, V_{1}\}$ and $\V' = \{V_{0}\}$ as usual.

It is well-known that Dirac's delta function is the difference of boundary 
values on $\Omega_+$ and $\Omega_-$
of 
%$\dfrac{-1}{2\pi\sqrt{-1} z}$, 
$\frac{-1}{2\pi\sqrt{-1}}\frac 1 z$, that is,
\[
%\begin{equationth}
\delta(x) = 
b_{\Omega_+} \left(\frac{-1}{2\pi\sqrt{-1}}\frac 1 z\right)
-
b_{\Omega_-}\Big(\frac{-1}{2\pi\sqrt{-1}}\frac 1 z\Big).
\]
%\end{equationth}
Define the functions
\[
\varphi_{\Omega_+}(z) = 
\begin{cases}
1 &  z \in \Omega_+, \\
0 &  z \in \Omega_- \\
\end{cases},
\qquad
\varphi_{\Omega_-}(z) = 
\begin{cases}
0 &  z \in \Omega_+, \\
1 &  z \in \Omega_- \\
\end{cases},
\]
and set
\[
\nu_{\Omega_\pm} = (0,\,-\varphi_{\Omega_\pm}) \in 
\scE^{(1)}(V_{1}) \oplus \scE^{(0)}(V_{01})
	 =\scE^{(1)}(\V,\V').
\]
Note that $\nu_{\Omega_\pm}$ may be thought of as representing  both a generator $\mathds{1}$ of $or_{M/X}(M)$ and $\rho^{1}(\mathds{1}^{c})\in H^{0,1}_{\bar\vt}(\V,\V')$ (cf. Remark \ref{remdimone}).
%$[\nu_{\Omega_\pm}] \in H^{1}_{D}(\V,\V')$ 
%can be identified with a generator of $or_{M/X}(M)$ (cf. Corollary~\ref{corgendim1}),
% it follows 
 Thus from
the definition of $b_{\Omega_\pm}: \scO(\Omega_{\pm}) \to \scB(M)$, 
 we have, for $F \in \scO(\Omega_{\pm})$,
\[
b_{\Omega_\pm}(F) = \left[\left(0,\, -F(z) \varphi_{\Omega_\pm}\right)\right]
\otimes [\nu_{\Omega_\pm}]
\,\,\in\, H^{0,1}_{\bar\vt}(\V,\V')
\otimes_{{}_{\Z_M(M)}}
or_{M/X}(M),
\]
which is often written  $F(x\pm\sqrt{-1}\,0)$ in Sato's context.
Hence we have
\[
%\begin{equationth}
\begin{aligned}
\delta(x) &= 
\left[\left(0,\, \dfrac 1{2\pi\sqrt{-1}}\frac{\varphi_{\Omega_+}} z\right)\right]
\otimes [\nu_{\Omega_+}] -
\left[\left(0,\, \dfrac 1{2\pi\sqrt{-1}}\frac{\varphi_{\Omega_-}} z\right)\right]
\otimes [\nu_{\Omega_-}] \\
&\\
&=
\dfrac{-1}{2\pi\sqrt{-1}} 
\left(\dfrac{1}{x + \sqrt{-1}\,0} - \dfrac{1}{x - \sqrt{-1}\,0}\right).
\end{aligned}
\]
%\end{equationth}
%Then we have
We may also express it as
\[
\begin{aligned}
\delta(x) &=
\left[\left(0,\, 
\dfrac{1}{2\pi\sqrt{-1}}\frac {\varphi_{\Omega_+} + \varphi_{\Omega_-}} z\right)\right]
\otimes [\nu_{\Omega_+}] 
=
\left[\left(0,\, \frac{1}{2\pi\sqrt{-1}}\frac 1 z\right)\right]
\otimes [\nu_{\Omega_+}] %\\
%&= \left[\left(0,\, \frac{1}{2\pi\sqrt{-1}}\frac 1 z\right)\right]
%\otimes [-\nu_{\Omega_+}]
%\quad \in H^{0,1}_{\bar\vt}(\V,\V')\otimes_{{}_{\Z}}
%or_{M/X}(M).
\end{aligned}
\]
in $H^{0,1}_{\bar\vt}(\V,\V')\otimes_{{}_{\Z}}or_{M/X}(M)$.
Recall that (cf. Remark \ref{remdimone})
%(cf. Example \ref{exone}),
 if  we orient the normal direction
%$\C$ and $\R$ so that $(y,x)$ and $x$ are 
so that $y$ is a positive coordinate, 
%$(0,-\frac 1 2\frac y {|y|})$ represents
%the canonical generator of $or_{M/X}(M)$.
%Thus 
$[\nu_{\Omega_+}]$ is the canonical generator of $or_{M/X}(M)$
% in this case 
and thus the above delta
\fcn\ coincides with the one in Definition \ref{defdeltafcn} for $n=1$.

If we fix the orientation as above, we have a canonical \iso
%In our convention, through the isomorphism
\[
%\begin{equationth}
	\Z=\Z_M(M) \simeq or_{M/X}(M)=H^{1}_M(X;\Z_X)
\]
%\end{equationth}
which sends $1 \in \Z$ to $[\nu_{\Omega_+}]\in or_{M/X}(M)$.
%the positively oriented generator
%under the standard orientation, 
If we identify $\scB(M)$ with
$H^1_M(X;\scO)=H^{0,1}_{\bar\vt}(\V,\V')$ via the above \iso, 
%i.e.,
%if we forget the orientation sheaf, 
we have the expression
%. Since $\varphi_{\Omega_+}$ is a 
%section on $V_{01}$ (not on $V_{10}$) in our context, we see that $[-\nu_{\Omega_+}]$ is the positively
%oriented generator under the standard orientation. Hence we have, under our convention,
\begin{equationth}{\label{eq:dirac_without_orientation}}
\delta(x) = \left[\left(0,\, \frac{1}{2\pi\sqrt{-1}}\frac 1 z\right)\right]
\in H^{0,1}_{\bar\vt}(\V,\V').
\end{equationth}
\end{example}

%\

\begin{remark}
Since we always assume $M$ to be orientable,
as we see in the above example, we may omit the relative orientation 
$or_{M/X}(M)$ in the definition of hyperfunctions via the
isomorphism $\Z_M(M) \simeq or_{M/X}(M)$.
Here we fix the isomorphism so that $1 \in \Z_M(M)$ is sent to the positively oriented
generator of $or_{M/X}(M)$.
This omission does not have much impact 
on usual treatment of  hyperfunctions, however, some particular 
operations such as  coordinate transformations or  integration of  hyperfunctions
require special attention: For example,
let us consider the coordinate transformation $x\mapsto  -x$ in the above example.
It follows from the definition that 
$\delta(x)$ remains unchanged by this.
% transformation.
The element defined by \eqref{eq:dirac_without_orientation}, however, 
changes its sign under this transformation. 
Thus, if we omit the orientation sheaf, 
we are required to change the sign of a hyperfunction manually 
under a coordinate transformation reversing the orientation.
For the integration of  hyperfunction densities, a similar consideration is needed,
see  Subsection~{\ref{subsection:integration}} 
%below 
for details.
\end{remark}

\subsection{Microlocal analyticity}{\label{subsec:micro-anal}}

We first recall the notion of microlocal analyticity of a hyperfunction
(\cite{Skk}, \cite{KS}). Then we will give its interpretation in our framework.
In this subsection, we assume that $X$ and $M$ are oriented so that we  omit the orientation sheaves.

Let $T^*_MX$ denote the conormal bundle of $M$ in $X$, which is isomorphic
to $\sqrt{-1}T^*M$, and $\pi: T^*_MX \to M$  the  projection. We describe the spectral map
$\op{sp}: \pi^{-1}\scB \to \SHmicro$ in our frame work, where $\scC$ denotes the sheaf of micro\fcn s on
$T^*_MX$.
Let $p_0 = (x_0;\,\sqrt{-1}\xi_0)$ be a point in $T^*_MX = \sqrt{-1}T^*M$. 
Then it is known that the stalk $\SHmicro_{p_0}$ 
at $p_0$ is given by the following formula (see Theorem 4.3.2 and Definition 11.5.1 in
\cite{KS})\,:
Under a ($C^1$-class) local trivialization near $x_0$, i.e., 
$(M,X) \simeq (\R_x^n,\C^n = \R_x^n \times \sqrt{-1}\R_y^n)$ near $x_0$,
we have
%\begin{equationth}
\[
\SHmicro_{p_0} = \underset{\underset{V,\,G}{\lra}}{\lim}\,
H^n_{\R^n_x \times {\sqrt{-1}G}}(V;\scO),
\]
%\end{equationth}
where $V$ runs through open neighborhoods of $x_0$ and
$G$ 
ranges 
 through closed cones in $\R_y^n$ satisfying
\begin{equationth}{\label{eq:microlocal-cone-condition}}
G \ssm \{0\} \subset 
\{\,y \in \R^n \mid\, \op{Re}\,\langle \sqrt{-1}y,\,\sqrt{-1}\xi_0\rangle > 0\,\}
=
\{\,y \in \R^n \mid\, \langle y,\,\xi_0\rangle < 0\,\}.
\end{equationth}
Therefore, setting $\W_{V,G} = \{V \ssm (\R^n_x \times \sqrt{-1}G),\, V\}$
and
$\W_{V,G}' = \{V \ssm (\R^n_x\times \sqrt{-1} G)\}$, we get the expression  in our framework\,:
%\begin{equationth}
\[
\SHmicro_{p_0} \simeq \underset{\underset{V,\,G}{\lra}}{\lim}\, 
H^{0,n}_{\bar\vt}(\W_{V,G},\W'_{V,G}).
\]
%\end{equationth}
%where we set
%\begin{equationth}
%\[
%\W_{V,G} = \{V \ssm (\R^n_x \times \sqrt{-1}G),\, V\}
%\quad \text{and} \quad
%\W_{V,G}' = \{V \ssm (\R^n_x\times \sqrt{-1} G)\}.
%\]
%\end{equationth}
We also set
%\begin{equationth}
%\[
$\W_V = \{V_0=V\ssm M,\,V_1= V\}$ and
%\quad \text{and} \quad
$\W'_V = \{V_0\}$.
%\]
%\end{equationth}
Then the spectral map at $p_0$
\[
\scB_{x_0} \simeq
\underset{\underset{V}{\lra}}{\lim}\, 
H^{0,n}_{\bar\vt}(\W_{V},\W'_{V})\overset{\op{sp}}{\lra}
\SHmicro_{p_0} \simeq \underset{\underset{V,\,G}{\lra}}{\lim}\,
H^{0,n}_{\bar\vt}(\W_{V,G},\W'_{V,G})
\]
  is simply the one  canonically induced from the restriction map $V \ssm M \hra V \ssm (\R^n_x \times \sqrt{-1}G)$.

\begin{definition}
For a hyperfunction $u$ at $x_0$, we say that $u$ is {\em microlocally analytic}
at $p_0$ if $\op{sp}(u)$ becomes zero at $p_0$ as a microfunction.
\end{definition}

We have the following equivalent characterization\,:

\begin{proposition}{\label{prop:equiv-ss}}
Let $u$ be a hyperfunction at $x_0$.
Then $u$ is microlocally analytic at $p_0$ if and only if
there exist a closed cone $G$ satisfying
\eqref{eq:microlocal-cone-condition}, an open neighborhood $V$ of $x_0$
and a representative 
\[
(\tau_1,\tau_{01}) \in \scE^{(0,n)}(V_1) \oplus
\scE^{(0,n-1)}(V_{01})= \scE^{(0,n)}(\W_V,\W'_V)
\]
of $u$ near $x_0$ which satisfies $\tau_1 = 0$ and
$\op{Supp}_{V\ssm M}(\tau_{01}) \subset \R^n_x \times \sqrt{-1}G$.
\end{proposition}

\begin{proof} Let $u$ be represented by $(\xi_{1},\xi_{01})\in  \scE^{(0,n)}(\W_V,\W'_V)$.
%Taking a Stein \nbd\ of $p_{0}$ as $V$, we may assume that $\xi_{1}=0$.
 If $u$ is microlocally analytic at
$p_{0}$, there exist  $G'$ and a cochain $(\eta_{1},\eta_{01})\in  \scE^{(0,n-1)}(\W_{V,G'},\W'_{V,G'})$ \st
\[
(\xi_{1},\xi_{01})=\bar\vt(\eta_{1},\eta_{01})=(\bp\eta_{1},\eta_{1}-\bp\eta_{01})\qquad\text{in}\ \ \scE^{(0,n)}(\W_{V,G'},\W'_{V,G'}),
\]
where $\xi_{01}$ is to be restricted to $V \ssm (\R^n_x \times \sqrt{-1}G')$. Let $G$ be a closed cone with \eqref{eq:microlocal-cone-condition} containing $G'\ssm\{0\}$ in its interior. Let $\psi$ be a $C^{\infty}$ \fcn\ on $V\ssm M$
\st\ $\psi\equiv 1$ on the complement of $\R^n_x \times \sqrt{-1}\op{Int}G$ in $V\ssm M$ and $\psi\equiv 0$ on $\R^n_x \times \sqrt{-1}(G'\ssm\{0\})$. Note that the both sets are closed in $V\ssm M$ and that such a $\psi$ may be 
constructed making it ``radially constant''. Then $\psi\eta_{01}$ is a $(0,n-2)$-form on $V\ssm M$. Set $\tau_{1}=0$ and $\tau_{01}=\xi_{01}-\eta_{1}+\bp(\psi\eta_{01})$. Then $\op{Supp}_{V\ssm M}(\tau_{01}) \subset \R^n_x \times \sqrt{-1}G$ and 
\[
(\xi_{1},\xi_{01})-(\tau_{1},\tau_{01})=\bar\vt(\eta_{1},\psi\eta_{01})\qquad\text{in}\ \ \scE^{(0,n)}(\W_{V},\W'_{V})
\]
so that $u$ is represented by $(\tau_{1},\tau_{01})$.
\end{proof}

We denote by $\op{SS}(u)$ the set of points in $T^*_MX$ 
at which $u$ is not microlocally analytic.
By the construction of the boundary value morphism in the previous subsection and the definition of microlocal analyticity,
we have\,:

\begin{proposition}
Let $M$ be an open subset of $\R_x^n$ and $X = M \times \sqrt{-1}\R_y^n$.
Also let $V$ be an open neighborhood of $M$ in $X$ and  $\Omega$
 an open set in $X$. 
Assume that $\Omega \cap (\{x_0\} \times \sqrt{-1}\R_y^n)$
{\rm (}resp. $V \cap (\{x_0\} \times \sqrt{-1}\R_y^n)${\rm )}
is a non-empty convex cone {\rm (}resp. a convex set{\rm )} for any $x_0 \in M$.
Then
we have
\[
%\begin{equationth}
\op{SS}(b_{\Omega\cap V}(f)) \subset \Omega^\circ,
\qquad f \in \scO(\Omega \cap V),
\]
%\end{equationth}
where $\Omega^\circ$ is the polar set of $\Omega$ defined by
\[
%\begin{equationth}
\underset{x \in M}{\bigsqcup}\,
\{\,\sqrt{-1}\xi \in (T^*_MX)_x\mid\, \langle \xi, y\rangle \ge 0\,\,
\text{for any $y$ with $x + \sqrt{-1}y \in \Omega$}\,\} \subset T^*_MX.
\]
%\end{equationth}
\end{proposition}

\subsection{External product of hyperforms}

For each $k=1,2,\dots,\ell$, let $M_k$  be a real analytic manifold 
of dimension $n_k$,  $X_k$ its complexification and 
 $p_k$ 
 %$(k=1,2,\dots,\ell)$ be 
a non-negative integer.  All the manifolds are 
assumed to be oriented, and thus, we omit the relative orientations $or_{M_k/X_{k}}(M_k)$ in 
the definition of hyperforms $\mathscr{B}^{(p_k)}(M_k)$ throughout this subsection. Set
\[
M = M_1 \times \dots \times M_\ell,\qquad
X = X_1 \times \dots \times X_\ell
\]
and denote
by $\pi_k: X \to X_k$ the canonical projection. We also set $n = n_1 + \dots + n_\ell$ and $p=p_{1}+\cdots+p_{\ell}$. For each $k$, we consider the coverings $\V_k = \{V_{k,\,0},\,V_{k,\,1}\}$
and $\V'_k = \{V_{k,\,0}\}$ of $X_{k}$ and $X_{k}\ssm M_{k}$ given by $V_{k,\,0} = X_k \ssm M_k$ and
$V_{k,\,1} = X_k$. We  set
$V_{k,\,01} = V_{k,\,0} \cap V_{k,\,1}$. We also consider coverings $\V = \{V_0,\,V_1\}$
and $\V' = \{V_0\}$ of $X$ and $X\ssm M$ with 
%\[
$V_0 = X \ssm M$, %\qquad
$V_1 = X$ and %\qquad
$V_{01} = V_0 \cap V_1$.

Given $\ell$ hyperforms $u_k \in \SHhyperf{(p_k)}(M_k)$, $k=1,\dots,\ell$, and
their representatives
\[
\tau_k = (\tau_{k,\,1},\,\tau_{k,\,01})
\in \scE^{(p_k,n_k)}(V_{k,\,1}) \oplus
\scE^{(p_k,n_k-1)}(V_{k,\,01})
=\scE^{(p_k,n_k)}(\V_{k},\V'_{k}),
\]
we compute a concrete representative of the external product 
\[
u = u_1 \times u_2 \times \cdots 
\times u_\ell \in \SHhyperf{(p)}(M)
\] 
%of the hyperforms $u_1$, $\cdots$, $u_\ell$ 
from the representatives $\tau_k$.
%Here we set $p = p_1 + \dots + p_\ell$.
Let $\varphi_k$, $k=1,\dots,\ell$, be $C^\infty$-functions on
$V_0$ which satisfy
\begin{enumerate}
	\item[(1)] $\op{Supp}_{V_0}(\varphi_k) \subset \pi_k^{-1}(V_{k,\,0})$,  
		$k=1,\dots,\ell$,
\item[(2)] $\varphi_1 + \cdots +\varphi_\ell = 1$ on $V_0$.
\end{enumerate}

First we introduce two families of forms $(\bar{\partial}\varphi)_\beta$'s and
$\tau_\alpha$'s.
Set $\Lambda = \{1,2,\dots,\ell\}$.
For $\beta = (\beta_1,\,\dots,\,\beta_k) \in \Lambda^k$, we define
\[
%\begin{equationth}
	(\bar{\partial} \varphi)_\beta
	= \bar{\partial} \varphi_{\beta_1} \wedge \dots \wedge
	   \bar{\partial} \varphi_{\beta_k}.
\]
%\end{equationth}
Note that $(\bar{\partial} \varphi)_\beta$ is a $(0,k)$-form 
%defined 
on $V_0$
whose support is 
%contained 
in $\pi^{-1}_{\beta_1}(V_{\beta_{1},\,0}) \cap \cdots \cap 
\pi^{-1}_{\beta_k}(V_{\beta_{k},\,0})$.
%\[
%\pi^{-1}_{\beta_1}(V_{\beta_{1},\,0}) \cap \cdots \cap 
%\pi^{-1}_{\beta_k}(V_{\beta_{k},\,0}).
%\]
Furthermore, for $\alpha = (\alpha_1,\,\dots,\,\alpha_k) \in \Lambda^k$ with
$\alpha_1 < \alpha_2 < \dots < \alpha_k$, we define $\tau_\alpha$ by
\[
%\begin{equationth}
\tau_\alpha = (-1)^{\sigma(\alpha)} 
\omega_1\, {\wedge} \,\omega_2\, {\wedge}\, \cdots\, {\wedge}\, \omega_\ell,
\]
%\end{equationth}
where, for $j=1,2,\dots,\ell$,
\[
\omega_j = 
\left\{
\begin{array}{ll}
	\tau_{j,\,{01}}	\qquad&\text{if $\alpha$ contains the index $j$}, \\
	\tau_{j,\,{1}}	\qquad&\text{otherwise},
\end{array}
\right.
\]
and
\[
%\begin{equationth}
\sigma(\alpha) = 
\dfrac{k(k - 1)}{2}
+
\sum_{j=1}^k \sum_{i=1}^{\alpha_j - 1} (n_i + p_i).
\]
%\end{equationth}
W extend $\tau_\alpha$ 
to general $\alpha = (\alpha_1,\,\dots,\,\alpha_k) \in \Lambda^k$ 
in the usual way, that is, 
%we set 
$\tau_\alpha = \op{sgn}(\mu)\,\tau_{\mu(\alpha)}$
for any permutation $\mu$ on $\alpha$.
It is easy to see that $\tau_\alpha$ is a $(p,\,n-k)$-form defined on
\[
\pi^{-1}_{\alpha_1}(V_{{\alpha_1},\,0}) \cap \dots \cap 
\pi^{-1}_{\alpha_k}(V_{{\alpha_k},\,0}).
\]

Then, for $i = 1,2,\dots,\ell$, we set
\[
%\begin{equationth}
\kappa_i = 
\underset{\beta \in \Lambda^{\ell-1}}{\sum} 
(\bar{\partial}\varphi)_\beta \wedge \tau_{\beta\, i}
+
\displaystyle\sum_{
0 \le k \le \ell-2,\,
\beta \in \Lambda^{k}}
(\bar{\partial}\varphi)_\beta \wedge 
\left(\displaystyle\sum_{\lambda \in \Lambda,\, 1 \le j \le k+1}
(-1)^{j+1} \varphi_\lambda\,\, \tau_{\overset{\overset{j}{\vee}}{\lambda\, \beta\, i}}
\right),
\]
%\end{equationth}
where $\overset{\overset{j}{\vee}}{\lambda\, \beta\, i}$ is the sequence
``$\lambda\, \beta\, i$'' with the $j$-th component  removed. 
%Note that h
Here the first $\lambda$ in 
%the sequence 
$\lambda\, \beta\, i$ is considered
to be the $0$-th component and the last $i$  the $(k+1)$-st component. 
\begin{lemma}{\label{lem:product_gluing}}
Each $\kappa_i$ is a $(p,n-1)$-form defined on $\pi^{-1}_i(V_{i,\,0})$.
Furthermore, for $i$ and $j$ in $\Lambda$, we have
$
\kappa_i = \kappa_j
$
on
$\pi^{-1}_i(V_{i,\,0}) \cap \pi^{-1}_j(V_{j,\,0})$.
\end{lemma}

It follows from the above lemma 
and the fact $V_0 = \underset{1 \le i \le \ell}{\cup} \pi^{-1}_i(V_{i,\,0})$
that the family $\{\kappa_j\}_{j=1}^\ell$ determines
a $(p,n-1)$-form on $V_{01}$, which is denoted by $\kappa_{01}$.
We also define a $(p,n)$-form $\kappa_1$ on $V_1$ by
\[
%\begin{equationth}
	\kappa_1 = \tau_{1,\,1} \,{\wedge}\, \tau_{2,\,1} \,{\wedge}\, \cdots
	\,{\wedge}\, \tau_{\ell,\,1}.
%\end{equationth}
\]
\begin{proposition}
Thus constructed
\[
\kappa = (\kappa_1,\kappa_{01}) \in
\scE^{(p,n)}(V_1) \oplus
\scE^{(p,n-1)}(V_{01})
= \scE^{(p,n)}(\V,\V')
\]
is a representative of the external product 
$u = u_1 \times u_2 \times \cdots 
\times u_\ell \in \SHhyperf{(p)}(M)$.
\end{proposition}
\begin{proof}
This formula is obtained by the cup product formula, and then, by
repeated applications of the remark after Lemma {\ref{lemma:fundamental_equiv}} 
in Appendix.
Note that Lemma {\ref{lem:product_gluing}} is an immediate consequence of
this procedure.
\end{proof}

The above expression appears to be rather complicated, however, 
it becomes much simpler in
some particular but important cases\,:
\begin{example}
Assume all the $X_k$'s are Stein. Then
we may take, for each $k$, a representative $(\tau_{k,\,1},\,\tau_{k,\,{01}})$ of $u_k$ 
so that 
$\tau_{k,\,1} = 0$.
%for all $k$. 
%In this case, a 
A representative of $u$ is then
given by
\[
%\begin{equationth}
(0,\,\, (-1)^e\,(\ell-1)!\,
\bar{\partial}\varphi_1 \wedge \dots \wedge
\bar{\partial}\varphi_{\ell-1} \wedge
\tau_{1,\,{01}}\,{\wedge} \,\cdots \,{\wedge}\, \tau_{\ell,\,01})
\in \scE^{(p,n)}(\V,\V').
\]
%\end{equationth}
Here the constant $e$ is 
$\frac{\ell(\ell -1)}{2} +
%\displaystyle
\sum_{k=1}^{\ell-1} (\ell - k)(n_k+p_k)$.

For example, the $n$-dimensional Dirac's delta function $\delta(x)$ 
is just the external product
$\delta(x_1) \times \delta(x_2) \times \cdots \times \delta(x_n)$
of the ones on $\R$. Hence, its representative is given by the
above formula using a representative of the one-dimensional Dirac's delta function, namely,
\[
\left(0,\, \dfrac{1}{(2\pi\sqrt{-1})^n}
\dfrac{(n-1)!\,\bar{\partial}\varphi_1\wedge \dots \wedge
\bar{\partial}\varphi_{n-1}}{z_1 \cdots z_n}\right)
\in \scE^{(0,n)}(\V,\V').
\]
\end{example}
\begin{example}
In the case $\ell = 2$,
a representative of 
$u = u_1 \times u_2 \in \SHhyperf{(p)}(M)$ is
given by
\[
%\begin{equationth}
\begin{aligned}
\big(\tau_{1,\,1} \,{\wedge}\, \tau_{2,\,1},\,\,
\varphi_1\, \tau_{1,\,{01}} \,{\wedge}\, \tau_{2,\,1} 
& + 
(-1)^{n_1+p_1} \varphi_2 \,\tau_{1,\,{1}}\, {\wedge}\, \tau_{2,\,{01}} \\
&\qquad -(-1)^{n_1+p_1}
\bar{\partial}\varphi_1 \,{\wedge}\, \tau_{1,\,{01}}\, \wedge\, \tau_{2,\,{01}}\big)
\end{aligned}
\in \scE^{(p,n)}(\V,\V').
\]
%\end{equationth}
\end{example}

We can easily show the following two propositions\,:
\begin{proposition}
We have, for $u_k \in \scB(M_k)$, $k=1,\dots,\ell$,
\[
\op{SS}(u_1 \times \cdots \times u_\ell) = 
\op{SS}(u_1) \times \cdots \times \op{SS}(u_\ell).
\]
\end{proposition}
\begin{proposition} For each $k=1,\dots,\ell$,
let $\Omega_k$  be an open subset of $X_k$ satisfying
the conditions $(\rm{B}_1)$~and $(\rm{B}_2)$~in $X_k$, and let $f_k \in \scO(\Omega_k)$.
Then we have
\[
b_{\Omega_1 \times \dots \times \Omega_\ell}(f_1 f_2 \cdots f_\ell)
=
b_{\Omega_1}(f_1) \times
b_{\Omega_2}(f_2) \times
\cdots  \times 
b_{\Omega_\ell}(f_\ell).
\]
\end{proposition}

\subsection{Restriction of hyperfunctions}{\label{ssec:restrict}}
Let $N$ be a closed real analytic hypersurface in $M$ and $Y$ its complexification
in $X$. It is known that the restriction to $N$ of a hyperfunction $u$ on $M$ cannot be defined in general.
%is generally impossible, h
However, if $\op{SS}(u)$ is non-characteristic to $N$,
i.e., $\op{SS}(u) \cap \sqrt{-1}T^*_NM \subset T^*_XX$ holds, then
we can consider its restriction to $N$. In this subsection, we will define
the restriction of a hyperfunction from the viewpoint of  relative Dolbeault representation.
We assume that $M$ and $N$ are oriented. Then we can take
a non-vanishing continuous section 
\[
	s: N \to T_N^*M \ssm T^*_XX.
\]
Note that, when $N$ is connected, 
there are essentially two choices of $s$, i.e., either $s$ or $-s$.
For such a choice, by noticing the morphisms of vector bundles
\[
0 \lra T^*_NM \lra N \times_{M}T^*M \lra T^*N \lra 0,
\]
we determine it so that, for any $x_0 \in N$, 
the vector $s(x_0)$ and a positively oriented frame of $(T^*N)_{x_0}$ 
form that of $(T^*M)_{x_0}$, where the frame of $(T^*N)_{x_0}$ follows $s(x_0)$.

Let $t: N \to T_N^*M$ be a continuous section on $N$ and $G$ a closed set in $X$.
\begin{definition}
We say that $G$ is conically contained in the half space spanned by $\sqrt{-1}t$
if, for any 
%point 
$x_0 \in N$, there exist an open neighborhood $W$ of $x_0$ 
with a ($C^1$-class) local trivialization 
$\iota: (M\cap W,\,W) \simeq (\R_x^n,\,\C^n = \R^n_x \times \sqrt{-1}\R^n_y)$ 
and a closed cone $\Gamma \subset \R_y^n$ with
\[
%\begin{equationth}
\Gamma \ssm\{0\} \subset 
\{\,y \in \R_y^n\mid\, \op{Re}\,\langle \sqrt{-1}y,\, \sqrt{-1}t(x_0) \rangle > 0\,\}
= \{\,y \in \R_y^n\mid\, \langle y,\, t(x_0) \rangle < 0\,\}
\]
%\end{equationth}
for which the following holds\,:
\[
%\begin{equationth}
\iota(G \cap W) \subset \R_x^n \times \sqrt{-1}\Gamma.
\]
%\end{equationth}
%holds.
\end{definition}

Recall that, for a closed set $G$, we set
%\[
%\begin{equationth}
$\W_{V,G} = \{V \ssm (\R^n_x \times \sqrt{-1}G),\, V\}$ and
%\quad \text{and} \quad
$\W_{V,G}' = \{V \ssm (\R^n_x\times \sqrt{-1} G)\}$.
%\]
%\end{equationth}
We also set
%\[
%\begin{equationth}
$\W_V = \{V_0=V\ssm M,\,V_1= V\}$ and
%\quad \text{and} \quad
$\W'_V = \{V_0\}$.
%\]
%\end{equationth}
Then we have a global version of Proposition {\ref{prop:equiv-ss}}\,:
\begin{lemma}
Let $u$ be a hyperfunction on $M$. 
Assume $\op{SS}(u) \cap \sqrt{-1}s = \emptyset$. Then there exist an open neighborhood
$V$ of $N$ in X, a closed set $G$ 
which is conically contained in the half space spanned
by $\sqrt{-1}s$ and an element
\[
\tau = (0,\tau_{01}) 
\in \scE^{(0,n)}(V_1) \oplus
\scE^{(0,n-1)}(V_{01})
=\scE^{(0,n)}(\W_{V},\W'_{V}) 
\]
for which $\tau$ is a representative of $u$ near $N$ and 
$\op{Supp}_{V\ssm M}(\tau_{01}) \subset G$ holds.
\end{lemma}
\begin{proof}
Set $\sqrt{-1}T_N^*M^+ = \sqrt{-1}\,\R^+ s$.
Then it follows from  \cite[Theorem 4.3.2]{KS} that 
%we have
\[
\SHmicro(\sqrt{-1}T_N^*M^+) = \underset{\underset{V,\,G}{\lra}}{\lim}\,
H^n_{G}(V;\scO),
\]
where $V$ runs through open neighborhoods of $N$ and
$G$ ranges through closed sets conically contained in the half space spanned by $\sqrt{-1}s$.
Therefore the argument goes  the same way as that in
Subsection {\ref{subsec:micro-anal}}.
\end{proof}

We first give the cohomological definition of restriction to $N$ of 
a hyperfunction on $M$.
Set $T^*_NM^\pm = \pm\R^+{s} \subset T^*_NM$.
Let us consider the map
\[
	\scB(M) \xrightarrow{\op{sp}|_{T^*_NM^+} \oplus\,
\op{sp}|_{T^*_NM^-}}
\SHmicro(\sqrt{-1}T^*_NM^+) \oplus
\SHmicro(\sqrt{-1}T^*_NM^-).
\]
By  \cite[Theorem 4.3.2]{KS}, the above diagram is equivalent to
\[
H^n_M(X;\scO)
\lra
\underset{\underset{V,\,G^+}{\lra}}{\lim}\,
H^n_{G^+}(V;\scO)
\oplus
\underset{\underset{V,\,G^-}{\lra}}{\lim}\,
H^n_{G^-}(V;\scO),
\]
where $V$ runs through open neighborhoods of $N$ in $X$ and
$G^+$ (resp. $G^-$) ranges through closed sets
conically contained in the half space spanned by $\sqrt{-1}s$ (resp. $-\sqrt{-1}s$), and
the morphism is just the canonical restriction. 
Furthermore, we may assume that $(G^+ \cap G^{-}) \cap V = M \cap V$ holds.
Then we have
(see the proof of the Lemma {\ref{lem:restriction-independent-of-choice}} below)
\[
\underset{\underset{V,\,G^\pm}{\lra}}{\lim}\,
H^{k}_{G^\pm}(V;\, \scO) = 0\qquad\text{for}\ \  k < n.
\]
From the Mayer-Vietoris sequence 
for the pair $(G^+,G^-)$, 
we have the exact sequence
\[
0\lra 
\underset{\underset{V,\,G^\pm}{\lra}}{\lim}\,
H^{n-1}_{G^+ \cup G^-}(V;\scO)
\lra H^n_{M}(V;\scO)
\lra
\underset{\underset{V,\,G^+}{\lra}}{\lim}\,
H^n_{G^+}(V;\scO)
\oplus
\underset{\underset{V,\,G^-}{\lra}}{\lim}\,
H^n_{G^-}(V;\scO).
\]
Note that the morphisms of the above Mayer-Vietoris sequence 
\[
H^k_{G^+ \cap G^-}(V;\scO)
\lra
H^k_{G^+}(V;\scO)
\oplus
H^k_{G^-}(V;\scO)
\lra
H^k_{G^+ \cup G^-}(V;\scO)
\]
are defined by sending $u$ to $u \oplus u$ and $u \oplus v$ to $u - v$, respectively,
for which the choices of sign, i.e., either $u-v$ or $v - u$, is determined by 
taking our choice of the orientation of fibers of $T^*_NM$ into account.
Let $i:Y \hra X$ be the closed embedding. Then we have the canonical sheaf morphism
%\[
$\scO_X \to i_*\scO_Y$,
%\]
which induces the morphism
\[
H^{n-1}_{G^+ \cup G^-}(V;\scO_X)
\lra
H^{n-1}_{G^+ \cup G^-}(V;i_*\scO_Y)
=
H^{n-1}_{N}(V \cap Y;\scO_Y)
= \scB(N)
\]
because of $i^{-1}(G^+ \cup G^-) = N$.
Summing up, we have the diagram with the exact row

\[
\SelectTips{cm}{}
\xymatrix
%@C=.5cm
%@R=.6cm
{0\ar[r]&\underset{\underset{V,\,G^\pm}{\lra}}{\lim}\,
H^{n-1}_{G^+ \cup G^-}(V;\scO)\ar[r]\ar[d]& \scB(M)\ar[r]^-{\op{sp} \oplus \op{sp}}&\SHmicro(\sqrt{-1}T^*_NM^+) \oplus
\SHmicro(\sqrt{-1}T^*_NM^-)\\
{}&\scB(N).}
\]
Therefore, if a hyperfunction $u \in \scB(M)$ satisfies
$\op{SS}(u) \cap \sqrt{-1}T^*_NM \subset T^*_XX$, which
is equivalent to saying that the image of $u$ is zero by the morphism
$\op{sp}|_{T^*_NM^+}\oplus\op{sp}|_{T^*_NM^-}$,
then we have the unique hyperfunction $u|_N$ in $\scB(N)$
by tracing the above diagram.

\

Now we compute a concrete representative of $u|_N$ in our framework.
Assume that $\op{SS}(u) \cap \sqrt{-1}T^*_NM \subset T^*_XX$.
Let $\tau = (\tau^n_1,\, \tau^{n-1}_{01})
\in \scE^{(0,n)}(\W,\W')$
be a representative of $u$ near $N$.
Then, by the assumption and the above lemma,
there exist an open neighborhood $V$ of $N$, a closed set $G^+$ (resp. $G^-$)
conically contained in the half space spanned by $\sqrt{-1}s$ (resp. $-\sqrt{-1}s$) and
representatives
\[
\tau^{n-1,\pm} = (\tau_1^{n-1,\pm},\, \tau_{01}^{n-2,\pm})
\in \scE^{(0,n-1)}(\W_{V,G^\pm},\W'_{V,G^\pm})
\]
%\[
%\left(\text{resp.}\,\,\, \tau^{n-1,-} = (\tau_1^{n-1,-},\, \tau_{01}^{n-2,-}\right)
%\in \scE^{(0,n-1)}(\W_{V,G^-},\W'_{V,G^-}))
%\]
such that
$\tau = \bar{\vartheta} \tau^{n-1,\pm}$ in 
$\scE^{(0,n)}(\W_{V,G^\pm},\W'_{V,G^\pm})$.
%\[
%\left(\text{resp. }\,\,\,
%\tau = \bar{\vartheta} \tau^{n-1,-} \qquad \text{in $\scE^{(0,n)}(\W_{V,G^-},\W'_{V,G^-})$}\right).
%\]
Define
\[
%\begin{equationth}
	\tau_Y = (\tau^{n-1,-} - \tau^{n-1,+})|_Y  \\
=
\left( (\tau_1^{n-1,-} - \tau_1^{n-1,+})|_Y,\, (\tau_{01}^{n-2,-} - \tau_{01}^{n-2,+})|_Y
\right).
\]
%\end{equationth}
Here $\bullet|_Y$ denotes the restriction of a differential form to $Y$.
Note that the choices of sign, i.e., 
either $(\tau^{n-1,-} - \tau^{n-1,+})|_Y$ or
$(\tau^{n-1,+} - \tau^{n-1,-})|_Y$, is a consequence of that in
the Mayer-Vietoris sequence.  
If $V$ is a sufficiently small neighborhood of $N$, 
then we have
\[
(V \ssm G^+) \cap Y = (V \ssm N) \cap Y \quad\text{and}\quad
(V \ssm G^-) \cap Y = (V \ssm N) \cap Y.
\]
Hence $\tau_Y$ belongs to $\scE^{(0,n-1)}(\W_{V\cap Y},\W'_{V\cap Y})$ with coverings
\[
\W_{V \cap Y} = \{(V\cap Y) \ssm N,\, V \cap Y\},\qquad
\W'_{V \cap Y} = \{(V\cap Y) \ssm N\}.
\]
Furthermore we have
\[
\bar{\vartheta} \tau_Y =
\bar{\vartheta} ((\tau^{n-1,-} - \tau^{n-1,+})|_Y) =
(\bar{\vartheta} \tau^{n-1,-} - \bar{\vartheta}\tau^{n-1,+})|_Y =
(\tau - \tau)|_Y = 0,
\]
which implies that the representative $\tau_Y$ defines a hyperfunction on $N$.
\begin{definition}
	The hyperfunction on $N$ defined by $\tau_Y
\in \scE^{(0,n-1)}(\W_{V\cap Y},\W'_{V\cap Y})$ 
is denoted by $u|_N$ and is called the restriction of $u$ to $N$.
\end{definition}
\begin{lemma}{\label{lem:restriction-independent-of-choice}}
The restriction $u|_N$ is well-defined, that is, $u|_N$ does not depend on
the choice of $\tau$, $\tau^{n-1,+}$ or $\tau^{n-1,-}$ 
in the above construction.
\end{lemma}
\begin{proof}
Recall that we set $\sqrt{-1}T_N^*M^+ = \sqrt{-1}\,\R^+ s$.
Clearly, by construction, $u|_N$ is independent of the choice of $\tau$.
By the same argument as that in the proof for Lemma~\ref{lem:boundary_map_ok},
the independency of the choice of $t^{n-1,+}$ comes from the fact
\[
\underset{\underset{V,\,G^+}{\lra}}{\lim}\,
H^{0,n-1}_{\bar\vt}(\W_{V,G^+},\W'_{V,G^+})
= 0,
\]
which can be shown in the following way:
Thanks to the edge of the wedge theorem for $\scO$ and  \cite[Theorem 4.3.2]{KS},
%and flabbiness of $\SHmicro$, 
we have the formula, for any $k \in \mathbb{Z}$,
\[
H^k(\sqrt{-1}T_N^*M^+;\SHmicro) = 
\underset{\underset{V,\,G^+}{\lra}}{\lim}\,
H^{n+k}_{G^+}(V;\scO),
\]
where $V$ runs through  open neighborhoods of $N$ in $X$ and $G^+$ ranges through
closed sets conically contained in the half space  spanned by $\sqrt{-1}s$.
From the above we have
\[
0=H^{k-n}(\sqrt{-1}T_N^*M^+;\SHmicro) = 
\underset{\underset{V,\,G^+}{\lra}}{\lim}\,
H^{k}_{G^+}(V;\scO) 
= 
\underset{\underset{V,\,G^+}{\lra}}{\lim}\,
H^{0,k}_{\bar\vt}(\W_{V,G^+},\W'_{V,G^+})
\qquad\text{for}\ \  k < n.
\]
Hence $u|_N$ is independent of the choice of $t^{n-1,+}$.

The independency of the choice of $t^{n-1,-}$ can be proved  the same way.
%This completes the proof.
\end{proof}
The following theorem assures that our construction coincides with the original
one in \cite{Skk}.
Let $V$ be an open neighborhood of $N$ in $X$ and $\Omega$ an open subset of $X$.
Set
\[
\Omega_Y = \Omega \cap Y,\qquad
V_Y = V \cap Y.
\]
Before stating the theorem, we introduce two conditions $(\rm{B}_1^{\dag\dag})$~and 
$(\rm{B}_2^{\rm{loc}})$:
The condition $(\rm{B}_1^{\dag\dag})$~is 
%just 
the one 
$(\rm{B}_1^{\dag})$~given in Appendix
with a non-characteristic condition of $\Omega$ along $N$.

\begin{enumerate}
\setlength{\leftskip}{.6mm}
	\item[$(\rm{B}_1^{\dag\dag})$] For any $x_0 \in N$, there exist an open neighborhood  $W$ of $x_0$
	with a ($C^1$-class) local trivialization 
	$\iota: (N \cap W,\, M \cap W,\,W) \simeq 
	(\R_{x'}^{n-1},\, \R_x^n,\, \C^n)$ and a non-empty open convex cone $\Gamma \subset \R^n$ such that
\[
	\R_x^n \times \sqrt{-1} \Gamma \subset \iota(W \cap \Omega)
	\quad\text{and}\quad
	\R_{x'}^{n-1} \cap \Gamma \ne \emptyset.
\]
\end{enumerate}
The condition $(\rm{B}_2^{\rm{loc}})$~is a local version of $(\rm{B}_2)$~introduced in Subsection \ref{ssbv}.
\begin{enumerate}
\setlength{\leftskip}{1mm}
	\item[$(\rm{B}_2^{\rm{loc}})$] For any $x_0 \in M$, there exist a fundamental system
%family 
$\{V_\lambda\}_{\lambda \in \Lambda}$ of 
%fundamental open 
neighborhoods of $x_0$ in $V$
that satisfies the same condition as $(\rm{B}_2)$ with
$V$  replaced by $V_\lambda$, that is, the canonical inclusion
$(V_\lambda \ssm \Omega) \ssm M  \hra V_\lambda \ssm \Omega$
is a homotopy equivalence for any $\lambda \in \Lambda$.
\end{enumerate}

\begin{theorem}
Assume that the pair $V$ and $\Omega$ satisfies the conditions 
$(\rm{B}_1^{\dag\dag})$, $(\rm{B}_2)$~and $(\rm{B}_2^{\rm{loc}})$. 
Assume also that the pair $V_Y$ and $\Omega_Y$ satisfies the conditions
$(\rm{B}_2)$~and $(\rm{B}_2^{\rm{loc}})$~in $Y$. Let $f \in \scO(\Omega)$. Then we have
\[
\op{SS}(b_\Omega(f)) \cap \sqrt{-1}T^*_NM \subset T^*_XX
\]
and
\[
b_\Omega(f)|_N = b_{\Omega_Y}(f|_Y).
\]
\end{theorem}
\begin{proof}
By 
%the condition 
$(\rm{B}_2^{\rm{loc}})$, it suffices to show the claim locally.
We may assume
$M$ is a convex open neighborhood of 
%the origin in 
$0\in\R_x^n$ with coordinates
$(x_1,\dots,x_n) = (x_1,x')$,
$X = M \times \sqrt{-1}\R_y^n$ with coordinates
$(z_1,\dots,z_n) = (z_1, z')$, 
%where 
$z_k = x_k + \sqrt{-1}y_k$, $k=1,\dots,n$,
$N=M \cap \{x_1 = 0\}$ and $Y$ = $X \cap \{z_1 = 0\}$.
We write $(y_{1},\dots,y_{n})=(y_{1},y')$.
Further, we assume
\[
V = \{\,z \in X\mid\, |x_1| < \epsilon,\, |y| < \epsilon\,\}
\quad\text{and}\quad \Omega = (M \times \sqrt{-1} \Gamma) \cap V
\]
for some $\epsilon > 0$ and an open proper convex cone $\Gamma \subset \R_y^n$ with 
$\Gamma \cap \{y_1 = 0\} \ne \emptyset$.

Let $\mathrm{e} = (1,0,\dots,0) \in \R^n$.
Then we may assume that our coordinate systems of $N$ and $M$ are
positively oriented and
the section $s$ is given by $s(x) = \mathrm{e} \in (T^*_NM)_x$ 
$(x \in N)$.
Now we determine,
in the similar way as those in Example {\ref{exa:support_cone_fundamental}},
convex subsets $H_k$'s in $\R^n_y$ and
$C^\infty$ functions $\varphi_k$'s on $X \ssm M$ where the index $k$ 
is either $k=\pm$ or $k=1,2,\dots,n$.
We first take linearly independent vectors 
$\tilde{\eta}_1$, $\dots$, $\tilde{\eta}_{n-1}$ in $\R_{y'}^{n-1}$
such that
\[
\bigcap_{1 \le j \le n-1}
\{\,y' \in \R_{y'}^{n-1}\mid \langle y',\, \tilde{\eta}_j\rangle > 0\,\}\,\,
\subset\,\, \Gamma \cap \{y_1 = 0\}.
\]
Set $\tilde{\eta}_n = - (\tilde{\eta}_1 + \dots + \tilde{\eta}_{n-1}) \in \R^{n-1}_{y'}$.
Let $c > 0$ and define convex subsets $H_j$, $j=1,2,\dots,n$, in $\R^n_y$ by
\[
H_j =
\{\,y \in \R^n_y\mid\,
-c\langle y', \tilde{\eta}_j \rangle < y_1 < c\langle y', \tilde{\eta}_j \rangle\,\}.
\]
We also define convex subsets $H_+$ and $H_-$ in $\R^n_y$ by
\[
H_\pm = \{\,y \in \R^n_y\mid\, \pm y_1 > c^2|y'| \,\}.
\]
Note that, for a sufficiently small $c > 0$, we have
\[
H_1 \cup \cdots \cup H_n \cup H_+ \cup H_- = \R^n_y \ssm \{0\}\quad\text{and}\quad  H_1 \cap \cdots \cap H_{n-1} \subset \Gamma.
\]
%and
%\[
%H_1 \cap \cdots \cap H_{n-1} \subset \Gamma.
%\]
We fix such a $c > 0$ in what follows. Note also that 
the intersection of $(n+1)$-choices 
in $(n+2)$-subsets $H_1$, $\dots$, $H_n$, $H_+$ and $H_-$
is always empty.

Now let $\varphi_k$, $k=\pm, 1,\dots,n$, be $C^\infty$ functions on $V \ssm M$
such that
\begin{enumerate}
\item[(1)] $\op{Supp}_{V\ssm M}(\varphi_k) 
	\subset M \times \sqrt{-1}H_k$ for $k=1,\dots,n,\pm$,
\item[(2)] $\varphi_1 + \cdots + \varphi_n + \varphi_+ + \varphi_- = 1$.
\end{enumerate}
In particular, it follows from the definition of $H_k$ that
$\varphi_\pm|_Y = 0$ and
\begin{equationth}{\label{eq:identity_on_Y}}
\varphi_1|_Y + \cdots + \varphi_n|_Y = 1
\end{equationth}
holds. Set
\[
\kappa^n = (0,\,\kappa^{n-1}_{01})
 = (0,\,\,
(-1)^n(n-1)!\, \check{\chi}_{H_{n} \cup H_-}\,
	 \bar{\partial}\varphi_+ \wedge 
	 \bar{\partial}\varphi_1 \wedge 
	 \cdots \wedge \bar{\partial}\varphi_{n-2}),
\]
where $\check{\chi}_{H_n \cup H_-}$ is the anti-characteristic function
of the set $H_n \cup H_-$.
Then, by the same arguments as those in Example {\ref{exa:support_cone_fundamental}}
and Lemma \ref{lem:tau_is_generator}, we see that
\[
\op{Supp}_{V\ssm M}(\kappa^{n-1}_{01})
\subset M \times \sqrt{-1} \Gamma
\]
and that
$[\kappa^n] \in H^n_M(X;\,\scO)$
corresponds to the image of the positively oriented generator in $or_{M/X}(M)$ 
%by $\rho$ 
under the standard orientation of $\R^n_y$.
Hence, by the definition of the boundary value map, 
$\tau = f\,\kappa^n = (0,\, f\,\kappa^{n-1}_{01})$
is a representative of $b_\Omega(f)$.
Then, as
\[
\op{Supp}_{X\ssm M}(\kappa^{n-1}_{01}) \subset H^+
\]
holds, we can take $\tau^{n-1,-} = 0$ in the construction of $\tau|^{n-1}_N$. Let us compute $\tau^{n-1,+}$.
Set
\[
\kappa^{n-2}_{01} =
-(n-2)!\,
\check{\chi}_{H_{n} \cup H_-}\,
\displaystyle\sum_{j=1}^{n-1} (-1)^{j} \varphi_{j}\,
%\overset{\overset{(j-1)\,\text{st}}{\vee}}{
\bar{\partial} \varphi_1 \wedge \dots \wedge\widehat{\bar{\partial} \varphi_j}\wedge\cdots
\wedge\bar{\partial} \varphi_{n-1}.
%}
\]
Then there exists a closed set $G^+$
conically contained in the half space spanned by $\sqrt{-1}s$ such that
$\kappa^{n-2}_{01}$ is a $C^{\infty}$ $(0,n-2)$-form on $V \ssm G^+$.
Define
\[
\tau^{n-1,+} = (0,\, f\,\kappa^{n-2}_{01}).
\]
Then, since $\kappa_{01}^{n-1} = - \bar{\partial} \kappa_{01}^{n-2}$ holds,
we see
\[
\tau = \bar{\vartheta} \tau^{n-1,+} \quad \text{in $\scE^{(0,n-1)}(\W_{V,G^+},\W'_{V,G^+})$}.
\]
Furthermore, it follows from \eqref{eq:identity_on_Y} that we have
\[
\kappa^{n-2}_{01}|_Y = (-1)^n(n-2)!\, \check{\chi}_{H_{n}\cap Y}\,
 \bar{\partial}(\varphi_1|_Y) \wedge \dots \wedge \bar{\partial}(\varphi_{n-2}|_Y),
\] 
for which, by Example \ref{exa:support_cone_fundamental}, the element
$[(0,\,-\kappa_{01}^{n-2}|_Y)] \in H^{n-1}_N(Y;\scO)$ corresponds to
the image of the positively oriented generator in $or_{N/Y}(N)$% by $\rho$
under the standard orientation of $\R^{n-1}_{y'}$.
Therefore we have obtained
\[
\begin{aligned}
\tau_Y &= (\tau^{n-1,-} - \tau^{n-1,+})|_Y \\
&= (0, f|_Y\, (-1)^{n-1}(n-2)!\, \check{\chi}_{H_{n} \cap Y}\,
 \bar{\partial}(\varphi_1|_Y) \wedge \dots \wedge \bar{\partial}(\varphi_{n-2}|_Y)).
 \end{aligned}
\]
This implies that the representative $\tau_Y$ gives the
hyperfunction $b_{\Omega_Y}(f|_Y)$.
%, which completes the proof.
\end{proof}

\subsection{Fiber integration of  hyperfunction densities}{\label{subsection:integration}}

Let $M$ and $N$ be real analytic manifolds of dimensions $m$ and $n$, \r, $m\ge n$.
% ($\dim\, M = m$ and  $\dim\, N= n$), 
%Also let 
We denote by $X$ and $Y$  their complexifications.
%respectively. 
Let $f:M \to N$ be a submersion whose complexification 
is 
%also 
denoted by $\tilde f: X \to Y$.
% for simplicity.
%Note that, in what follows, 
%the restriction $f|_M$ of $f$ to $M$ is denoted by the same symbol $f$
%if there is no fear of confusion.
We assume that $M$ and $N$ are  orientable.  

In this subsection, 
we introduce an integration morphism $\int_f$ of  hyperfunction densities along 
fibers of $f:$
\[
f_!(\scB_M {\otimes}_{{}_{\scA_{M}}} \scV_M)(N) 
\xrightarrow{\quad\int_{f}\quad}
(\scB_N {\otimes}_{{}_{\scA_{N}}} \scV_N)(N),
\]
where $\scV_M$ is the sheaf of real analytic densities on $M$, i.e.,
%\[
$\scV_M = \SHanalf{(m)}_M {\otimes}_{{}_{\Z_{M}}} {or}_M$
%\]
with  $\SHanalf{(m)}_M$ the sheaf of real analytic $m$-forms on $M$
and  ${or}_M$ the orientation sheaf of $M$. The sheaf $\scV_N$ is defined similarly for $N$. Note that $\scB_M {\otimes}_{{}_{\scA_{M}}} \scV_M=\scB^{(m)}_M {\otimes}_{{}_{\Z_{M}}} or_M$. Note also that, as the integration we define is performed
on densities, the morphism $\int_f$ is independent of the orientation of fibers of $f$.
%,  that is, the result remains unchanged if we change the orientation of fibers of $f$.

We may
%, hereafter, 
assume
each fiber of $f$ to be connected.
Since otherwise, $M$ is a disjoint union of real analytic manifolds $M_\lambda$
%,  $\lambda \in \Lambda$,
such that 
%$M_\lambda$ and 
each fiber of $f|_{M_\lambda}$
%:M_\lambda \to N$ 
is connected. 
Then 
the integrations along  $f|_{M_\lambda}$
% $(\lambda \in \Lambda)$ 
can be summed up, as  the support of the integrand is proper along $f$.
%For simplicity, we also assume that $M$ and $N$ are connected so that $\Z_{M}(M)=\Z$
%and $\Z_{N}(N)=\Z$.

Set $d= m - n$ and let $p$ be a non-negative integer.
We will construct slightly extended version of
the integration morphism\,:
\[
	f_!(\SHhyperf{(p)}_M \otimes_{{}_{\Z_{M}}} {or}_M)(N)
\xrightarrow{\quad\int_{f}\quad}
(\SHhyperf{(p-d)}_N \otimes_{{}_{\Z_N}} or_N)(N),
\]
where $\SHhyperf{(p)}_M$ is the sheaf of $p$-hyperforms on $M$. Note that
$\SHhyperf{(p)}_M = \scB_M {\otimes}
_{{}_{\scA_{M}}} \SHanalf{(p)}_M$.

Since $M$ is assumed to be orientable, 
for any closed set $K$ in $M$,
\[
	\vG_K(M; \SHhyperf{(p)}_M)=H^m_K(X;\scO^{(p)}_X)\otimes_{{}_{Z_{M}(M)}}
{or}_{M/X}(M).
\]
Hence it suffices to construct the following integration morphism
for any closed set $K \subset M$ with $f|_K: K \to N$ being proper\,:
\[
	H^m_K(X;\scO^{(p)}_X)\otimes_{{}_{Z_{M}(M)}}
	({or}_{M/X}{\otimes}_{{}_{\Z_{M}}} {or}_{M})(M)
\xrightarrow{\,\,\int_{f}\,\,}
H^n_N(Y;\,\scO^{(p-d)}_Y)
\otimes_{{}_{Z_{N}(N)}}
({or}_{N/Y}
\otimes_{{}_{\Z_N}} {or}_N)(N).
\]

Denoting by $i:M\hra X$ and $j:N\hra Y$ the  inclusions, we have the 
identifications by Convention 1\,:
\begin{equationth}\label{id1}
{or}_{M/X}{\otimes}_{{}_{\Z_{M}}} {or}_{M}= i^{-1}{or}_{X},\qquad {or}_{N/Y}{\otimes}_{{}_{\Z_{N}}} {or}_{N}= j^{-1}{or}_{Y}.
\end{equationth}
With these, the above integration morphism is expressed as
\begin{equationth}\label{intmor}
	H^m_K(X;\scO^{(p)}_X)\otimes_{{}_{Z_{M}(M)}}i^{-1}{or}_{X}(M)
\xrightarrow{\,\,\,\,\int_{f}\,\,\,\,}
H^n_N(Y;\,\scO^{(p-d)}_Y)
\otimes_{{}_{Z_{N}(N)}}j^{-1}{or}_{Y}(N).
\end{equationth}
%%%%

Before we further proceed, we recall the usual fiber integration of $C^{\infty}$ forms with
some more orientation convention.

%\paragraph{Fiber integration of $C^{\infty}$ forms\,:} 
In general, let $g:P\ra Q$ be a submersion of $C^{\infty}$ \mfd s with fiber dimension $r$. We denote by $Tg$
the bundle of tangent vectors on $P$ that are tangent to the fibers of $g$. We define the orientation sheaf  of $g$, denoted by $or_{g}$, to be the orientation sheaf of the bundle $Tg$ (cf. Subsection~\ref{subsOT}). It is a sheaf on $P$ and describes the orientation of fibers of $g$. From the 
exact sequence
\[
0\lra Tg\lra TP\lra g^{*}TQ\lra 0,
\]
%first horizontal sequence in \eqref{diagsubm}, 
we have  an \iso 
%\begin{equation}\label{orisubm}
\[
{or}_P \simeq {or}_{g}{\otimes}_{{}_{\Z_{P}}}g^{-1}{or}_Q.
%,\qquad {or}_{f} \simeq f^{-1}{or}_N {\otimes}_{{}_{\Z_{M}}} {or}_M.
\]
%{or}_M \simeq {or}_{f}{\otimes}_{{}_{\Z_{M}}} f^{-1}{or}_N,\qquad {or}_{f} \simeq {or}_M{\otimes}_{{}_{\Z_{M}}} f^{-1}{or}_N.
%\end{equation}
Note that there are several possible choices for the above \iso s. 

Now we assume that $P$ and $Q$ are orientable. Thus $Tg$ is also orientable.

\paragraph{Convention 2.} Let $\psi_{P}$ and $\psi_{Q}$ be prescribed orientations of $P$ and $Q$.
Then we take the orientation $\psi_{g}$ of $T{g}$ so that a positive fiber coordinate system
followed by (the pull-back by $g$ of) a positive coordinate system on $Q$ gives a positive 
coordinate system on $P$.

\

Thus we make identification
\begin{equationth}\label{oriidsubm}
or_{P}= or_{g}\otimes g^{-1}or_{Q}\qquad\text{by}\qquad \psi_{P}\longleftrightarrow\psi_{g}\otimes
g^{-1}\psi_{Q}
\end{equationth}
with  $\psi_{P}$, $\psi_{Q}$ and  $\psi_{g}$ as in Convention 2.

Once we fix various orientations according to Convention 2, we have the fiber integration
\[
%\int_{g}:
(g_{!}\scE_{P}^{(p)})(Q)\xrightarrow{\,\,\,\,\int_{g}\,\,\,\,} \scE_{Q}^{(p-r)}(Q),
\]
which has the property
%with the property
%\[
$\int_{g}\circ \,d=(-1)^{r}d\circ\int_{g}$. Thus it induces a morphism of complexes 
%\[
$g_{!}\scE_{P}^{(\bullet)}[\dim P]\ra
\scE_{Q}^{(\bullet)}[\dim Q]$.
%\]

If $P$ and $Q$ are complex \mfd s and if $g$ is a \h\ submersion with complex fiber dimension $s$, the 
above integration induces
\[
%\int_{g}:
(g_{!}\scE_{P}^{(p,q)})(Q)\xrightarrow{\,\,\,\,\int_{g}\,\,\,\,} \scE_{Q}^{(p-s,q-s)}(Q),
\]
which commutes with both $\partial$ and $\bar\partial$.
%%%%%%%

\

Coming back to our situation, we have two submersions, namely, $f:M\ra N$ and its complexification $\tilde f:X\ra Y$. We make the following identifications according to Convention 2\,:
\begin{equationth}\label{id2}
{or}_M={or}_{f}{\otimes}_{{}_{\Z_{M}}} f^{-1}{or}_N,\qquad {or}_X = {or}_{\tilde f}{\otimes}_{{}_{\Z_{X}}} \tilde f^{-1}{or}_Y.
\end{equationth}
We may
 make the identifications \eqref{id1} and \eqref{id2} compatible, by determining various coordinate
systems  as follows\,:
\begin{enumerate}
\item[(1)] Let $\psi_{M}$ and $\psi_{N}$ be prescribed orientations of $M$ and $N$.
%, \r.
We orient the fibers of $f$ so that if $x'=(x_{1},\dots,x_{d})$ is a positive fiber coordinate system, $d=m-n$,
and if
$x''=(x_{d+1},\dots,x_{m})$
 is a positive
coordinate system on $N$, then 
%$(x_{1},\dots,x_{n},t_{1},\dots,t_{d})$
%$x=(x_{1},\dots,x_{n},x_{n+1},\dots,x_{m})$ 
$x=(x',x'')$ is a positive coordinate system on $M$.

\item[(2)] The orientations $\psi_{X}$ and $\psi_{Y}$ of $X$ and $Y$ are determined by Convention 1
 (cf. \eqref{id1}). 
Let $(z_{1},\dots,z_{d})$ be a complex fiber coordinate system of $\tilde f$, $z_{i}=x_{i}+\sqrt{-1}y_{i}$.
We orient the fibers of $\tilde f$ so that if $(-1)^{dn}(y',x')$, $y'=(y_{1},\dots,y_{d})$, is a positive fiber coordinate system  and if
$(y'',x'')$, $y''=(y_{d+1},\dots,y_{m})$, is a positive coordinate system on $Y$, then $(y,x)$ is a positive coordinate system on $X$, where $y=(y',y'')$ and $x$ is as above.
\end{enumerate}

Now we define the integration morphism \eqref{intmor}. We do it separately for 
$H^m_K(X;\scO^{(p)}_X)$ and $i^{-1}{or}_{X}(M)$.
First we do this for the latter. 
%We define the  orientation sheaf $or_{X/Y}$ of $f:X\ra Y$ by
%\[
%{or}_{X/Y} = {or}_X{\otimes}_{{}_{\Z_{X}}} f^{-1}{or}_Y\qquad\text{or}\qquad
%{or}_X = {or}_{X/Y}{\otimes}_{{}_{\Z_{X}}} f^{-1}{or}_Y.
%\]
%It is a sheaf on $X$ and describes the orientation of fibers of $f$.
%Let $i: M \hra X$ denote the closed embedding.
%In the sequel tensor products of various orientation sheaves are over $\Z_{M}$. Since orientation sheaves 
%are self-dual,
%the canonical \iso\ $or_{M}\otimes or_{M}\simeq\Z_{M}$, 
Note that we  have the identifications
%\begin{equationth}\label{oriid}
\[
%{or}_{M/X}{\otimes}_{{}_{\Z_{M}}} {or}_{M} = 
i^{-1}{or}_{X}
= i^{-1}(or_{\tilde f} \otimes_{{}_{\Z_{X}}} \tilde f^{-1}or_Y)
= i^{-1}or_{\tilde f}\otimes_{{}_{\Z_{M}}} f^{-1}j^{-1}{or}_{Y}.
\]
%\end{equationth}
Given 
%a section  
$a_{X}\in i^{-1}{or}_{X}(M)$, we take $a_{\tilde f} \in i^{-1}{or}_{\tilde f}(M)$ and
$a_{Y}\in j^{-1}{or}_{Y}(N)$ so that
\[
a_{X}=a_{\tilde f} \otimes f^{-1}a_{Y}.
\]
For any point $x \in M$, we have the identification
\[
or_{\tilde f,x}= H^{2d}_{\{x\}}(\tilde f^{-1}(\tilde f(x));\, \Z_{\tilde f^{-1}(\tilde f(x))}).
\]
Hence, once we fix an orientation of the fiber, which is a complex \mfd\ of dimension $d$,
we can determine an  integer $n(a_{\tilde f})(x)$ as the cap product
\[
[\tilde f^{-1}(\tilde f(x))] \smallfrown  a_{\tilde f}(x) \in \Z
\]
with the fundamental class $[\tilde f^{-1}(\tilde f(x))] \in {H}_{2d}(\tilde f^{-1}(\tilde f(x));\Z)$ 
of 
the fiber $\tilde f^{-1}(\tilde f(x))$ (cf. Subsection \ref{subsOT}).
Note that $n(a_{\tilde f})(x)$ is a locally constant function on $M$. Furthermore,
as each fiber of $f$ is assumed to be connected, we can regard $n(a_{\tilde f})$
as a locally constant function on $N$, i.e., $n(a_{\tilde f}) \in \Z_N(N)$.
% and we have $\int_{f}a$ for $a\in \vG(M;{or}_{M/X}{\otimes} {or}_{M/N})$. 
Then we define the integration of
$a_{X}$ by
\[
\int_f a_{X}
=n(a_{\tilde f})\, a_{Y}
\in j^{-1}or_{Y}(N).
\]
Clearly, the integration thus defined is independent of the choice of 
$a_{\tilde f}$ or $a_Y$.

\begin{remark}{\label{rem:integration-orientation}}
The point here is that we
%There are several ways to 
choose an orientation of the fiber of $\tilde f$
%a complex manifold.
%, which are so-called ``natural''.
%Here it suffices to choose one 
%of them 
and fix it throughout the integration procedure.
The final outcome, after multiplication by the integration of an element $u$ in $H^m_K(X;\scO^{(p)}_X)$, does not depend on such a choice.
\end{remark}

Now we define the integration morphism 
for 
$u \in H^m_K(X;\scO^{(p)}_X)
\simeq \hdccppk{p}{m}{\V_K}{\V'_K}$,
where
%\[
$\V_K = \{V_0=X\ssm K,\,V_1= X\}$ and 
%\quad \text{and} \quad
$\V'_K = \{V_0\}$.
%\]
We also set
%\[
$\W = \{W_0=Y\ssm N,\,W_1= Y\}$ and
%\quad \text{and} \quad
$\W' = \{W_0\}$.
%\]
Let 
\[
\tau = (\tau_1,\,\tau_{01}) \in 
\scE^{(p,m)}(V_1) \oplus
\scE^{(p,m-1)}(V_{01})
=
\scE^{(p,m)}(\V_K,\V'_K)
\]
be a representative of $u \in \hdccppk{p}{m}{\V_K}{\V'_K}$.
Let $\varphi$ be a $C^\infty$-function on $X$ satisfying
\begin{enumerate}
\item[(1)] $\tilde f|_{\op{Supp}(\varphi)}:\, \op{Supp}(\varphi) \to Y$
	is proper,
\item[(2)] $\varphi$ is identically $1$ on an open neighborhood of $K$.
\end{enumerate}
Set
\[
	\hat{\tau}
	= (\hat{\tau}_1,\, \hat{\tau}_{01})
	=
 (\varphi \tau_1 + \bar{\partial}\varphi \wedge \tau_{01},\, \varphi \tau_{01})
\in
\scE^{(p,m)}(V_1) \oplus \scE^{(p,m-1)}(V_{01}).
\]
Then, by noticing $\bar{\partial} \tau_{01} = \tau_1$ on $V_{01}$,
we have
\[
\tau - \hat{\tau}
= \bar{\vartheta} ((1-\varphi)\tau_{01},\,0),
\]
and hence,
$\hat{\tau}$ is also a representative of $u$. 
%Clearly $\hat{\tau}_1$ belongs to $\scE^{(p,m)}(X)$
%and $\hat{\tau}_{01}$ is in $\scE^{(p,m-1)}(f^{-1}(Y\ssm N))$. 
Furthermore, $\tilde f$ is proper on
$\op{Supp}(\hat{\tau}_1)$ and $\op{Supp}(\hat{\tau}_{01})$.
Hence we may apply the usual integration of  differential forms along fibers of $\tilde f$
to 
%both the forms 
$\hat{\tau}_1$ and $\hat{\tau}_{01}$, 
and we  see that
\[
\left(\int_{\tilde f} \hat{\tau}_1,\, \int_{\tilde f}  \hat{\tau}_{01}\right)
\in \scE^{(p-d,n)}(W_1) \oplus
\scE^{(p-d,n-1)}(W_{01})
\]
gives an element of
$\hdccppk{p-d}{n}{\W}{\W'}$.
Here 
%the last two 
the fiber integration $\int_{\tilde f} $ is to be performed with the  orientation of 
 fibers 
 %same as the one 
 chosen 
 %in the first part 
 for the integration on $i^{-1}{or}_{X}(M)$
(cf. Remark~{\ref{rem:integration-orientation}}). 
%%%
%\[
%\int_{f}:	H^m_K(X;\scO^{(p)}_X)\lra H^n_N(Y;\,\scO^{(p-d)}_Y).
%\]
%%%

Summing up, the integration along fibers of $f$ for  
$[(\tau_1,\tau_{01})] \otimes a_{X}$
 is given by
\[
\left[\left(\int_{\tilde f}  \varphi \tau_1 + \bar{\partial}\varphi \wedge \tau_{01},\,\,
\int_{\tilde f}  \varphi \tau_{01}\right)\right] \otimes
\int_f a_{X}\in \hdccppk{p-d}{n}{\W}{\W'}
\otimes_{{}_{Z_{N}(N)}}j^{-1}{or}_{Y}(N).
\]
It can  easily be confirmed that
the above definition does not depend on the choice of the  representative of the hyperform
or of $\varphi$. It does not depend on the choices of various orientations involved either.

\begin{remark} By definition of $d:\scB^{(p)}\ra \scB^{(p+1)}$ (cf. \eqref{defdiff}), we see that the above integration
$\int_{f}$ induces a morphism of complexes $f_!(\SHhyperf{(\bullet)}_M \otimes_{{}_{\Z_{M}}} {or}_M)[m]\ra
(\SHhyperf{(\bullet)}_N \otimes_{{}_{\Z_N}} or_N)[n]$.
Moreover, it is compatible with the definition in terms of derived functors as given in \cite{Skk}.
\end{remark}

\begin{example}{\label{rem:standard-integration}} 
%Suppose $m\ge n$ and set $d=m-n$.
Let $M=\R^m=\{(x_1,\dots,x_{m})\}$ and $N=\R^n=\{(x_{d+1},\dots, x_m)\}$, $d=m-n$. Also let
$X = \C^m=\{(z_1,\dots,z_{m})\}$ and $Y = \C^n=\{(z_{d+1},\dots,z_{m})\}$, $z_{i} = x_{i} +\sqrt{-1}y_{i}$.
%,\, \cdots,\,z_m = x_m + \sqrt{-1}y_m)$.
For $\tilde f:X \to Y$, we take the  projection\,;
$\tilde f(z_1,\dots,z_m) = (z_{d+1},\dots,z_m)$.
We take, for 
%the complex manifolds
$X$, $Y$ and  the fibers of $\tilde f$,  the  orientations for which
$(y_1,\dots,y_{m},x_1,\dots,x_m)$, $(y_{d+1},\dots,y_{m},x_{d+1},\dots,x_m)$ and
$(-1)^{dn}(y_{1},\dots,y_{d},x_{1},\dots,x_d)$ are positive systems, according to Conventions 1 and 2.
% (cf.  Remark~\ref{remfund}.\,1).

In this case, we usually omit the orientation sheaves appearing in the above construction
via the identifications $i^{-1}or_{X}(M)=\Z$,
$j^{-1}or_{Y}(N)=\Z$
 and
$i^{-1}or_{\tilde f}(M)=\Z$, where each orientation specified as above corresponds to
%$\Z \simeq i^{-1}or_{X}(M)$,
%$\Z \simeq j^{-1}or_{Y}(N)$,
 %and
%$\Z \simeq i^{-1}or_{\tilde f}(M)$, where 
%the image of 
$1 \in \Z$.
%is sent to the  generator of each orientation specified as above.
If we denote by $a_{X}$, $a_{Y}$ and $a_{\tilde f}$ 
such orientations, \r, we have
\[
a_{X}=a_{\tilde f}\otimes f^{-1}a_{Y}\quad\text{so that}\quad \int_f a_{X} = 1.
\]
Thus the integration along fibers of $f$ for the hyperform 
$u=[(\tau_1,\,\tau_{01})] 
\in \hdccppk{p}{m}{\V_K}{\V'_K}$
with the  orientations as above is given by
\[
%(-1)^{(m(m+1) - n(n+1))/2}
\left[\left(\int_{\tilde f} \varphi \tau_1 + \bar{\partial}\varphi \wedge \tau_{01},\,\,
\int_{\tilde f} \varphi \tau_{01}\right)\right] 
\in \hdccppk{p-d}{n}{\W}{\W'}.
\]
\end{example}

\
An auxiliary function $\varphi$
is involved in the definition of  integration along fibers, however
we can eliminate it
% from the definition
 in some special cases as we will see.
Let $N$ and $N_1$ be real analytic manifolds with
 $\dim N = n$ and $\dim N_1 = d$,
and let $Y$ and $Y_1$ be their complexifications.
Set $X = Y_1 \times Y$, $M = N_1 \times N$ and $m = d + n$. 
We take $\tilde f$ to be  the canonical projection $X = Y_1 \times Y\to Y$ and orient the various space involved
according to Conventions 1 and 2, regarding $Y_{1}$ as the fiber. 
Let $K_{1}$ be a compact set  in $Y_1$ and set
$K =K_1\times Y$.
We may choose a compact set $D_1$ in $Y_1$ with 
 $C^\infty$ boundary $\partial D_1$ containing $K_1$ in its interior.
Further, we may assume 
that $\varphi$ depends only on the variables in $Y_1$ and
\[
K = K_1\times Y \subset \op{Supp}(\varphi) 
\subset  \op{Int}(D_1) \times Y.
\]
For any differential form $\kappa$ on $X$, we denote by $\int_{D_1} \kappa$
the partial integration 
of $\kappa$ with respect to the variables in $Y_1$ on the domain $D_1$.
%, where
%the orientation of $D_1$ is the natural one of the complex manifold $Y_1$.

Now we have
\[
\begin{aligned}
\left(\int_{\tilde f}  \varphi \tau_1 + \bar{\partial}_X\varphi \wedge \tau_{01},\,\,
\int_{\tilde f}  \varphi \tau_{01}\right)
&=
\left(\int_{D_1} \varphi \tau_1 + \bar{\partial}_X\varphi \wedge \tau_{01},\,\,
\int_{D_1} \varphi \tau_{01}\right) \\
&=\left(\int_{D_1} \tau_1 - \bar{\partial}_X((1 - \varphi)\tau_{01}),\,\,
\int_{D_1} \varphi \tau_{01}\right).
\end{aligned}
\]
By applying Stokes' formula and noticing $\bar{\partial}_X = \bar{\partial}_{Y_1} + \bar{\partial}_Y$, the above is equal to
\[
\left(\int_{D_1} \tau_1 
-\int_{\partial D_1} \tau_{01}|_{\partial D_1 \times Y}
- \bar{\partial}_{Y} \int_{D_1} (1 - \varphi) \tau_{01},\,\,
\int_{D_1} \varphi \tau_{01}\right).
\]
Since $\int_{D_1} (1 - \varphi) \tau_{01}$ is a $C^\infty$ form on the total
space $Y$, by adding the coboundary $\bar{\vartheta}_Y \big(\int_{D_1} (1 - \varphi) \tau_{01},0\big)$ 
to the above expression,
we have the following proposition\,:
\begin{proposition}
In the situation  above, the integration 
of $[(\tau_1,\,\tau_{01})] \otimes a_{X}$
along fibers of $f$ is given by
\[
\left[\left(\int_{D_1} \tau_1 
-\int_{\partial D_1} \tau_{01}|_{\partial D_1 \times Y},\,\,
\int_{D_1} \tau_{01}\right)\right]
\otimes
\int_f a_{X}.
\]
\end{proposition}
%%%%%%

We finish this subsection by clarifying the relation between the external product and the integration.
Let $M_1$ and $M_2$ be real analytic manifolds.
%Recall that $\scV_{M_k}$ denote the sheaf of real analytic densities on $M_k$ $(k=1,2)$.

\begin{proposition}
	Let $u_k \otimes v_k \in (\scB_{M_k}\otimes_{{}_{\scA_{k}}}\scV_{M_k})(M_k)$, 
$k=1,2$, with $u_k$ having a compact support in $M_k$.
Then we have
\[
\int_{M_1 \times M_2} (u_1 \times u_2) \otimes (v_1 \wedge v_2)
=
\left(\int_{M_1} u_1 \otimes v_1\right) \,\,
\left(\int_{M_2} u_2 \otimes v_2\right).
\]
\end{proposition}

\lsection{Embedding of  distributions to the space of hyperfunctions}\label{secembdist}
In this section, we study embedding of  distributions to the space of
hyperfunctions. Let $M=\R^n$ and $X=\C^n$, and
let $u$ be a distribution on an open set $U$ in $M$. Then, in our theory,
a representative of $u$ as a hyperfunction is given by the following theorem\,:
\begin{theorem}{\label{thm:formula_embedding}}
The pair
\begin{equationth}{\label{eq:formula_embedding}}
\begin{aligned}
&\left(\int_{U} (\theta\tau_1 + \bar{\partial}_z\theta \wedge \tau_{01})(z-t,t)\,
u(t)\, dt,\,\,
\int_{U} (\theta \tau_{01})(z - t,t)\, u(t)\,dt \right) \\
&\qquad\qquad\qquad\qquad\qquad \in \scE^{(0,n)}(V_1) \oplus \scE^{(0,n-1)}(V_{01})
\end{aligned}
\end{equationth}
is well-defined and   represents the hyperfunction image of $u$.
Here $\theta(z,t)$ is a $C^\infty$ function on $X\times U$ such that
the conditions {\rm (1)}~and {\rm (2)}~before Lemma \ref{lem:lemma_inverse} hold
for $K=\{0\}$ and an open set $T$ satisfying {\eqref{eq:condition_for_T}},
and $(\tau_1(z), \tau_{01}(z))$ is a representative of Dirac's delta density,
for example, we may choose the one of Bochner-Martinelli type $($cf. Definition \ref{defdf}$)$.
\end{theorem}
Note that the precise definition of the above integrations in the theorem
is given in {\eqref{eq:def-integral-of-distribution}}.

Traditionally, embedding of a distribution $u$ was realized by taking
a convolution of $u$ with the Cauchy kernel $K(z)$. To make the integration converge,
we first decompose $u$ as a sum $\sum_\lambda u_\lambda$ of distributions $u_\lambda$'s
with compact support. Unfortunately, 
since the support of each $\langle u_\lambda(t), K(z-t)\rangle$ is not compact,
the sum $\sum_\lambda \langle u_\lambda(t), K(z-t)\rangle$ itself 
does not converge in general.
Therefore, to obtain an embedding image of $u$,
we should sum up these terms as an element of quotient classes, i.e., 
we consider the locally finite sum $\sum_\lambda [\langle u_\lambda(t), K(z-t)\rangle]$
in the space of hyperfunctions.
Hence, in the traditional approach, we cannot directly obtain a representative of an embedding image of $u$.
As we have already seen, our formula \eqref{eq:formula_embedding} gives globally
a representative of the embedding image of a distribution, which
is one of advantages in our framework. 

In the subsequent arguments, we establish the above theorem and
several properties of the embedding morphism, in particular, it is canonical.

\

Let $M = \R^n$ and $X = \C^n$. We denote by $\mathscr{O}\mathscr{E}^{(p;r)}$ the sheaf on $X \times M$ of  $C^\infty$ forms on $X \times M$ 
that are holomorphic $p$-forms in the complex variables $z$ of $X$ and
$r$-forms in the variables $t$ of $M$, i.e., they have the form
\[
\sum_{|I|=p,\,|J|=r} f^{I,J}(z,t) dz_I dt_J,
\]
where each $f^{I,J}(z,t)$ is a $C^\infty$ function that is holomorphic in $z$.
Let $U$ be an open set in $M$ and $K$ a compact set in $X$. Let also $\pi: X \times U \to X$ be the canonical projection.
Then the canonical sheaf morphism $\pi^{-1} \mathscr{O}^{(n)} \to \mathscr{O}\mathscr{E}^{(n;0)}$ induces the canonical one
\[
\iota: H^n_K(X; \mathscr{O}^{(n)}) \longrightarrow
H^n_{K\times U}(X \times U; \mathscr{O}\mathscr{E}^{(n;0)}).
\]

Let $\mathscr{E}_{X\times M}^{(p,q;r)}$ denote the
sheaf of $C^\infty$ forms on $X \times M$ that
are $(p,q)$-forms in the variables $z$ of $X$ and
$r$-forms in the variables $t$ of $M$.
Note that $\mathscr{E}_{X\times M}^{(p,q;r)}$ is fine. 
%We have the following vanishing theorem\,:
\begin{theorem}
Let $\Omega$ be a Stein open set in $X$ and $U$ an open set in $M$.
Then the following complex is exact\,{\rm :}
\[
0 \to \mathscr{O}\mathscr{E}^{(p;r)}(\Omega \times U)
\to 
\mathscr{E}_{X\times M}^{(p,0;r)}(\Omega \times U)
\overset{\bar{\partial}_z}{\to}
\mathscr{E}_{X\times M}^{(p,1;r)}(\Omega \times U)
\overset{\bar{\partial}_z}{\to}
\cdots
\overset{\bar{\partial}_z}{\to}
\mathscr{E}_{X\times M}^{(p,n;r)} 
(\Omega \times U)
\to 0.
\]
In particular, the $\bar{\partial}_z$-complex $\mathscr{E}_{X \times M}^{(p,\bullet;r)}$
is a fine resolution of $\mathscr{O}\mathscr{E}^{(p;r)}$.
\end{theorem}
\begin{proof}
	Since all the spaces $\mathscr{O}^{(p)}(\Omega)$, $\mathscr{E}^{(p,q)}(\Omega)$,
	$\mathscr{E}^{(r)}(U)$,
$\mathscr{O}\mathscr{E}^{(p;r)}(\Omega \times U)$ and
$\mathscr{E}_{X \times M}^{(p,q;r)}(\Omega \times U)$ are {\bf{FN}} 
(Fr\'echet-Nuclear topological vector spaces) and
since we have the following formulas for the topological tensor products on these spaces
\[
\mathscr{O}^{(p)}(\Omega) \widehat{\otimes} 
\mathscr{E}^{(r)}(U) =
\mathscr{O}\mathscr{E}^{(p;r)}(\Omega \times U),\quad
\mathscr{E}^{(p,q)}(\Omega) \widehat{\otimes} 
\mathscr{E}^{(r)}(U) =
\mathscr{E}_{X \times M}^{(p,q;r)}(\Omega \times U),
\]
by noticing that the topological tensor product is an exact functor
on {\bf{FN}},
the claim follows from the usual vanishing theorem of Oka. 
\end{proof}
Let us consider the coverings 
%\[
$\mathcal{S}_K = \{S_0 = X \ssm K,\, S_1 = X\}
$
%,\quad
and
$\mathcal{S}'_K = \{S_0\}$
%\]
of the pair $(X,\, X \ssm K)$.
Consider also the coverings
%\[
$\mathcal{S}_K \times U = \{S_0 \times U,\, S_1 \times U\}$ and
%, \quad
$\mathcal{S}'_K \times U = \{S_0 \times U\}$
%\]
of $(X \times U,\, (X \ssm K) \times U)$.
Hereafter, we sometime denote $(\mathcal{S}_K \times U, \mathcal{S}'_K \times U)$
by $(\mathcal{S}_K, \mathcal{S}'_K) \times U$.
% for convenience.

For a complex $\scF^\bullet$ of fine sheaves on $X$, we define a complex
$C(\mathcal{S}_K,\mathcal{S}'_K)(\scF^\bullet)$ as in the case of  Dolbeault complex.
Namely its $q$-th degree term is given by
\[
C^q(\mathcal{S}_K,\mathcal{S}'_K)(\scF^\bullet) =
\scF^q(S_1) \oplus \scF^{q-1}(S_{01})
\]
and the differential is defined as before. For a complex $\scG^\bullet$ of fine shaves on $X \times U$,
the complex
$C((\mathcal{S}_K,\mathcal{S}'_K) \times U)(\scG^\bullet)$ 
is similarly defined.
Then, by the above theorem and the same reasoning as that for the relative Dolbeault cohomology,
we have the isomorphisms
\[
\begin{aligned}
&H^n_K(X;\mathscr{O}^{(n)}) \simeq 
H^n(C(\mathcal{S}_K,\mathcal{S}'_K)(\mathscr{E}^{(n,\bullet)})), \\
&H^n_{K \times U}(X\times U;\mathscr{O}\mathscr{E}^{(n;0)}))
\simeq H^n(C((\mathcal{S}_K,\mathcal{S}'_K) \times U)
(\mathscr{E}^{(n,\bullet;0)}_{X \times M})).
\end{aligned}
\]
Hence an element in $H^n_{K \times U}(X\times U;\mathscr{O}\mathscr{E}^{(n;0)}))$
is represented by 
\[
(\tau_1(z,t),\,\tau_{01}(z,t))
\in \mathscr{E}^{(n,n;0)}_{X \times M}(S_1 \times U) \oplus 
\mathscr{E}^{(n,n-1;0)}_{X \times M}(S_{01} \times U)
\]
satisfying $\tau_1 = \bar{\partial}_z \tau_{01}$ on $(X \ssm K) \times U$.
Under these identifications, the canonical morphism
\[
\iota: H^n_K(X; \mathscr{O}^{(n)}) \longrightarrow
H^n_{K\times U}(X \times U; \mathscr{O}\mathscr{E}^{(n;0)})
\]
is just regarding a representative
$(\tau_1(z), \tau_{01}(z))$ of an element in
$H^n_K(X; \mathscr{O}^{(n)})$ as that in
$H^n_{K\times U}(X \times U; \mathscr{O}\mathscr{E}^{(n;0)})$.
Furthermore, by fixing $t$ to some point $t_0 \in U$, 
we can easily see\,:
\begin{lemma}
The above morphism $\iota$ is injective.
\end{lemma}

Let $T$ be an open set in $X \times U$containing $K \times U$ with
%let
$j_T: T \hookrightarrow X \times U$ 
%denote 
the open embedding.
For any sheaf $\mathscr{F}$, let us consider
the subsheaf ${j_T}_!{j_T}^{-1} \mathscr{F} \subset \mathscr{F}$.
By  definition, 
%we have
\[
({j_T}_!{j_T}^{-1}\mathscr{F})(W)
= \{\,s \in \mathscr{F}(W)\mid\, \operatorname{Supp}_W(s) \subset T\,\}
\]
for any open set $W \subset X \times U$, in particular, we have
%\[
$({j_T}_!{j_T}^{-1}\mathscr{F})|_T = \mathscr{F}|_T$.
%\]

Since ${j_T}_!{j_T}^{-1} \mathscr{E}_{X\times M}^{(p,q;r)}$ is a $C^\infty$-module,
it is a fine sheaf. Furthermore, since ${j_T}_!{j_T}^{-1}$ is an exact functor,
the complex ${j_T}_!{j_T}^{-1}\mathscr{E}_{X\times M}^{(n,\bullet;0)}$ is
a fine resolution of ${j_T}_!{j_T}^{-1}\mathscr{O}\mathscr{E}^{(n;0)}$.
Hence, we have the isomorphism
\[
H^n_{K \times U}(X \times U; {j_T}_!{j_T}^{-1} \mathscr{O}\mathscr{E}^{(n;0)})
\simeq
H^n(C( (\mathcal{S}_K, \mathcal{S}'_K) \times U)
({j_T}_!{j_T}^{-1}\mathscr{E}_{X\times M}^{(n,\bullet;0)})).
\]
This implies that
an element in $H^n_{K \times U}(X\times U;{j_T}_!{j_T}^{-1}\mathscr{O}\mathscr{E}^{(n;0)}))$
is represented by 
\[
(\tau_1(z,t),\,\tau_{01}(z,t))
\in \mathscr{E}^{(n,n;0)}_{X \times M}(S_1 \times U) \oplus 
\mathscr{E}^{(n,n-1;0)}_{X \times M}(S_{01} \times U)
\]
which satisfies $\tau_1 = \bar{\partial}_z \tau_{01}$ on $(X \ssm K) \times U$
and the support condition
\begin{equationth}{\label{eq:support_condition_T}}
\operatorname{Supp}_{X \times U}(\tau_1) \subset T,\quad
\operatorname{Supp}_{(X\ssm K) \times U}(\tau_{01}) \subset T.
\end{equationth}
Further, the canonical morphism (induced from the one 
${j_T}_!{j_T}^{-1} \mathscr{O}\mathscr{E}^{(n;0)} \to \mathscr{O}\mathscr{E}^{(n;0)}$)
\[
\begin{matrix}
H^n_{K \times U}(X \times U; {j_T}_!{j_T}^{-1} \mathscr{O}\mathscr{E}^{(n;0)})
&\longrightarrow
&H^n_{K \times U}(X \times U; \mathscr{O}\mathscr{E}^{(n;0)}) \\
\text{\rotatebox{90}{$\in$}} & & \text{\rotatebox{90}{$\in$}} \\
[(\tau_1(z,t), \tau_{01}(z,t))]
&
\longmapsto
&
[(\tau_1(z,t), \tau_{01}(z,t))]
\end{matrix}
\]
is isomorphic. In fact, we have
\[
\begin{aligned}
&H^n_{K \times U}(X \times U; {j_T}_!{j_T}^{-1} \mathscr{O}\mathscr{E}^{(n;0)})
\simeq
H^n_{K \times U}(T; {j_T}_!{j_T}^{-1} \mathscr{O}\mathscr{E}^{(n;0)}) \\
&\qquad\simeq
H^n_{K \times U}(T; \mathscr{O}\mathscr{E}^{(n;0)})
\simeq
H^n_{K \times U}(X \times U; \mathscr{O}\mathscr{E}^{(n;0)}).
\end{aligned}
\]
Here the first and third isomorphisms are due to excision 
and the second morphism is isomorphic thanks to
$({j_T}_!{j_T}^{-1}\mathscr{O}\mathscr{E}^{(n;0)})|_T = 
\mathscr{O}\mathscr{E}^{(n;0)}|_T$.

Summing up, we have the commutative diagram whose arrows are all isomorphisms.
\[
\SelectTips{cm}{}
\xymatrix
@C=.7cm
@R=.7cm
{H^n_{K \times U}(X \times U; {j_T}_!{j_T}^{-1}\mathscr{O}\mathscr{E}^{(n;0)})\ar[r]^-{\sim}& H^n_{K \times U}(X \times U;\mathscr{O}\mathscr{E}^{(n;0)}) \\
H^n(C( (\mathcal{S}_K, \mathcal{S}'_K) \times U)
({j_T}_!{j_T}^{-1}\mathscr{E}_{X\times M}^{(n,\bullet;0)}))\ar[u]^-{\wr}\ar[r]^-{\sim}& H^n(C( (\mathcal{S}_K, \mathcal{S}'_K) \times U)
(\mathscr{E}_{X\times M}^{(n,\bullet;0)}))\ar[u]^-{\wr}.}
\]
This also implies\,:
\begin{lemma}{\label{lem:support_dirac}}
Let $T$ be an open set in $X \times U$ containing $K \times U$.
Then, any element in $H^n_{K \times U}(X \times U;\mathscr{O}\mathscr{E}^{(n;0)})$
has a representative 
\[
(\tau_1(z,t), \tau_{01}(z,t)) \in
C^n( (\mathcal{S}_K, \mathcal{S}'_K) \times U)(\mathscr{E}_{X\times M}^{(n,\bullet;0)})
\]
which satisfies the support condition {\eqref{eq:support_condition_T}}.
\end{lemma}

Now we give a concrete method to obtain such a representative appearing in the above lemma.
Let $\theta(z,t)$ be a $C^\infty$-function on $X \times U$ satisfying that
\begin{enumerate}
\item[(1)] $\theta(z,t)$ is identically $1$ in a neighborhood of $K \times U$.
\item[(2)] $\operatorname{Supp}(\theta) \subset T$.
\end{enumerate}
\begin{lemma}{\label{lem:lemma_inverse}}
The inverse of the above isomorphism
$H^n_{K \times U}(X \times U; {j_T}_!{j_T}^{-1} \mathscr{O}\mathscr{E}^{(n;0)})
\rightarrow
H^n_{K \times U}(X \times U; \mathscr{O}\mathscr{E}^{(n;0)})$
is given by
\[
[(\tau_1, \tau_{01})] \longmapsto
[(\theta \tau_1 + \bar{\partial}_z \theta \wedge \tau_{01},\, \theta \tau_{01})]
\in H^n_{K \times U}(X \times U; {j_T}_!{j_T}^{-1} \mathscr{O}\mathscr{E}^{(n;0)}).
\]
\end{lemma}
\begin{proof}
This comes from the fact
\[
(\tau_1, \tau_{01}) - 
(\theta \tau_1 + \bar{\partial}_z \theta \wedge \tau_{01},\, \theta \tau_{01})
= \bar{\vartheta} ( (1 - \theta) \tau_{01}, 0).
\]
\end{proof}

%\

In what follows, we assume $K = \{0\}$. 
We first take an open set $T$ in $X \times U$ satisfying 
%the condition
\begin{equationth}{\label{eq:condition_for_T}}
\{0\} \times U \,\subset\,
T \,\subset\,
\{\,(z,t) \in X \times U\mid\, |z| < 3^{-1} \min\{1,\,\operatorname{dist}(t, M \ssm U)\}\,\}.
\end{equationth}
Here we set $\operatorname{dist}(t, \emptyset) = +\infty$.
Let $i: M \hookrightarrow X$ be the canonical embedding.
Set 
\[
V_U = U \times \sqrt{-1} \R^n \subset X.
\]
Then the following lemma easily follows from 
{\eqref{eq:condition_for_T}}\,:
%{\eqref{eq:support_estimate}}.
\begin{lemma}{\label{lem:compact_support}}
For any compact set $L$ in $V_U \ssm U \subset X \ssm M$, the following set is compact in $U$\,{\rm :}
\[
\overline{\underset{z \in L}{\bigcup} \{\,t \in U\mid \, (z - t,t) \in T \,\}}.
\]
%is compact in $U$.
\end{lemma}

Let $u \in \mathscr{D}b(U)$, that is, $u$ is an element of
the topological dual of the space 
\[
\vG_c(U;\mathscr{E}_M^{(n)}) \otimes_{{}_{\Z(U)}} or_M(U) =
\vG_c(U;\mathscr{E}_M) \otimes_{{}_{\mathscr{A}(U)}} \scV_M(U), 
\]
where $\scV_M$ is the sheaf of density on $M$.
Recall that $i$ denotes the canonical inclusion $M \hookrightarrow X$. 
In what follows, the canonical inclusion $M \times M \hookrightarrow X \times M$ 
is also denoted by $i$.
% for simplicity.
Let $\pi: X \times M \to X$ be the canonical projection.
For any $\varphi \in \mathscr{E}^{(n,q;0)}_{X \times M} \otimes_{ {}_{\Z_{X \times M}}} i_*i^{-1}\pi^{-1}or_X$,
we formally define  $u * \varphi = \langle u, \varphi(z-t,t) \rangle$ 
in the following way\,: Fix the coordinates $x$ and $z = x + \sqrt{-1}y$ of 
$M$ and $X$, respectively, and assume $\varphi$ has a form
%%%%%%%%%%%%%%%%%%%%%%%%%%%%%%%%%%%%%definition
\[
\varphi = \sum_{|J|=q} \varphi^J(z,t)\, d\bar{z}_J \wedge dz \otimes a_X
\in \mathscr{E}^{(n,q;0)}_{X\times M} \otimes_{ {}_{\Z_{X\times M}}} i_*i^{-1}\pi^{-1}or_X,
\]
where $dz = dz_1 \wedge \cdots \wedge dz_n$.
Then we define
\begin{equationth}{\label{eq:def-integral-of-distribution}}
u * \varphi = \langle u(t),\, \varphi(z-t,t) \rangle =
\sum_{|J|=q} \langle u(t),\, \varphi^J(z-t,t) dt \otimes e_M\rangle\, d\bar{z}_J
\otimes (a_X \otimes e_M^{-1}).
\end{equationth}
%%%%%%%%%%%%%%%%%%%%%%%%%%%%%%%%%%%%%%%%%
Here $e_M$ is a section in $or_M(U)$ which generates $or_{M,x}$ over $\Z$ 
at each  $x\in U$ and, through the identification
$i^{-1} or_X \simeq or_{M/X} \otimes_{ {}_{\Z_M}} or_M$, we regard 
$a_X \otimes e_M^{-1} \in (i^{-1}or_X \otimes or_M^{-1})(U)$ as a section in  $or_{M/X}(U)$.
Set the coverings
%\[
$\mathcal{V}_U = \{V_0 = V_U \ssm U,\, V_1 = V_U\}$,
%\qquad
$\mathcal{V}'_U = \{V_0\}$,
%\]
where $V_U = U \times \sqrt{-1}\R^n \subset X$.
Then, for 
\[
\mu \otimes a_X = (\mu_1 \otimes a_X, \mu_{01} \otimes a_X) 
\in C^q((\mathcal{S}_{\{0\}}, \mathcal{S}'_{\{0\}}) \times U) 
(j_{T!}{j_T}^{-1} \mathscr{E}^{(n,\bullet;0)}_{X\times M}) 
\otimes_{ {}_{\Z(X \times U)}} i^{-1}or_X(M),
\]
it follows from Lemma {\ref{lem:compact_support}} that
$u * (\mu_1 \otimes a_X)$ and
$u * (\mu_{01} \otimes a_X)$ are well-defined, and they belong to
$\mathscr{E}^{(0,q)}(V_1) \otimes or_{M/X}(U)$
and
$\mathscr{E}^{(0,q-1)}(V_{01}) \otimes or_{M/X}(U)$, respectively.
That is,
\[
u * (\mu \otimes a_X) = 
(u * (\mu_1 \otimes a_X),\, u * (\mu_{01}  a_X)) 
\in \mathscr{E}^{(0,q)}(\mathcal{V}_U,\,\mathcal{V}'_U) 
\otimes_{ {}_{\Z(U)}} or_{M/X}(U).
\]

Let
$h \otimes a_X
\in H^n_{\{0\} \times U}(X \times U;\mathscr{O}\mathscr{E}^{(n;0)}) 
\otimes_{ {}_{\Z(X \times U)}} i^{-1}or_X(M)$.
It follows from Lemma {\ref{lem:support_dirac}} that
we find a representative 
\[
\tau = (\tau_1(z,t), \tau_{01}(z,t))
\in \mathscr{E}^{(n,n;0)}_{X \times M}(X \times U) 
\oplus \mathscr{E}^{(n,n-1;0)}_{X \times M}((X \ssm \{0\}) \times U)
\] 
of $h$ which satisfies
\begin{equationth}{\label{eq:support_estimate}}
\operatorname{Supp}_{X \times U}(\tau_1) \subset T,\quad
\operatorname{Supp}_{(X\ssm \{0\}) \times U}(\tau_{01}) \subset T.
\end{equationth}
\begin{lemma}
The morphism 
\[
u\,*\,(\bullet) :
C^\bullet((\mathcal{S}_{\{0\}}, \mathcal{S}'_{\{0\}}) \times U) 
(j_{T!}{j_T}^{-1} \mathscr{E}^{(n,\bullet;0)}_{X\times M}) \otimes i^{-1}or_X(M)
\longrightarrow
\mathscr{E}^{(0,\bullet)}(\mathcal{V}_U,\,\mathcal{V}'_U) \otimes or_{M/X}(U)
\]
is that of complexes.
In particular, %we have $\bar{\vartheta} (u * (\tau \otimes a_X)) = 0$. 
$[u * (\tau \otimes a_X)]$ defines a hyperfunction on $U$
and it does not depend on the choices of a representative 
$\tau = (\tau_1(z,t),\,\tau_{01}(z,t))$ of $h$.
%with the support condition {\eqref{eq:support_estimate}}.
\end{lemma}
%%%%%%%%%%%%%%%%%%%%%%%%%%%%%%%%%%%%%%%%%%%%%%%%%%%%%%%%%%%%%%%%%%%%%%%%%%%%%%%%%%add a proof
\begin{proof}
For  fixed coordinates,
we set 
\[
	\varphi=\sum_{|J|=q}\varphi(z,t)^{J}d\bar{z}_J\wedge dz\in j_{T!}{j_T}^{-1} \mathscr{E}^{(n,q;0)}_{X \times M}.
\]
Then it is easy to check
\begin{align*}
u* \bar\partial_z(\varphi\otimes a_X) &=u* \Bigl( \sum_{|J|=q} \bar\partial_z\varphi^{J}\wedge d\bar{z}_J\wedge dz \otimes a_X \Bigr)\\ 
%&= u* \Bigl( \sum_{|J|=q} \left(\sum_{j=1}^n \frac{\partial \varphi^{J}}{\partial \bar{z}_j}d\bar{z_j}\right)\wedge d\bar{z}_J\wedge dz \otimes a_X \Bigr)\\ 
%&= u* \Bigl( \sum_{|J|=q} \sum_{j=1}^n \frac{\partial \psi^{J}}{\partial \bar{z}_j}\bar{z_j}\wedge d\bar{z}_J\wedge dz  \Bigr)\\ 
%&= \sum_{|J|=p} \sum_{j=1}^n u* \left(\frac{\partial \varphi^{J}}{\partial \bar{z}_j}\varPhi(z)\wedge d\bar{z_j}\wedge d\bar{z}_J\right)\\ 
&= \sum_{|J|=q} \sum_{j=1}^n \langle u,\,  \frac{\partial \varphi^{J}}{\partial \bar{z}_j}(z-t,t)dt\otimes e_M \rangle\ d\bar{z_j}\wedge d\bar{z}_J
\otimes (a_X \otimes e_M^{-1})\\ 
&= \sum_{|J|=q} \left(\sum_{j=1}^n \frac{\partial}{\partial \bar{z}_j}\langle  u,\, \varphi^{J}(z-t,t)dt\otimes e_M\rangle d\bar{z}_j\right)\wedge d\bar{z}_J
\otimes (a_X \otimes e_M^{-1})\\ 
%&= \sum_{|J|=p} \bar\partial\langle  u,\, \varphi^{J}(z-t,t)\varPhi(t)\rangle \wedge d\bar{z}_J\\ 
%&= \bar\partial\left(\sum_{|J|=p} \langle  u,\, \varphi^{J}(z-t,t)\varPhi(t)\rangle d\bar{z}_J \right)\\
&= \bar\partial\left(u*(\varphi\otimes a_X )\right).
\end{align*}
Now if we set 
\[
 \tau\otimes a_X=(\tau_1\otimes a_X,\, \tau_{01}\otimes a_X)\in C^q((\mathcal{S}_{\{0\}}, \mathcal{S}'_{\{0\}}) \times U) 
 (j_{T!}{j_T}^{-1} \mathscr{E}^{(n,\bullet;0)}_{X \times M}) \otimes i^{-1}or_X(M),
\]
the lemma follows from the following commutativity;
\begin{align*}
u*\bar\vartheta_z (\tau\otimes a_X) &= u*(\bar\partial_z\tau_1\otimes a_X,\, (\tau_1-\bar\partial_z\tau_{01})\otimes a_X))\\
%&=(u*\bar\partial\eta_1,\, \, u*\eta_1-u*\bar\partial\eta_{01}) \\
&=( \bar\partial (u*(\tau_1\otimes a_X)),\,  u*(\tau_1\otimes a_X) -\bar\partial (u*(\tau_{01}\otimes a_X))\\
&=\bar\vartheta(u*(\tau\otimes a_X)).
\end{align*}
\end{proof}

Let
$h \otimes a_X
\in H^n_{\{0\} \times U}(X \times U;\mathscr{O}\mathscr{E}^{(n;0)}) 
\otimes_{ {}_{\Z(X \times U)}} i^{-1}or_X(M)$ 
and $\tau$  a representable of $h$.
Then thanks to the lemma, hereafter, we write $[u * (h \otimes a_X)]$ 
instead of $[u * (\tau \otimes a_X)]$.

\begin{lemma}{\label{lem:local_emb_db}}
We have $\operatorname{Supp}([u * (h \otimes a_X)]) \subset \operatorname{Supp}(u)$.
\end{lemma}
\begin{proof}
%Here w
We forget the orientation for simplicity.
For $\epsilon > 0$, take an open set $T_\epsilon$ with
%such that
\[
\{0\} \times U \,\subset\,
T_\epsilon \,\subset\,
\{\,(z,t) \in X \times U\mid\, |z| < 3^{-1} \min\{\epsilon,\,\operatorname{dist}(t, M \ssm U)\}\,\}.
\]
Choose a representative $\tau_\epsilon = (\tau_{1,\epsilon}, \tau_{01, \epsilon})$
of $h$ with the condition
\[
\operatorname{Supp}_{X \times U}(\tau_{1,\epsilon}) \subset T_\epsilon,\quad
\operatorname{Supp}_{(X\ssm \{0\}) \times U}(\tau_{01,\epsilon}) \subset T_\epsilon.
\]
Then we can easily see that
\[
\operatorname{Supp}(u*\tau_{1,\epsilon}) \cup
\operatorname{Supp}(u*\tau_{01,\epsilon})
\subset
\{\,z = x + \sqrt{-1}y \in \mathbb{C}^n\mid 
\operatorname{dist}(x, \operatorname{Supp}(u)) \le \epsilon\,\},
\]
from which we have
\[
\operatorname{Supp}([u * h]) = 
\operatorname{Supp}([u * \tau_\epsilon])
\subset
\{\,x \in U\mid\operatorname{dist}(x, \operatorname{Supp}(u)) \le \epsilon\,\}.
\]
Since $\epsilon > 0$ is arbitrary, this completes the proof.
\end{proof}

%\

Now we take $h \otimes a_X$ in the above argument
to be the image of Dirac's density 
$\delta(x)dx \otimes a_X
\in H^n_{\{0\}}(X; \mathscr{O}^{(n)}) \otimes_{ {}_{\Z(X)}} i^{-1}or_X(M)$ by $\iota$.
Note that $\delta(x)dx \otimes a_X$ 
is defined by the image of $1 \in \C$ by the canonical morphism
\[
\C \longrightarrow H^n_{\{0\}}(X; \mathscr{O}^{(n)}) \otimes_{ {}_{\Z(X)}} i^{-1}or_X(M),
\]
which is the topological dual of the continuous morphism
\begin{equationth}
	\begin{matrix}
	\mathscr{A}(\{0\}) = \underset{\{0\} \subset V}{\underset{\longrightarrow}{\lim}}\, \mathscr{O}(V)
&\longrightarrow 
&\C, \\
\text{\rotatebox{90}{$\in$}} & & \text{\rotatebox{90}{$\in$}} \\
\varphi
&
\longmapsto
&\varphi(0)
\end{matrix}.
\end{equationth}
We denote by the same symbol $\delta(x)dx \otimes a_X$ 
its image through the morphism $\iota \otimes \operatorname{id}:
H^n_{\{0\}}(X; \mathscr{O}^{(n)}) \otimes_{ {}_{\Z(X)}} i^{-1}or_X(M)\rightarrow
H^n_{\{0\}\times U}(X \times U; \mathscr{O}\mathscr{E}^{(n;0)}) 
\otimes_{ {}_{\Z(X \times U)}} i^{-1}or_X(M)$.

\begin{definition}
We define the map $\iota_{\mathscr{D}b(U)}: \mathscr{D}b(U) \to \mathscr{B}(U)$ by
\begin{equationth}
	\iota_{\mathscr{D}b(U)}(u) =[u * (\delta(x)dx \otimes a_X)].
\end{equationth}
\end{definition}
\begin{remark}
The definition of $u*\varphi$ given in {\eqref{eq:def-integral-of-distribution}}
depends on the choices of coordinates. However,
$\iota_{\mathscr{D}b(U)}$ is independent of such choices, which
will be shown later.
\end{remark}
Note that, by Lemma \ref{lem:local_emb_db}, $\iota_{\mathscr{D}b(U)}$ 
is a local operator, 
that is, we have 
\begin{equationth}
	\operatorname{Supp}(\iota_{\mathscr{D}b(U)}(u)) \subset \operatorname{Supp}(u)
	\qquad\text{for}\ \ u \in \mathscr{D}b(U)).
\end{equationth}
Using this local property, we can show
\begin{lemma}
For any open sets $U_1 \subset U_2$ in $M$, we have the following commutative
diagram{\rm \,:}
\[
\SelectTips{cm}{}
\xymatrix
@C=1cm
@R=.5cm
{\mathscr{D}b(U_2)\ar[d]\ar[r]^{\iota_{\mathscr{D}b(U_2)}} & 
\mathscr{B}(U_2)\ar[d] \\
\mathscr{D}b(U_1) \ar[r]^{\iota_{\mathscr{D}b(U_1)}} & 
\mathscr{B}(U_1).}
\]
In particular, $\iota_{\mathscr{D}b} = \{\iota_{\mathscr{D}b}(U)\}_U$ induces the sheaf morphism $\mathscr{D}b \to \mathscr{B}$.
\end{lemma}
\begin{proof}
Let $\tau^{1}$ (resp. $\tau^2$) be representatives of $\delta(x)dx$ on 
$X \times U_1$ (resp. $X \times U_2$)
with the support condition, where the corresponding open set $T$ is denoted by
$T_1$ (resp. $T_2$). Let $u \in \mathscr{D}b(U_2)$.
Since $\mathscr{B}$ is a sheaf, it is enough to show,
for any $W$ being relatively compact in $U_1$, 
\[
\iota_{\mathscr{D}b(U_2)}(u)|_{W} = \iota_{\mathscr{D}b(U_1)}(u|_{U_1})|_W.
\]
Let $\theta$ be a $C^\infty$-function on $X$ whose support
is contained in $U_1$ and which is identically $1$ on
a neighborhood of $\overline{W}$. By the local property of $\iota_{\mathscr{D}b}$,
we have
\[
\iota_{\mathscr{D}b(U_2)}(u)|_{W} = 
\iota_{\mathscr{D}b(U_2)}(\theta u)|_{W}, \qquad
\iota_{\mathscr{D}b(U_1)}(u|_{U_1})|_W =
\iota_{\mathscr{D}b(U_1)}((\theta u)|_{U_1})|_W.
\]
Hence, from the beginning, we may assume that the support of $u$
is contained in $U_1$.
Let $W'$ be a relatively compact open set in $U_1$ which contains
the support of $u$ and $\overline{W}$.
Then, we can take $\tau^1$ and $\tau^2$ so that $\tau^1 = \tau^2$ holds
on a neighborhood of $X \times \overline{W'}$ and their supports are contained in
the set
\[
\{\,(z,t) \in X \times U\mid\, |z| < \epsilon\,\}
\]
for a sufficiently small $\epsilon > 0$. For such $\tau^1$ and $\tau^2$,
if $\epsilon > 0$ is sufficiently small, 
we have, on $W$,
\[
\iota_{\mathscr{D}b(U_1)}(u|_{U_1}) = 
[u * (\tau^1 \otimes a_X)] = 
[u * (\tau^2 \otimes a_X)] = \iota_{\mathscr{D}b(U_2)}(u).
\]
\end{proof}
\begin{lemma}
For $u \in \mathscr{D}b(U)$ with a compact support,  we have
\[
\langle u,\, \psi \rangle = \langle \iota_{\mathscr{D}b(U)}(u),\, \psi \rangle
\]
for any $\psi \in 
(\mathscr{A}^{(n)} \otimes or_M)(U)$, where the left-side is the inner product of
a distribution with a compact support and the right-side is that of a hyperfunction.
\end{lemma}
%%%%%%%%%%%%%%%%%%%%%%%%%%%%%%%%%%%%%%%%%%%%%%%%%%%%%%%%%%%%%%%%%%%%%%%%%%%%%%%%%%add a proof
\begin{proof}
Let $h$ be a real analytic function on $U$.
We also denote by $h$ its complexification.
We set $\psi=h(t)dt\otimes e_M$ on $U$.
Then for the delta current $\delta =[(\tau_1,\tau_{01})]$,
the lemma follows from 
%that
\begin{align*}
\langle\ u,\  h(t)dt\otimes e_M\ \rangle
&= \langle\ u,\, \left( \int_{R_1} h(z)(\tau_1\otimes a_X) + \int_{R_{01}} h(z)(\tau_{01}\otimes a_X)\right)\, dt\otimes e_M \ \rangle\\
&= \int_{R_1} (u* (\tau_1\otimes a_X))\wedge h(z)dz\otimes e_M \\
&\hspace{4cm}+ \int_{R_{01}} (u* (\tau_{01}\otimes a_X))\wedge h(z)dz\otimes e_M\\
&= \int_X \iota_{\mathscr{D}b(U)}(u)\smallsmile (h(z)dz\otimes e_M)%\\
%&
= \langle\ \iota_{\mathscr{D}b(U)}(u),\  h(t)dt\otimes e_M\ \rangle.
\end{align*}
\end{proof}

%%%%%%%%%%%%%%%%%%%%%%%%%%%%%%%%%%%%%%%%%%%%%%%%%%%%%%%%%%%%%%%%%%%%%%%%%%%%%%%%%%add a proof
\begin{corollary}
$\iota_{\mathscr{D}_b(U)}$ is independent of the choices of coordinates
in {\eqref{eq:def-integral-of-distribution}}.
\end{corollary}
\begin{proof}
If $u$ has a compact support, 
the claim is a consequence of 
the duality theorem for hyperfunctions with  compact support
and the above lemma.
The general case 
%easily 
follows from the local property of $\iota_{\mathscr{D}b(U)}$
by considering $\theta u$ instead of $u$ with $\theta \in \vG_c(U;\mathscr{E}_M)$.
\end{proof}
\begin{proposition}
	$\iota_{\mathscr{D}b(U)}$ is injective.
\end{proposition}
\begin{proof}
From the beginning, we may assume $U$ to be relatively compact.
Through the proof, we fix the orientation and forget $e_M$, $a_X$, etc.
Let $u \in \mathscr{D}b(U)$ with $\iota_{\mathscr{D}b(U)}(u) = 0$.
Since $\mathscr{D}b$ is a sheaf,
%For any $W$ being a relatively compact open set in $U$,
 it suffices to show $u|_W = 0$ for any   relatively compact open  set  $W$ in $U$.
This is equivalent to saying that 
\[
\langle u,\varphi\rangle = 0 \qquad\text{for}\
\varphi \in \vG_c(U; \mathscr{E}^{(n)}_M)\
\text{with $\operatorname{Supp}(\varphi) \subset W$}.
\]
Let $\varphi \in \vG_c(U; \mathscr{E}^{(n)}_M)$
with $\operatorname{Supp}(\varphi) \subset W$, and
let $W'$ and $W''$ be relatively compact open sets in $U$ with
$\overline{W} \subset W' \subset \overline{W'} \subset W''$.
Let $\theta$ a be $C^\infty$ function on $U$ 
satisfying $\theta = 1$ on a neighborhood of $\overline{W'}$
and $\operatorname{Supp}(\theta) \subset W''$. Then, by the local
property of $\iota_{\mathscr{D}b(U)}$, we have
\[
\iota_{\mathscr{D}b(U)}(\theta u)|_{W'} = 0.
\]
Set, for $\epsilon > 0$,
\[
\varphi_{\epsilon}(z) = \dfrac{1}{(\sqrt{\pi\epsilon})^n}
\int_{\R^n} \varphi(t) e^{-(z-t)^2/\epsilon} \,dt.
\]
It is well-known that
\begin{enumerate}
\item[(1)] $\varphi_\epsilon(z) \in \mathscr{O}^{(n)}(\C^n)$.
\item[(2)]  $\varphi_\epsilon|_{\R^n} \to \varphi$ uniformly convergent on
a compact subset in $\R^n$.
\item[(3)]  $\varphi_\epsilon \to 0$ uniformly convergent on a compact subset in
\[
L=\bigcap_{a \in W} \{\,z = x+\sqrt{y} \in \C^n\mid\, |y| < |x-a|\,\}.
\]
\end{enumerate}
One should be aware that we can say nothing about convergence of $\varphi_\epsilon$
outside of the region specified in the above (2)~and (3).

Since $\operatorname{Supp}(\iota_{\mathscr{D}b(U)}(\theta u)) \subset W'' \ssm W'$,
by the same reasoning as that for Lemma {\ref{lem:support_dirac}}, 
we can find a representative $\tau(z) = (\tau_1(z), \tau_{01}(z))$ of 
$\iota_{\mathscr{D}b(U)}(\theta u)$ which satisfies the condition
\[
\operatorname{Supp}(\tau_1) \subset  L,
\qquad \operatorname{Supp}(\tau_{01}) \subset L.
\]
Then we can take an open subset $R_1$ (with $C^{\infty}$ boundary $R_{01}$) that
contains the support of $\tau$ and that is contained in $L$, for which 
we get
\[
\langle \theta u, \varphi_\epsilon|_{U} \rangle
=
\langle \iota_{\mathscr{D}b(U)}(\theta u), \varphi_\epsilon \rangle
=
\int_{R_1} \varphi_\epsilon \tau_1 
+ \int_{R_{01}} \varphi_\epsilon \tau_{01}
=
\int_{R_1} \varphi_\epsilon \tau_1. 
\]
When $\epsilon \to 0$,
the rightmost term of the above equation 
tends to $0$ by the property (3)~of $\varphi_\epsilon$.
Furthermore, it follows from the property  (2)~that we have, when $\epsilon \to 0$,
\[
\langle \theta u, \varphi_\epsilon|_{U} \rangle \to 
\langle \theta u, \varphi \rangle =
\langle u, \varphi \rangle.
\]
Therefore we finally obtained $\langle u, \varphi \rangle = 0$, which completes
the proof.
\end{proof}
%%%%%

\appendix

\lsection{Compatibility of boundary value morphisms}{\label{appen:A}}
\newcommand{\rh}{\boldsymbol{R}\mathscr{H}\hspace{-.6mm}om}
\newcommand{\Rh}{\boldsymbol{R}\mathrm{Hom}}

The boundary value morphism was first constructed in \cite{Skk}, and
then, it was extended to more general cases by P.~Schapira from the viewpoint of boundary value problems (see Section 11.5  in \cite{KS}). 
In this appendix, we briefly recall its functorial
construction due to P.~Schapira, and then, we will show that the boundary value morphism 
given in Subsection \ref{ssbv} coincides with the current one.

Let $V$ be an open neighborhood of $M$ and $\Omega$ an open subset in $X$ which 
satisfies the conditions $(\rm{B}_1)$~and $(\rm{B}_2)$~given in Subsection \ref{ssbv}. For simplicity,
we also assume the following additional condition in functorial construction\,:
%First recall that, for an open subset $j: U \hra X$, 
%the sheaf $\C_U$ is defined by $j_!j^{-1}\C_X$, and 
%$\C_K = i_*i^{-1}\C_X$ for a closed subset $i: K \hra X$.
\begin{enumerate}
	\item[$(\rm{B}_3)$] $\Omega$ is cohomologically trivial, that is, $\Omega$ satisfies
\[
\rh_{\C_X}(\C_\Omega,\C_X) = \C_{\overline{\Omega}}
\quad\text{and}\quad
\rh_{\C_X}(\C_{\overline{\Omega}},\C_X) = \C_{\Omega}.
\]
\end{enumerate}
Note that the above condition is satisfied if the inclusion $j: \Omega \to X$
is locally homeomorphic to the inclusion of an open convex subset into $\C^n$.

\

From the assumption $\bar{\Omega} \supset M$,
we have a canonical morphism 
%of sheaves
%\[
	$\C_{\overline{\Omega}} \overset{i}{\lra} \C_M$.
%\]
%Then, by a
Applying the functor $\rh_{\C_X}(\bullet,\C_X)$ to this
morphism, 
thanks to the condition $(\rm{B}_3)$~and the fact
\[
\rh_{\C_X}(\C_M,\C_X)
\simeq \C_M \otimes_{{}_{\Z_M}} H^n_M(X;\Z_X)[-n],
\]
we have the morphism in ${\bf{D}}^b({\Z}_X)$
\begin{equationth}{\label{eq:boundary_base}}
\C_{\Omega} 
\overset{i^\star}{\lra}
\C_M \otimes_{{}_{\Z_M}} 
H^n_M(X;\Z_X)[-n],\qquad i^\star=\rh_{\C_X}(i,\, \C_X).
\end{equationth}
%where the morphism $i^\star$ is $\rh_{\C_X}(i,\, \C_X)$.

Note that, for any complex $\scF$, we have the formulas
\[
\Rh_{\C_X}(\C_\Omega,\scF)
=\boldsymbol{R}\vG(\Omega;\scF)\quad\text{and}\quad
\Rh_{\C_X}(\C_M,\scF)
=\boldsymbol{R}\vG_M(X;\scF).
\]
Then, applying $\Rh_{\C_X}(\bullet,\scO)$ to {\eqref{eq:boundary_base}}
and taking the $0$-th cohomology groups,
we have obtained the boundary value map $\hat{b}_\Omega$ in a functorial way\,:
\[
	\hat{b}_\Omega: \scO(\Omega) \overset{i^{\star\star}}{\lra}
H^n_M(X;\scO) \
\otimes_{{}_{\Z_M(M)}} H^n_M(X;\Z_X) = \scB(M),\qquad i^{\star\star} = \mathrm{Hom}(i^\star,\scO).
\]
%Here $i^{\star\star} = \mathrm{Hom}(i^\star,\scO)$.

Now we give the theorem which guarantees the coincidence of the boundary value morphism
in our framework and the functorial one constructed above.
First recall the condition $(\rm{B}_2)$ in Subsection~\ref{ssbv} 
and its local version 
$(\rm{B}_2^{\rm{loc}})$ in Subsection~\ref{ssec:restrict}.
We also introduce the condition $(\rm{B}_1^{\dag})$~which
is stronger version of $(\rm{B}_1)$ in Subsection~\ref{ssbv}.
\begin{enumerate}
%\item[i'] 
\item[$(\rm{B}_1^{\dag})$] For any $x_0 \in M$, there exist an open neighborhood  $W$ of $x_0$
	with a ($C^1$-class) local trivialization 
	$\iota: (M \cap W,\,W) \simeq 
	(\R_x^n,\, \C^n = \R^n_x \times \sqrt{-1}\R^n_y)$ and a non-empty open cone $\Gamma \subset \R_y^n$ such that
\[
	\R_x^n \times \sqrt{-1} \Gamma \subset \iota(W \cap \Omega).
\]
\end{enumerate}
\begin{theorem}
Assume the pair $(V,\Omega)$ satisfies the conditions $(\rm{B}_1^{\dag})$, 
$(\rm{B}_2)$, $(\rm{B}_2^{\rm{loc}})$~and $(\rm{B}_3)$.
Then the morphisms $b_\Omega$ and $\hat{b}_\Omega$ coincide.
\end{theorem}
Since $\{\scB(U)\}_{U \subset M}$ forms a sheaf and since
the boundary value morphisms and restriction maps of sheaves commute by 
Corollary {\ref{cor:restriction_boundary_map}}, 
by taking the condition $(\rm{B}_2^{\rm{loc}})$~into account, 
it suffices to show the claim locally.  
Hence we may assume that 
$M$ is a convex open subset in $\R^n$,
$X = M \times \sqrt{-1}\R^n$ and $V$ is the open tubular neighborhood
\[
V = \{\,(z = x + \sqrt{-1}y) \in \C^n\mid\, x \in M,\, |y| < \epsilon \,\}
\]
with some $\epsilon > 0$. Further, $\Omega$ is assumed to be $
(M \times \sqrt{-1} \Gamma) \cap V$ where $\Gamma$ is an open proper convex cone in $\R_y^n$.
Then, we take $(n+1)$-vectors $\eta_1$, $\dots$, $\eta_{n+1}$ in $\R^n_y$
as was specified in Example \ref{exa:support_cone_fundamental} and define
the open half spaces $H_k$, $k=1,\dots, n+1$, in $\R^n_y$ and 
$C^{\infty}$-functions $\varphi_k$, $k=1,\dots,n+1$, as in the same example. Since the boundary value morphisms and restriction maps 
of sheaves commute, we may assume 
\[
\Gamma = \underset{1 \le k \le n}{\bigcap} H_{k} 
\]
from the beginning. Now let us define another pair of coverings $(\mathcal{S},\mathcal{S}')$
of $(V,\, V \ssm M)$ by
\[
\mathcal{S} = \{S_1,\dots, S_{n+1}, S_{n+2}\}, \quad
\mathcal{S}' = \{S_1,\dots, S_{n+1}\},
\]
where $S_{n+2} = V$ and $S_k = (M \times \sqrt{-1} H_k) \cap V$,  $k=1,\dots,n+1$.
Note that $(\mathcal{S},\mathcal{S}')$ is a Leray covering with respect to
either of 
%the sheaves 
$\C_X$ and $\scO$. Note also that $(\mathcal{S},\mathcal{S}')$
is a pair of  coverings finer than $(\W,\W')$. We denote by $C^{\bullet}(\mathcal{S},\mathcal{S}';\scS)$ the complex of relative \v{C}ech cochains on $(\mathcal{S},\mathcal{S}')$ with coefficients in a sheaf $\scS$.

Let $\nu = (0,\,\nu_{01})$ be the element of $\scE^{(n)}(\W,\W')$
defined in Example {\ref{exa:support_cone_fundamental}}, and
set 
$\Lambda = \{1,2,\dots,n+1,n+2\}$.
We also set, for $\alpha = (\alpha_1,\dots,\alpha_k) \in \Lambda^k$,
\[
S_\alpha = S_{\alpha_1} \cap \dots \cap S_{\alpha_k}.
\]

\begin{lemma}{\label{lem:tau_is_generator}}
$\nu$ is a generator of $\hDCcpk{n}{\W}{\W'} \simeq H_M^n(X;\C_X)$.
\end{lemma}
\begin{proof}
Let us consider the diagram of complexes:
\[
\ccp{\mathcal{S}}{\mathcal{S'}}{\C_X}
\overset{h}{\lra}
\scE^{(\bullet)}(\cS,\cS')
\overset{r}{\longleftarrow}
\scE^{(\bullet)}(\W,\W'),
\]
where $h$ and $r$ are canonical morphisms of complexes.
Since these complexes are quasi-isomorphic to $\boldsymbol{R}\vG_M(X;\C_X)$,
the morphisms $h$ are $r$ are quasi-isomorphic.

Let us define $\sigma = \{\sigma_{\alpha}\}_{\alpha \in \Lambda^{n+1}} \in
\ccpk{n}{\mathcal{S}}{\mathcal{S'}}{\Z_X}$ by
\begin{equationth}{\label{eq:orientation_unit}}
\sigma_\alpha =
\left\{
	\begin{array}{ll}
(-1)^n\qquad &\text{if $\alpha = (1,2,\dots,n,\,n+2)$}, \\
0\qquad &\text{if $\alpha$ contains $n+1$}.
\end{array}
\right.
\end{equationth}
Then $[\sigma]$ belongs to $\hccpk{n}{\mathcal{S}}{\mathcal{S'}}{\Z_X}$
and it is a generator of 
$\hccpk{n}{\mathcal{S}}{\mathcal{S'}}{\Z_X}$. 
Here we take $(-1)^n$ so that
$[\sigma]$ becomes the positively oriented generator
under the standard orientation on $\R_y^n$.
Hence it suffices to show that $h(\sigma)$ and $r(\nu)$ are the same in
$\hDCcpk{n}{\mathcal{S}}{\mathcal{S}'}$, and it can be shown by repeated applications
of the following easy lemma.
\end{proof}
\begin{lemma}{\label{lemma:fundamental_equiv}}
Assume $q_1 > 0$.
Let $\omega = \{\omega_\alpha\}_{\alpha \in \Lambda^{q_1+1}}$ be in
$\DCcdkk{q_1+1}{q_2}{\mathcal{S}}{\mathcal{S}'}$ with $\delta(\omega) = 0$,
and define 
$\omega' = \{\omega'_\beta\}_{\beta \in \Lambda^{q_1}} \in
\DCcdkk{q_1}{q_2}{\mathcal{S}}{\mathcal{S}'}$ by
%\begin{equationth}
\[
\omega'_\beta = \sum_{\lambda \in \Lambda} \varphi_\lambda \omega_{\lambda\,\beta}
\qquad (\beta \in \Lambda^{q_1}).
\]
%\end{equationth}
Then we have $\omega = \delta(\omega')$.
\end{lemma}
\begin{remark}
The same result holds for $\dccdp{p}{\mathcal{S}}{\mathcal{S}'}$. That is:
Let $q_1 > 0$ and let $\omega = \{\omega_\alpha\}_{\alpha \in \Lambda^{q_1+1}}$ be in
$\dccdpkk{p}{q_1+1}{q_2}{\mathcal{S}}{\mathcal{S}'}$ with $\delta(\omega) = 0$.
Define 
$\omega' = \{\omega'_\beta\}_{\beta \in \Lambda^{q_1}} \in
\dccdpkk{p}{q_1}{q_2}{\mathcal{S}}{\mathcal{S}'}$ by
the same formula as in the above lemma.
Then we have $\omega = \delta(\omega')$.
\end{remark}

Let us compute functors $\rh_{\C_X}(\bullet,\C_X)$ 
and $\Rh_{\C_X}(\bullet,\scO)$ concretely.
Set $\Lambda'=\{1,2,\dots,n+1\}$.
First define the sheaf 
$\underset{{\alpha \in \Lambda'^k}}{\wedge}\C_{\overline{S_\alpha}}$:
It is a subsheaf of
$\underset{{\alpha \in \Lambda'^k}}{\oplus}\C_{\overline{S_\alpha \cap S_{n+2}}}$
which consists of alternating sections with respect to the index $\alpha \in \Lambda'^k$.
We 
%also 
define the differential 
\[
\delta_k:
\underset{{\alpha \in \Lambda'^k}}{\wedge}\C_{\overline{S_\alpha}}
\lra
\underset{{\alpha \in \Lambda'^{k+1}}}{\wedge}\C_{\overline{S_\alpha}}
\]
the same way as for a \v{C}ech complex. We have the complex
\[
\mathcal{L}:
0
\lra 
\underset{{\alpha \in \Lambda'^0}}{\wedge}\C_{\overline{S_\alpha}}
\lra
\underset{{\alpha \in \Lambda'^1}}{\wedge}\C_{\overline{S_\alpha}}
\lra
\cdots
\overset{\delta_{n-1}}{\lra}
\underset{{\alpha \in \Lambda'^n}}{\wedge}\C_{\overline{S_\alpha}}
\lra
0,
\]
where 
%the term
$
\underset{{\alpha \in \Lambda'^0}}{\wedge}\C_{\overline{S_\alpha}}
$
is located at degree $-n$ 
%in this complex 
and
$
\underset{{\alpha \in \Lambda'^n}}{\wedge}\C_{\overline{S_\alpha}}
$
at degree $0$.  We 
%also construct 
define the morphism 
\[
\iota:
\underset{{\alpha \in \Lambda'^n}}{\wedge}\C_{\overline{S_\alpha}}
\lra
\C_M \otimes_{{}_{\Z_X(X)}} \hccpk{n}{\mathcal{S}}{\mathcal{S'}}{\Z_X}
\]
by assigning an alternating section $c_\alpha$ on ${\overline{S_\alpha \cap S_0}}$
to $c_\alpha|_M \otimes 1_\alpha$. Here $1_\alpha$ denotes an alternating section
with value $1$ on $S_\alpha$.
Then we can extend $\iota$ to a morphism of complexes
from $\mathcal{L}$ to the single complex
$\C_M \otimes_{{}_{\Z_X(X)}} \hccpk{n}{\mathcal{S}}{\mathcal{S'}}{\Z_X}$.
Now we can easily see\,:
\begin{lemma}
The morphism
\[
\mathcal{L} \overset{\iota}{\lra}
\C_M \otimes_{{}_{\Z_X(X)}} \hccpk{n}{\mathcal{S}}{\mathcal{S'}}{\Z_X}
\]
is quasi-isomorphic. That is,
$\mathcal{L}$ is a resolution complex of
$\C_M \otimes_{{}_{\Z_X(X)}} \hccpk{n}{\mathcal{S}}{\mathcal{S'}}{\Z_X}$.
\end{lemma}
\begin{proof}
For $k=1,\dots,n+1$, we have the exact sequence
\[
0 \lra \C_{\overline{V} \ssm \overline{S_k}}
\lra \C_{\overline{V}} 
\lra \C_{\overline{S_k}} 
\lra 0.
\]
Define the complex $\mathscr{L}_k$ by
\[
0 \lra \C_{\overline{V}} \lra \C_{\overline{S_k}} \lra 0,
\]
where 
%the term 
$\C_{\overline{S_k}}$ is located in degree $0$ of this complex.
Then the above exact sequence implies that
$\mathscr{L}_k$ is quasi-isomorphic to $\C_{\overline{V} \ssm \overline{S_k}}\,[1]$.
Hence the single complex $\widetilde{\mathscr{L}}$ associated with
\[
\mathscr{L}_1 \otimes_{ {}_{\C_X}}
\mathscr{L}_2 \otimes_{ {}_{\C_X}} \cdots
\otimes_{ {}_{\C_X}} \mathscr{L}_{n+1}
\]
is quasi-isomorphic to the complex
\[
(\C_{\overline{V}\ssm \overline{S_1}} \otimes_{ {}_{\C_X}}
\C_{\overline{V}\ssm \overline{S_2}} \otimes_{ {}_{\C_X}}
\cdots
\otimes_{ {}_{\C_X}} \C_{\overline{V}\ssm \overline{S_{n+1}}})[n+1]
\]
which is isomorphic to $0$ because of
%\[
$(\overline{V}\ssm \overline{S_1})
\cap
(\overline{V}\ssm \overline{S_2})
\cap
\cdots
\cap
(\overline{V}\ssm \overline{S_{n+1}})
= \emptyset$.
%\]
Therefore the complex $\widetilde{\mathscr{L}}$ becomes exact,
and the result follows from the fact 
that the degree $0$ term of the complex $\widetilde{\mathscr{L}}$ is
\[
\C_{\overline{S_{1}}}
\otimes_{ {}_{\C_X}} 
\C_{\overline{S_{2}}}
\otimes_{ {}_{\C_X}} 
\cdots
\otimes_{ {}_{\C_X}} \C_{\overline{S_{n+1}}} = \C_M
\]
due to 
$\overline{S_{1}} \cap
\overline{S_{2}} \cap
\cdots
\cap
\overline{S_{n+1}} = M
$.
\end{proof} 

Furthermore, since we have $S_{12\cdots n} \cap S_{n+2} = (M \times \sqrt{-1}\Gamma) \cap V = \Omega$, we have
the following commutative diagram\,:
\[
\SelectTips{cm}{}
\xymatrix
@C=.6cm
@R=.5cm
{\underset{\alpha = {12\dots n}}{\wedge}\C_{\overline{S_\alpha}}\ar[r]^-{\iota} \ar[d]
	&\C_{\overline{\Omega}} \otimes_{{}_{\Z_X(X)}} \hccpk{n}{\mathcal{S}}{\mathcal{S'}}{\Z_X}\ar[d]
\\
 \underset{{\alpha \in \Lambda'^n}}{\wedge}\C_{\overline{S_\alpha}}\ar[r]^-{\iota}  & 
 \C_{M} \otimes_{{}_{\Z_X(X)}} \hccpk{n}{\mathcal{S}}{\mathcal{S'}}{\Z_X},}
\]
where $\underset{\alpha = {12\dots n}}{\wedge}\C_{\overline{S_\alpha}}$
is the subsheaf of $\underset{{\alpha \in \Lambda'^n}}{\wedge}\C_{\overline{S_\alpha}}$ which consists of alternating sections only on $S_{12\dots n} \cap S_{n+2}$.

Since $(\mathcal{S},\mathcal{S}')$ is an acyclic covering with respect to $\scO$
and since each open subset $S_\alpha$ $(\alpha \in \Lambda^k)$ is cohomologically trivial,
we can compute $\Rh_{\C_X}(\C_M,\,\scO)$ by
first applying
$\rh_{\C_X}(\bullet,\, \C_X)$ to the above resolution $\mathcal{L}$, and then,
applying $\Rh_{\C_X}(\bullet,\,\scO)$ to the resulting complex.
As a conclusion, we have
\[
\Rh_{\C_X}(\C_M,\, \scO) \simeq
\ccp{\mathcal{S}}{\mathcal{S'}}{\scO}.
\]
Further, by applying the functor
$\rh_{\C_X}(\bullet,\, \C_X)$ to
the above commutative diagram, and then,
applying $\Rh_{\C_X}(\bullet,\,\scO)$ to the resulting diagram,
we see that the morphism
\[
%\hat{b}_\Omega:
\hat{b}_\Omega: \scO(\Omega) \lra 
H^n_M(X;\,\scO) 
\otimes_{{}_{\Z_M(M)}} H^n_M(X;\,\mathbb{Z}_X) 
\simeq
\hccpk{n}{\mathcal{S}}{\mathcal{S'}}{\scO}
\otimes_{{}_{\Z_X(X)}}
\hccpk{n}{\mathcal{S}}{\mathcal{S'}}{\mathbb{Z}_X}
\]
is given by
\[
\hat{b}_\Omega(f) = [f \sigma] \otimes [\sigma]
\in \hccpk{n}{\mathcal{S}}{\mathcal{S'}}{\scO}
{\otimes}_{ {}_{\Z_X(X)}}
\hccpk{n}{\mathcal{S}}{\mathcal{S'}}{\mathbb{Z}_X},
\]
where $\sigma$ is defined in {\eqref{eq:orientation_unit}}.

Now we consider the diagram
\[
\ccp{\mathcal{S}}{\mathcal{S'}}{\scO}
\overset{h'}{\lra}
\dccp{\mathcal{S}}{\mathcal{S}'}
\overset{r'}{\longleftarrow}
\dccp{\W}{\W'},
\]
where $h'$ and $r'$ are canonical morphisms of complexes and
they are quasi-isomorphic.
Then, by the remark after Lemma \ref{lemma:fundamental_equiv},
we see that $h'(f\sigma)$ and $r'(f \rho(\nu))$ are the same element
in $\hdccpk{n}{\mathcal{S}}{\mathcal{S}'}$.
This implies $\hat{b}_{\Omega}(f) = b_{\Omega}(f)$ and the theorem follows.

Clearly $\hat{b}_\Omega$ is a morphism of $\mathscr{D}$-modules. Therefore,
by the theorem, we have\,:
% the following corollary:
\begin{corollary}
The boundary value morphism $b_\Omega$ is a morphism of $\mathscr{D}$-modules.
\end{corollary}

%%%%%%
\bibliographystyle{plain}

\

N. Honda

Department of Mathematics 

Faculty of Science

Hokkaido University

Sapporo 060-0810, Japan

honda@math.sci.hokudai.ac.jp

\

T. Izawa

Department of Information and Computer Science 

Hokkaido University of Science

Sapporo 006-8585, Japan

t-izawa@hus.ac.jp

\

T. Suwa

Department of Mathematics 

Faculty of Science

Hokkaido University

Sapporo 060-0810, Japan

suwa@math.sci.hokudai.ac.jp

\end{document}